\documentclass[11pt]{amsart}
\usepackage{amssymb, amsmath,latexsym,amsfonts,amsbsy, amsthm,mathtools,graphicx,color,mathrsfs}
\usepackage{float,subfigure}
\usepackage{hyperref}

\numberwithin{equation}{section}
\allowdisplaybreaks

\usepackage{enumerate}
\let\al=\alpha

\let\g=\gamma

\let\e=\varepsilon

\let\la=\lambda

\let\om=\omega
\let\G= \Gamma

\let\wt=\widetilde

\let\th=\theta
\let\pa=\partial
\let\var=\varphi

\def\cA{{\mathcal A}}
\def\cB{{\mathcal B}}

\def\cD{{\mathcal D}}
\def\cE{{\mathcal E}}
\def\cF{{\mathcal F}}
\def\cG{{\mathcal G}}
\def\cH{{\mathcal H}}
\def\cI{{\mathcal I}}
\def\cJ{{\mathcal J}}

\def\cL{{\mathcal L}}
\def\cM{{\mathcal M}}
\def\cN{{\mathcal N}}
\def\cO{{\mathcal O}}
\def\cP{{\mathcal P}}
\def\cQ{{\mathcal Q}}

\def\cS{{\mathcal S}}
\def\cT{{\mathcal T}}

\def\C{\mathbb C}
\def\R{\mathbb R}

\def\Z{\mathbb Z}
\def\N{\mathbb N}
\def\e{\mathrm e}
\def\ii{\text{i}}

\def\ds{\displaystyle}

\def\pv{\operatorname{p.v.}}

\newcommand{\norm}[1]{\left\|#1\right\|}

\newcommand{\beq}{\begin{equation}}
	\newcommand{\eeq}{\end{equation}}
\newcommand{\ben}{\begin{eqnarray}}
	\newcommand{\een}{\end{eqnarray}}
\newcommand{\beno}{\begin{eqnarray*}}
	\newcommand{\eeno}{\end{eqnarray*}}

\newtheorem{theorem}{Theorem}[section]

\newtheorem{lemma}[theorem]{Lemma}
\newtheorem{proposition}[theorem]{Proposition}
\newtheorem{corollary}[theorem]{Corollary}
\theoremstyle{remark}
\newtheorem{remark}[theorem]{Remark}

\begin{document}

\title[Self-similar algebraic spiral vortex sheets]
{Self-similar algebraic spiral vortex sheets of 2-D incompressible Euler equations}
		
\author[F. Shao]{Feng Shao}
\address{School of Mathematical Sciences, Peking University, Beijing 100871,  China}
\email{fshao@stu.pku.edu.cn}
		
\author[D. Wei]{Dongyi Wei}
\address{School of Mathematical Sciences, Peking University, Beijing 100871,  China}
\email{jnwdyi@pku.edu.cn}
	
\author[Z. Zhang]{Zhifei Zhang}
\address{School of Mathematical Sciences, Peking University, Beijing 100871, China}
\email{zfzhang@math.pku.edu.cn}
		
\date{\today}

\begin{abstract}
This paper provides the first rigorous construction of the self-similar algebraic spiral vortex sheet solutions to the 2-D incompressible Euler equations. These solutions are believed to represent the typical roll-up pattern of vortex sheets after the formation of curvature singularities. The most challenging part of this paper is to handle the Cauchy integral for the algebraic spiral curve, which falls outside the classical theory of singular integral operators.  
\end{abstract}

\maketitle


\section{Introduction}

In this paper, we consider  the 2-D incompressible Euler equations on $(t,x)\in (0, +\infty)\times\R^2$:
\begin{equation}\label{Eq.Euler-velocity}
	\pa_t u+u\cdot\nabla u+\nabla P=0,\quad \operatorname{div} u=0,\quad u|_{t=0}=u_0,
\end{equation}
where $u=(u^1, u^2)$ is the velocity of an inviscid fluid and $P$ is the pressure function. The vorticity-stream formulation of \eqref{Eq.Euler-velocity}  takes
\begin{equation}\label{Eq.Euler-vorticity}
	\pa_t\omega+u\cdot\nabla\omega=0,\quad\Delta \psi=\omega,\quad u=\nabla^\perp\psi=K_2*\omega,
\end{equation}
where $\omega=\pa_1u^2-\pa_2u^1$ is the vorticity, $\psi$ is the stream function, $\nabla^\perp=(-\pa_2,\pa_1)$ and
\begin{equation}\label{Eq.kernel-K2}
	K_2(x)=\frac1{2\pi}\frac{x^\perp}{|x|^2},\quad x^\perp:=(x_1,x_2)^\perp=(-x_2,x_1).
\end{equation}

For smooth initial data, the 2-D Euler equations are globally well-posed due to  the conservation of the $L^\infty$ norm of the vorticity, $\|\omega(t)\|_{L^\infty}\le \|\omega_0\|_{L^\infty}$. However, the long-time behavior of smooth solutions remains a long-standing problem \cite{Sve}; see \cite{KS, BM, IJ, Bou, WZZ} for recent progress.  For non-smooth data, the pioneering work by Yudovich \cite{Yudovich} proved the global existence and uniqueness of weak solutions for the initial vorticity $\omega_0\in L^1\cap L^\infty(\R^2)$, which is related to vortex patch solutions. The global existence of weak solutions for $\om_0\in L^1\cap L^p(\R^2)$, for $1<p\leq+\infty$, was later proved by Diperna and Majda \cite{DM1987} (see also \cite[Chapter 10]{MB}). Since then, a longstanding open problem in the field has been whether the uniqueness of weak solutions holds when $p<+\infty$. Recently, in the works \cite{Vishik1,Vishik2}, Vishik proved the non-uniqueness of weak solutions by introducing an external force $f\in L_t^1(L_x^1\cap L_x^p)$; see \cite{ABCDGJK} for a review of Vishik's proof. We also refer to \cite{EJ2020CPAM, EJ2023, EJ2020} for recent works on singular vortex patch solutions, and  \cite{Che, MB, MP} for more classical results on the Euler equations.

In this paper, we focus on vortex sheet solutions of \eqref{Eq.Euler-velocity}, where the velocities are discontinuous across the sheet, resulting in a concentrated distribution of vorticity along the sheet. This vorticity is typically represented mathematically as a delta function along a 1-D curve. In the pioneering work \cite{Delort1991}, Delort proved the global existence of weak solutions to \eqref{Eq.Euler-velocity} for vorticity being a Radon measure with distinguished sign. The quantitative behavior of Delort's solutions remains an important question.

The more natural point of view in the mathematical description of a vortex sheet is to explicitly propagate the sheet itself using a time-dependent parametrization. Assume that the vorticity is supported on a curve $Z(t,\Gamma)$ parameterized by the total circulation $\Gamma$. It is well-known that $Z$ solves the following Birkhoff-Rott equation:
\begin{equation}\label{Eq.BR-normal}
	\pa_tZ^*(t,\Gamma)=\frac1{2\pi \ii}\operatorname{p.v.}\int\frac{\mathrm d\Gamma'}{Z(t,\Gamma)-Z(t,\Gamma')},
\end{equation}
where we have switched to complex variable notation for the position of the sheet, $Z^*$ denotes the complex conjugate and the integral is to be understood in the principal value sense
at the singularity $\Gamma'=\Gamma$. This equation was originally derived by Birkhoff \cite{Birkhoff} and it is implicit in the work of Rott \cite{Rott}. See also \cite[Chapter 9]{MB}. The well-known Kelvin-Helmholtz instability \cite{Helmholtz,Kelvin} suggests ill-posedness of \eqref{Eq.BR-normal} in Sobolev spaces. The local well-posedness of \eqref{Eq.BR-normal} for analytic initial data was established in \cite{CO1986,SSBF1981}, based on the abstract Cauchy-Kovalevskaya theorem. An asymptotic analysis by Moore \cite{Moore1979} suggests that for an analytic vortex sheet, a curvature singularity forms at a finite critical time. In \cite{CO1989}, Caflisch and Orellana constructed exact solutions of \eqref{Eq.BR-normal} that are analytic and periodic for $t<0$ and have a curvature singularity at $t=0$.

\begin{figure}[htbp]
	\centering
	\includegraphics[width=0.7\textwidth]{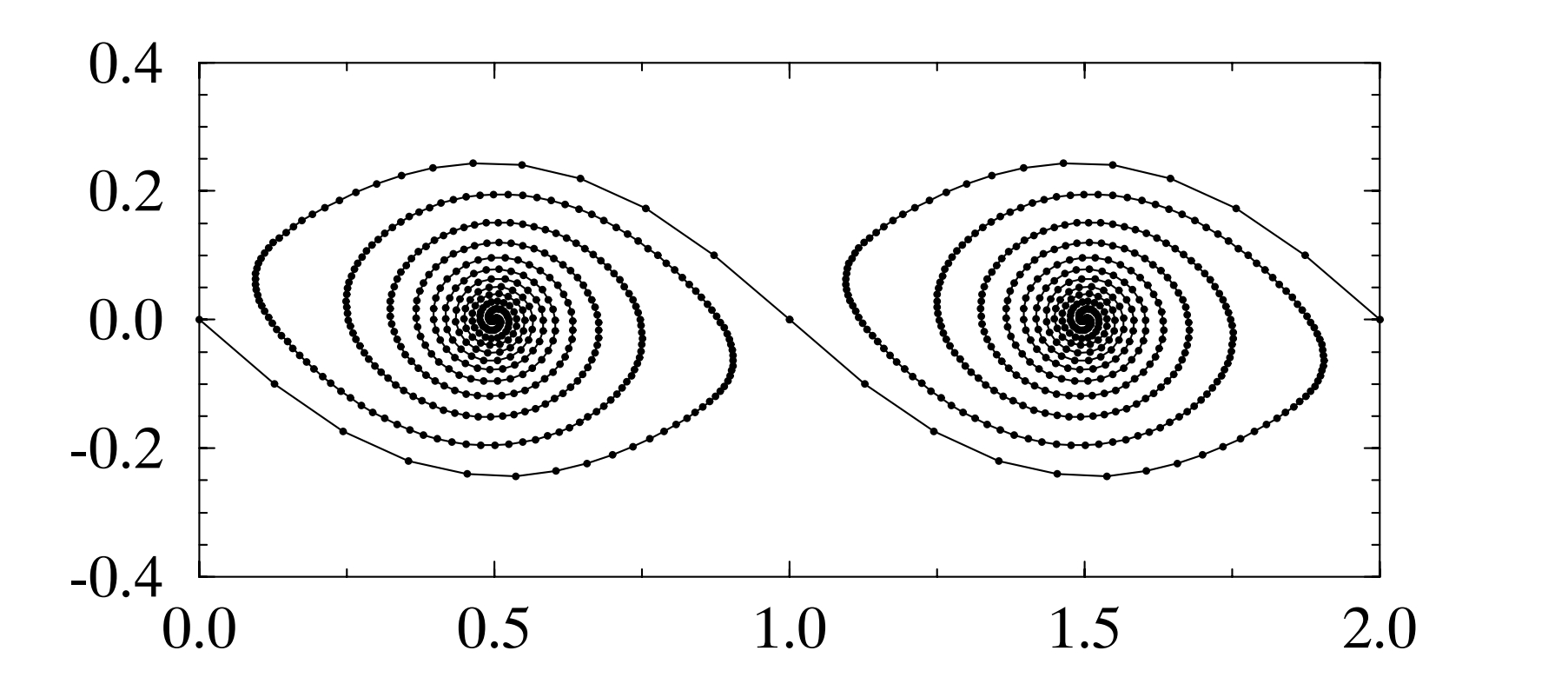}
	\caption{Computation of the periodic vortex sheet after singularity by means of a desingularization of the kernel \eqref{Eq.kernel-K2}. The computation is shown at $t=3.9$, well after the initial singularity. Note the approximation to a double-branched spiral. This picture is \cite[Figure 9.3]{MB}.}
	\label{Fig.Majda}
\end{figure}

A longstanding open problem in the field concerns the behavior of the vortex sheet after singularity formation. Generally, one expects vortex sheet roll-up after the initial singularity \cite{MB,MP,Moore1979,MS1973,Saffman1992}. In \cite{Krasny1986}, Krasny presented compelling numerical evidence that a double-branched spiral vortex sheet forms after singularity formation for periodic perturbations of a straight sheet, as illustrated in Figure \ref{Fig.Majda}, which is taken from \cite{MB} (see also \cite[Figure 3]{Krasny1986}). Another numerical result by Krasny \cite{Krasny1987} showed that the roll-up is in good agreement with Kaden’s asymptotic spiral \cite{Kaden}, an infinite-length algebraic spiral. Based on Pullin’s numerical results on self-similar algebraic spiral vortex sheets \cite{Pullin1978,Pullin1989} and the method of asymptotic analysis, Pullin and Shen \cite{PS2023} constructed a vortex sheet solution after its singularity time and observed the tearing-up behavior.
Moreover, the recent work of Wu \cite{Wu2006} indicates that, after the singularity time, the vortex sheet will fail to be in the class of chord-arc curves that do not roll up too fast near the singularity. This naturally leads one to study self-similar algebraic spiral-shaped vortex sheets, for which no rigorous existence was known prior to the present work.
Let us mention recent works concerning the construction of self-similar algebraic spiral vortex patches, as reported in \cite{E1, E2, GG, SWZ}.

\subsection{Main results}
Now we state our main results. 
 
\begin{theorem}\label{mainthm}
    Let $\mu>1/2$, there exists $m_*\in\Z$ such that for all $m\in \Z\cap(m_*,+\infty)$, the following Birkhoff-Rott equation
		\begin{equation}\label{BR}
			\partial_t Z_m^*(t,\G)=\frac1{2m\pi\mathrm{i}}\sum_{k=0}^{m-1}\pv\int_{-\infty}^0\frac{\mathrm d\Gamma'}{Z_m(t,\Gamma)-\xi_m^kZ_m(t,\Gamma')},
		\end{equation}
		where $\xi_m:=\e^{\frac{2\pi\mathrm{i}}{m}}$ is the $m$-th unit root, with the initial data ($Z_m^*$ denoting the complex conjugate of $Z_m$)
		\begin{equation}\label{Eq.initial-mainthm}
			Z_m(0,\Gamma)=\left(-\frac{2\mu-1}{\mu}\Gamma\right)^{\frac{\mu}{2\mu-1}},
		\end{equation}
        possesses a solution $Z_m(t,\Gamma)$ with the following properties.
        \begin{itemize}
            \item The solution $Z_m$ is self-similar: $Z_m(t,\Gamma)=t^\mu z_m(t^{1-2\mu}\Gamma)$ for all $t>0$ and $\Gamma<0$.
            \item If we write $z_m(\g)=r_m(\g)\mathrm e^{\mathrm i\th_m(\g)}$, then $r_m:(-\infty,0)\to\R_+$ and $\th_m:(-\infty,0)\to\R_+$ are both $C^1$, strictly monotone and bijective. Moreover, we have $\th^\mu r_m(\th_m^{-1}(\th))\in L^\infty$, $\th^{\mu+1}(r_m\circ\th_m^{-1})'\in L^\infty$, and
            \begin{equation}\label{Eq.center-behavior}
			\lim_{\theta\to+\infty}\th^\mu r_m(\th_m^{-1}(\th))=\wt\beta_m,\quad \lim_{\theta\to+\infty}\th^{\mu+1}(r_m\circ\th_m^{-1})'(\th)=-\mu\wt\beta_m
		\end{equation}	
        for some $\wt\beta_m>0$ (see \eqref{Eq.tilde-beta_m}).
        \end{itemize}
        See Figure \ref{Fig.solution} for the shape of the sheets at different times.
\end{theorem}

\begin{figure}[htbp]
	\centering
	\subfigure[$t=0$]{\includegraphics[width=0.3\textwidth]{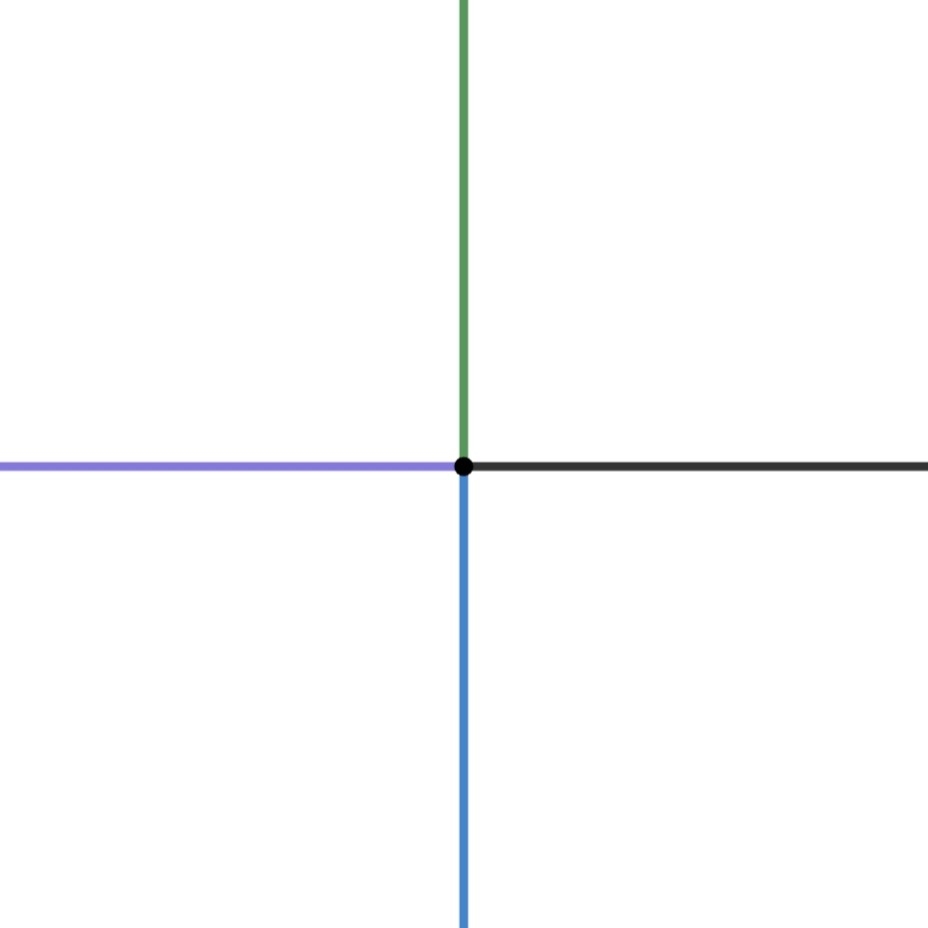}}
	\subfigure[$t=1$]{\includegraphics[width=0.3\textwidth]{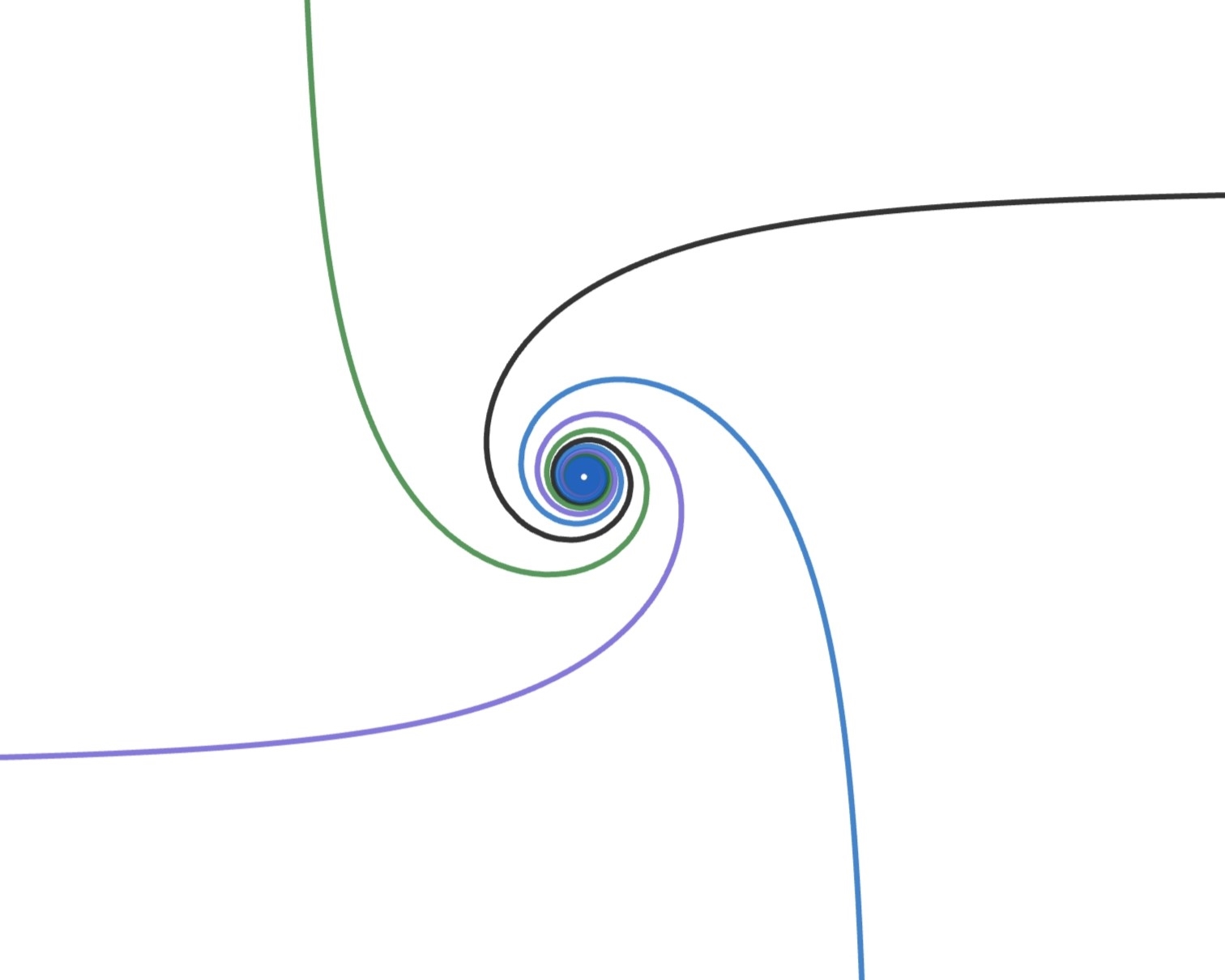}}
	\subfigure[Close-up near the center]{\includegraphics[width=0.3\textwidth]{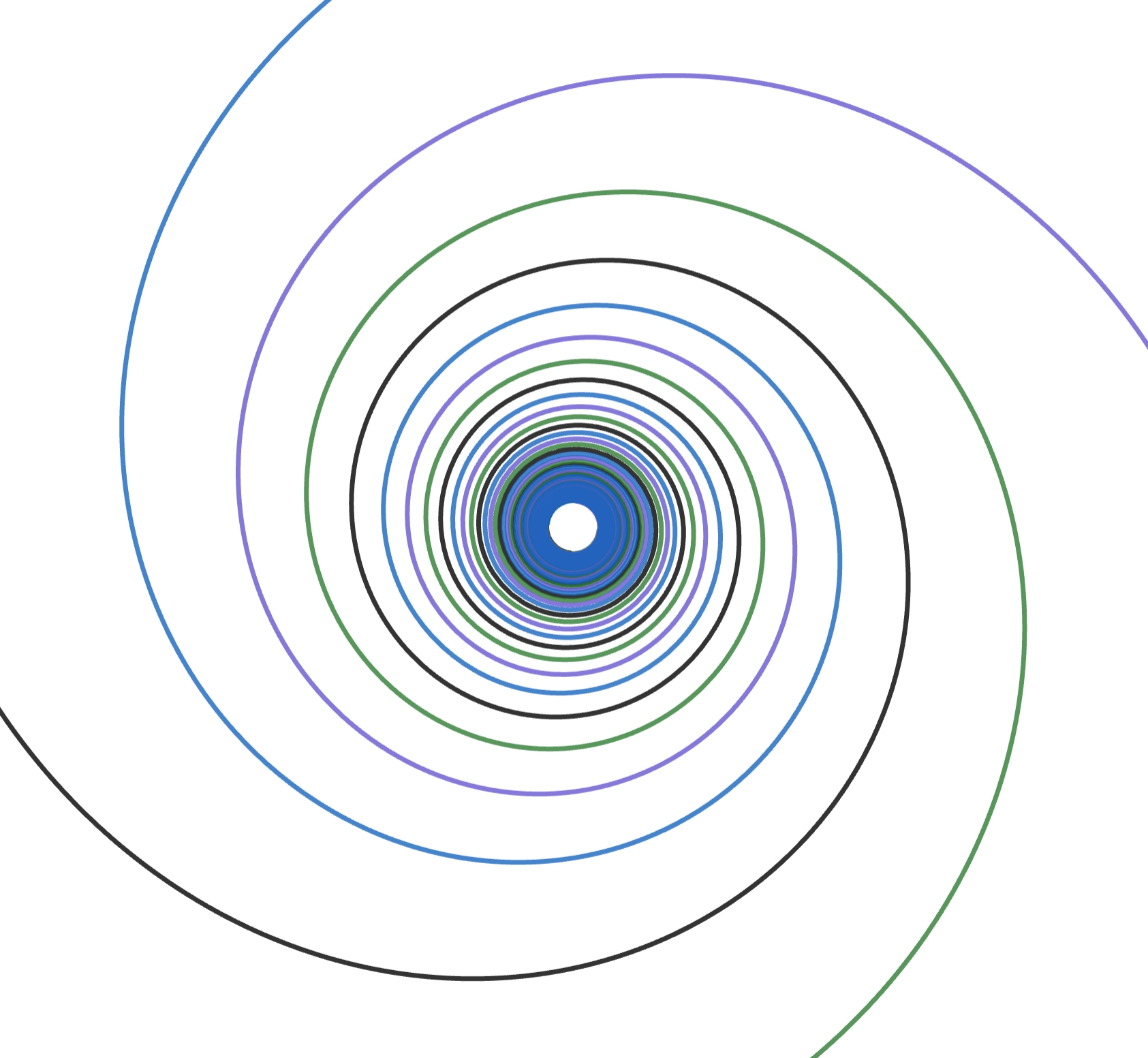}}
	\caption{Symmetric spiral vortex sheets: $m=4$.}
	\label{Fig.solution}
\end{figure}

Several remarks are in order.

\begin{enumerate}[1.]
	\item In view of \eqref{Eq.center-behavior}, our solutions exhibit algebraic spiral behavior near the origin. To the best of our knowledge, this is the first rigorous result in the literature constructing algebraic spiral vortex sheet solutions.

	\item Our solution has the feature that it is invariant under the $\frac{2\pi}m$-rotation, hence it is referred to as {\it $m$-fold}. This kind of solutions form a subclass of the so-called {\it $m$-branched} vortex sheets, which consist of vortex sheets supported on $m$ different curves. Here we require $m$ to be sufficiently large to facilitate a perturbation argument. 
We conjecture that our theorem holds for all $m\in \Z\cap[2,+\infty)$. This will be investigated in future work.

\item Let us mention another type of self-similar spiral vortex sheets: logarithmic spirals, where $r\sim \e^{-a\th}$ near the spiral center ($a>0$, $\theta\to+\infty$). Logarithmic spiral solutions were first discovered by Prandtl \cite{Prandtl1922} (before Kaden \cite{Kaden}), then extended by Alexander \cite{Alexander1971} to families of multi-branched spirals, and were further studied by Kambe \cite{Kambe1989} and Kaneda \cite{Kaneda1989}. Recently, the justification of logarithmic spirals as weak solutions to the 2-D Euler equations was considered by Elling and Gnann \cite{EG2019}, which inspired a lot of subsequent works \cite{Cho2023,CKO2021,CKO2022,CKO2024,JS2024}. Prandtl's logarithmic spirals provided insights into the process of vortex sheet roll-up, but these flows are not simply related to the problems with initial conditions. Moreover, it is known that logarithmic spirals are chord-arc curves, hence Wu's result \cite{Wu2006} excludes the possibility of using logarithmic spirals to describe the behavior of vortex sheets after singularity.

\end{enumerate}

In the following second result, we prove that our constructed vortex sheet solutions are weak solutions to \eqref{Eq.Euler-velocity}.

\begin{theorem}\label{mainthm2}
    Assume that $\mu>1/2$ and $m_*\in\Z$ is given by Theorem \ref{mainthm}. For each $m\in \Z\cap(m_*,+\infty)$, let $Z_m=Z_m(t,\Gamma)$ be the solution to \eqref{BR} with initial data \eqref{Eq.initial-mainthm}. Define $u_m=u_m(t,x)$ by
    \begin{equation}
        u_m(t,x):=\frac1m\sum_{k=0}^{m-1}\int_{-\infty}^0 K_2\left(x-\xi_m^kZ_m(t,\Gamma)\right)\,\mathrm d\Gamma,\quad \forall\ t\geq0,\ x\notin S_m(t)\cup\{0\},
    \end{equation}
    where $S_m(t):=\left\{\xi_m^kZ_m(t,\Gamma): k\in\Z\cap[0,m-1], \Gamma<0\right\}$. Define $\omega_m=\omega_m(t,\cdot)\in \cM(\R^2)$ by
    \begin{equation}\label{Eq.omega-def}
        \langle \omega_m(t,\cdot),\varphi\rangle=\frac1m\sum_{k=0}^{m-1}\int_{-\infty}^{0}\varphi\left(\xi_m^kZ_m(t,\Gamma)\right)\,\mathrm d\Gamma,\quad\forall\ t\geq0,\ \forall\ \varphi\in C_c(\R^2).
    \end{equation}
    Then $u_m$ is a weak solution to \eqref{Eq.Euler-velocity} with the vortex sheet initial data
    \begin{align}
		\omega_m(0,x)&=\frac1m\sum_{k=0}^{m-1}|x|^{1-\frac1\mu}\delta_{\{\xi_m^k\Sigma_0\}}(x),\quad x\in\R^2,\label{Eq.initial-vorticity}\\
		u_m(0,x)=K_2*\omega_m(0,x)&=\frac1m\sum_{k=0}^{m-1}\int_0^\infty K_2\left(x-\xi_m^k(r,0)\right)r^{1-\frac1\mu}\,\mathrm dr,\quad\mathrm{a.e.}\  x\in\R^2,\label{Eq.initial-velocity}
	\end{align}
	where $\delta$ is the 2-D Dirac measure, and $\Sigma_0:=\{x=(x_1,0)\in\R^2|x_1>0\}$ is the semi-infinite initial sheet. Hereafter we also identify $\R^2\simeq\C$. Moreover, 
    \begin{itemize}
        \item $u_m\in C\left([0,+\infty); L^2_{\mathrm{loc}}(\R^2)\right)$ and $\omega_m\in C\left([0,+\infty); H^{-1}_{\mathrm{loc}}(\R^2)\right)\cap C_{\mathrm w}\left([0,+\infty); \cM(\R^2)\right)$.
        \item For all $t\geq0$, $\omega_m(t)=\operatorname{curl}u_m(t)$  is supported on $S_m(t)$.
        \item It is self-similar: $u_m(t,x)=t^{\mu-1}v_m(t^{-\mu}x)$ for $t>0$ and $x\in\R^2$, with $v_m\in L^2_{\operatorname{loc}}(\R^2)$.
    \end{itemize}
\end{theorem}

Several remarks are in order.

\begin{enumerate}[1.]
    \item In \cite{SWZ}, we construct a class of solutions to \eqref{Eq.Euler-velocity} with the following initial vorticity:
    \begin{equation*}
    	\omega(0,x)=r^{-1/\mu}\mathring\omega(\th), \quad x=(r\cos\theta, r\sin\theta),\quad \th\in\mathbb T:=\R/(2\pi\Z),
    \end{equation*}
    where $\mu>1/2$, for $m$-fold ($m\geq 2$) $\mathring \omega\in L^1(\mathbb T)$ satisfying a dominant condition. Thanks to the $L^1$ result, using a standard compactness argument, we \cite{SWZ} proved the existence of solutions if $\mathrm \omega\in \cM(\mathbb T)$ is a Radon measure that is $m$-fold and satisfies the corresponding dominant condition (see \cite[Corollary 1.4]{SWZ}). However, the quantitative behavior of these solutions was not obtained in \cite{SWZ}. Here in the present paper, for $m$ sufficiently large, our initial data \eqref{Eq.initial-vorticity} satisfies all the conditions of \cite[Corollary 1.4]{SWZ}, and our Theorem \ref{mainthm2} gives the quantitative behavior of such solutions.
    
    \item We verify the weak formulation of \eqref{Eq.Euler-velocity} for {\it all} $\mu>1/2$, which is consistent with our previous work \cite{SWZ}. Nevertheless, a fundamental question remains unresolved: whether our algebraic spiral vortex sheet solutions constructed in the present paper coincide with \cite[Corollary 1.4]{SWZ} or not?

    \item The uniqueness of our solution in the class of Delort's weak solutions remains open. Numerical results by Pullin \cite{Pullin1989} suggest possible non-uniqueness, especially for $m=2$.

    \item Due to the scale invariant nature of \eqref{Eq.Euler-velocity}, the same result holds for initial data
    \[\omega_m(0,x)=\frac Am\sum_{k=0}^{m-1}|x|^{1-\frac1\mu}\delta_{\xi_m^k\Sigma_0}(x)\]
    instead of \eqref{Eq.initial-vorticity}, where $A\in\R$, with \eqref{Eq.initial-mainthm} and $\wt\beta_m$ in \eqref{Eq.center-behavior} adapted accordingly. In the proof, for simplicity, we choose appropriate initial data such that $\wt\beta_m=1$ in \eqref{Eq.center-behavior}, which corresponds to $Z_m(0,\G)=d_m(-\G)^{\frac{\mu}{2\mu-1}}$ for some $d_m>0$ such that $\lim_{m\to\infty}d_m=(2\pi)^{-\mu/(2\mu-1)}$. Using the scaling property of \eqref{ss_BR1}, if the initial sheet is given by $Z_m(0,\G)=(-\G)^{\frac{\mu}{2\mu-1}}$, then \eqref{Eq.center-behavior} should be adapted to
    \begin{equation}\label{Eq.beta_m}
			\lim_{\theta\to+\infty}\th^\mu r_m(\th_m^{-1}(\th))=\beta_m,\quad \lim_{\theta\to+\infty}\th^{\mu+1}(r_m\circ\th_m^{-1})'(\th)=-\mu\beta_m
    \end{equation}
    with $\beta_m=d_m^{2\mu-1}$, thus $\lim_{m\to\infty}\beta_m=(2\pi)^{-\mu}$. This relates to the Betz constant \cite{Betz,MS1973}.
\end{enumerate}

\subsection{Difficulties and key ideas}  
Our goal is to find an algebraic spiral solution to the $m$-fold Birkhoff-Rott equation \eqref{BR}. The most serious difficulty is to handle the Cauchy integral for the algebraic spiral curve, which falls outside the classical theory of singular integral operators.  

The first important observation is that the Birkhoff-Rott equation \eqref{BR}, in the self-similar form \eqref{maineq}, can be approximated (formally) by the limiting equation as $m\to+\infty$.  This approximation precisely yields Kaden's algebraic spiral solutions \eqref{special_sol}. This motivates us to rewrite \eqref{maineq} in the form(see \eqref{Eq.perturbation})
$$\mathcal N[r,\g]=\cI_m[r,\g],$$ 
where $\cN$ is a nonlinear differential operator (hence $\cN$ is local), and $\cI_m$ is given by a complicated Cauchy integral.  A natural idea is to consider $\mathcal F_m=\cN-\cI_m$, and prove that the Fr\'echet derivative $\mathcal F_m'$ near Kaden's spiral \eqref{special_sol} is (uniformly) invertible in some well-chosen functional spaces, and then apply the Newton-Kantorovich method. Here we emphasize that Kaden's spiral does NOT solve $\mathcal F_m=0$ for any $m<+\infty$. However, this approach seems overly ambitious, as the linearization of the integral  $\cI_m$ is quite difficult  to analyze. Our observation is that Kaden's spiral indeed solves $\cN=0$, and when $m$ is sufficiently large, $\cI_m$ can be viewed as a perturbation; see Proposition \ref{Prop.I}. The proof of Proposition \ref{Prop.I} is much involved. At this stage, in order to solve $N[r,\g]=\cI_m[r,\g]$, it suffices to apply Banach's fixed point theorem to the nonlinear map $\cN^{-1}\circ\cI_m$.

The proof of the invertibility of $\cN$ is also highly non-trivial. For this, we need to prove the invertibility of the linearized operator $\mathcal L$ of $\cN$ at Kaden's spiral, which takes (see Lemma \ref{Lem.L-expression})
\begin{equation*}
	\mathcal L[r,\g](\th)=\begin{pmatrix}
		\mu\th^{-\mu}r(\th)+\th^{1-\mu}r'(\th)+\frac\mu{2\pi}\th^{-1}\g(\th)+\frac\mu{2\pi(2\mu-1)}\g'(\th)\\
		-2\th^{1-\mu}r(\th)+\frac{\th\g'(\th)}{2\pi(2\mu-1)}
	\end{pmatrix}.
\end{equation*}
By carefully exploring the structure of $\cL$, we are able to construct the inverse of $\cL$(see Lemma \ref{Lem.inversion_L}), thereby proving the invertibility of $\cN$.

Finally, the analysis of the Cauchy integral $\cI_m$ is the most challenging part of this paper. This is because the curve $\th\mapsto r_0(\th)$ in \eqref{Eq.I_m} is given by Kaden's spiral \eqref{special_sol}, which is not a chord-arc curve. Our approach is to decompose $\cI_m$ into a summation of oscillation integrals and then perform weighted $W^{1,\infty}$ estimates on these integrals; see Proposition \ref{Prop.P_n_est}. Our method provides new insights into the analysis of singular integrals along singular curves.

\subsection{Notations and conventions}

\begin{itemize}
	\item $\R_+:=(0,+\infty)$, $\Z_+:=\Z\cap \R_+$. $\langle\th\rangle:=(1+\th^2)^{1/2}$ for $\th\in\R$.
	
	\item Throughout this paper, we identify $\R^2\simeq\C$. For a complex number $z\in\C$, we denote its real part by $\operatorname{Re}z$, its imaginary part by $\operatorname{Im}z$, and its complex conjugate by $z^*$. 
	
	\item The logarithmic function, denoted by $\ln$, while acting on some complex number, should be understood as the principal branch unless otherwise specified.
	
	\item In a Banach space $X$, we denote the open ball $B_X(x_0, R):=\{x\in X: \|x-x_0\|<R\}$ for $x_0\in X$ and $R>0$. We abbreviate $B_X(0, R)$ to $B_X(R)$. Additionally, $\overline{B_X(x_0, R)}$ denotes the closed ball centered at $x_0\in X$ with radius $R$. Similarly for $\overline{B_X(R)}$.
	
	\item Let $X$ and $Y$ be two Banach spaces, we denote $L(X, Y)$ as the set consisting of all bounded linear operators from $X$ to $Y$. 
	
	\item For a map $\mathcal F: U\subset X\to Y$ between two functional spaces (where $U\subset X$ is a subset), we use the square bracket symbol $\mathcal F[x]$ (or just $\cF x$ for simplicity) to denote the image of $x\in U$ under the map $\mathcal F$. Hence, $\mathcal F[x]=\cF x\in Y$ is a function. We use the parenthesis symbol $\mathcal F[x](\th)$ to denote the value of the function $\cF[x]$ at some point $\th$. The Fr\'echet derivative of $\cF$ is denoted by $\cF'$, which is a map from $U$ to $L(X,Y)$. Hence, $\cF[x_0][x]\in Y$ for $x_0\in U$ and $x\in X$. The second-order derivative is $\cF'': U\to L(X\times X, Y)=L(X, L(X, Y))$. Hence, $\cF''[x_0][x_1, x_2]\in Y$ for $x_0\in U$ and $x_1, x_2\in X$. 
	
	\item There is an ambiguity in the use of $(r,\g)$ as functions in $\theta$. We aim to solve \eqref{maineq} for $(r,\g)$ near the approximation solution $(r_0,\g_0)$. However, in the solving process, we still denote the perturbation by $(r,\g)$ (which satisfies \eqref{Eq.perturbation}). Hence,  $(r+r_0,\g+\g_0)$ is a solution to \eqref{maineq}. The context will clarify which usage is intended.	
\end{itemize}


\section{A road map of the proof}\label{Sec.proof-sketch}
In this section, we outline a road map for proving Theorem \ref{mainthm} and Theorem \ref{mainthm2}, starting with a brief derivation of the Birkhoff-Rott equation \eqref{BR}.

\subsection{Birkhoff-Rott equation}\label{Subsec.BR}

In the study of vortex sheets, it is natural to explicitly propagate the sheet itself using a time-dependent parametrization. Since the seminal contributions of Birkhoff \cite{Birkhoff}, and Rott \cite{Rott}, it has been well-known that the dynamical evolution of a vortex sheet is precisely governed by a nonlinear integro-differential equation, which is called the {\it Birkhoff-Rott equation}. In this subsection, we set out to formally derive the Birkhoff-Rott equation corresponding to the m-fold vortex sheet solutions that are the focus of our investigation, with the initial data given by \eqref{Eq.initial-vorticity}.

We should remark that the derivation process presented in this subsection closely mirrors the conventional approach for vortex sheet solutions with vorticity supporting on a single curve, a method widely elaborated in numerous classical books ({\it e.g.} \cite{MB,MP}). Given the prevalence of this method in standard references, we shall streamline our exposition by omitting some of the more detailed steps. Readers seeking a comprehensive understanding are directed to consult \cite[Section 9.2]{MB} for further clarification and in-depth analysis.

Since the initial sheet $\Sigma_0$ is the right semi-infinite horizontal axis, it is natural to parameterize $\Sigma_0$ by the distance from the origin, i.e.,
\begin{equation}\label{Eq.Sigma0}
	\Sigma_0: X(0,r)=r,\quad r>0.
\end{equation}
Here we emphasize again that we have identified $\R^2\simeq\C$ throughout this paper. With the parametrization \eqref{Eq.Sigma0}, it follows from \eqref{Eq.initial-vorticity} that
\begin{equation*}
	\int_{\R^2}\varphi(x)\,\mathrm d\omega_0(x)=\frac1m\sum_{k=0}^{m-1}\int_0^\infty \varphi\left(\xi_m^kX(0,r)\right)w(0,r)|X_r(0,r)|\,\mathrm dr,\quad\forall\ \varphi\in C_c(\R^2),
\end{equation*}
where $w(0,r)=r^{1-\frac1\mu}$ and $X_r(0,r)=1$ is the derivative of $X(0,r)$ with respect to $r$. 

Assume that our solution $\omega(t)$ pertains the $m$-fold symmetry and supports on $m$ curves $\{\xi_m^k\Sigma_t: k\in\Z\cap[0, m-1]\}$ for $t>0$, where $\Sigma_t$ is parameterized by $X(t,r)$ in a Lagrangian sense. This is to say, for a particle lying initially at $X(0,r)$, its position at time $t>0$ is $X(t,r)$ (thus $|X(t,r)|\neq r$ in general for $t>0$). We denote the vortex strength at $X(t,r)$ by $w(t,r)\in\R$, then
\begin{equation}\label{Eq.omega(t)-pairing}
	\int_{\R^2}\varphi(x)\,\mathrm d\omega(t,x)=\frac1m\sum_{k=0}^{m-1}\int_0^\infty \varphi\left(\xi_m^kX(t,r)\right)w(t,r)|X_r(t,r)|\,\mathrm dr,\quad\forall\ \varphi\in C_c(\R^2).
\end{equation}
By Biot-Savart law and \eqref{Eq.omega(t)-pairing}, for any $t>0$ and $x\notin \{\xi_m^k\Sigma_t: k\in\Z\cap[0, m-1]\}$, we have
\begin{align}
	u(t,x)&=\frac1m\sum_{k=0}^{m-1}\int_0^\infty K_2\left(x-\xi_m^kX(t,r)\right)w(t,r)|X_r(t,r)|\,\mathrm dr\nonumber\\
	&=\frac1m\sum_{k=0}^{m-1}\int_0^\infty\frac{\ii}{2\pi}\frac{w(t,r)|X_r(t,r)|\,\mathrm dr}{\left(x-\xi_m^kX(t,r)\right)^*}.\label{Eq.velocity-expression}
\end{align}
By means of standard complex variables, the limit as $x$ approaches a point on $\Sigma_t$ from either side gives two different velocities $u_+$ and $u_-$ with a same normal component. Using the fact that the sheet advects with the speed $(u_++u_-)/2$, we obtain
\begin{equation}\label{Eq.X-eq}
	\pa_tX(t,r)=\frac1m\sum_{k=0}^{m-1}\pv\int_0^\infty\frac{\ii}{2\pi}\frac{w(t,r')|X_r(t,r')|\,\mathrm dr'}{\left(X(t,r)-\xi_m^kX(t,r')\right)^*},\quad\forall\ t>0,\ r>0.
\end{equation}
Using Bernoulli's law (see \cite[Section 9.2]{MB}) for the potential function, one has the conservation of circulation along particle paths of the flow, i.e.,
\begin{equation}\label{Eq.conservation-cirulation}
	w(t,r)|X_r(t,r)|=w(0,r)|X_r(0,r)|=r^{1-\frac1\mu},\quad\forall\ t>0,\ r>0.
\end{equation}
We introduce another parameter (circulation) $\Gamma$ by $\Gamma=-\frac{\mu}{2\mu-1}r^{2-\frac1\mu}$, then $\mathrm d\Gamma=-r^{1-\frac{1}{\mu}}\mathrm dr$, which along with 
\eqref{Eq.X-eq} and \eqref{Eq.conservation-cirulation} implies that
\begin{equation*}
	\partial_t Z^*(t,\G)=\frac1{2m\pi\mathrm{i}}\sum_{k=0}^{m-1}\pv\int_{-\infty}^0\frac{\mathrm d\Gamma'}{Z(t,\Gamma)-\xi_m^kZ(t,\Gamma')},
\end{equation*}
where $Z(t,\Gamma)=X(t,r)=X\left(t,\left(-\frac{2\mu-1}{\mu}\Gamma\right)^{\frac\mu{2\mu-1}}\right)$. This is our Birkhoff-Rott equation \eqref{BR}.

The rigorous justification of the equivalence between velocity formulation \eqref{Eq.Euler-velocity} (in a weak sense) and Birkhoff-Rott equation \eqref{Eq.BR-normal} for $C^1$ graphs $Z(t,\cdot)$ and $C^1$ vortex strength can be found in \cite{MP}. The proof can be trivially extended to our scenario \eqref{BR} for $m$ sheets. It was shown in \cite{LNS2007} that the equivalence holds if the sheets are regular curves and the strength is (locally) square integrable with respect to arc-length. So, if $Z(t,\cdot)$ is a $C^1$ solution to \eqref{BR} satisfying \eqref{Eq.center-behavior}, then 
\begin{equation}
	u(t,x)=\frac{\ii}{2\pi m}\sum_{k=0}^{m-1}\int_{-\infty}^0\frac{\mathrm d\Gamma}{\left(x-\xi_m^kZ(t,\Gamma)\right)^*},\quad\forall\ x\notin \{\xi_m^k\Sigma_t: k\in\Z\cap[0, m-1]\}
\end{equation}
is a weak solution to \eqref{Eq.Euler-velocity} in $\R^2\setminus\{0\}$. It suffices to prove that $u$ is a weak solution to \eqref{Eq.Euler-velocity} in the whole plane $\R^2$, which requires some quantitive estimates of the velocity fields, see Section \ref{Sec.Proof-weak-sol} for details. 

\subsection{Self-similar formulation}

We introduce the self-similar variables:
\begin{equation}\label{ss_variable}
	Z(t,\G)=t^\mu z(\g),\qquad \G=t^{2\mu-1}\g,
\end{equation}
where $\mu>0$. Inserting \eqref{ss_variable} into \eqref{BR} gives the equation for $z(\g)$:
\begin{equation}\label{ss_BR1}
	\mu z^*(\g)+(1-2\mu)\g\left(z^*\right)'(\g)=\frac1{2\pi\ii}\frac1m\sum_{k=0}^{m-1}\pv\int_{-\infty}^0 \frac{\mathrm d\wt{\g}}{z(\g)-\xi_m^k z(\wt\g)}.
\end{equation}
Using the elementary identity
\[\sum_{k=0}^{m-1}\frac1{z-\xi_m^kw}=\frac{mz^{m-1}}{z^m-w^m},\]
the self-similar equation \eqref{ss_BR1} is converted to
\begin{equation}\label{ss_BR2}
	\mu z^*(\g)+(1-2\mu)\g\left(z^*\right)'(\g)=\frac1{2\pi\ii}\pv\int_{-\infty}^0\frac{\left(z(\g)\right)^{m-1}}{\left(z(\g)\right)^m-\left(z(\wt\g)\right)^m}\,\mathrm d\wt\g.
\end{equation}

Now we consider the equation in the polar coordinates. Write $z(\g)=r(\theta(\g))\e^{\ii\theta(\g)}$, then
\[z'(\g)=r'(\theta)\frac1{\g'(\theta)}\e^{\ii\theta}+r(\theta)\ii \e^{\ii\theta}\frac1{\g'(\theta)}.\]
Hence, \eqref{ss_BR2} becomes
\[\left(\mu r\e^{\ii\theta}+(1-2\mu)\g\frac{r'(\theta)+\ii r(\theta)}{\g'}\e^{\ii\theta}\right)^*=\frac1{2\pi\ii}\pv\int_0^\infty\frac{\left(r(\theta)\right)^{m-1}e^{\ii(m-1)\theta}\g'(\wt\theta)}{\left(r(\theta)\right)^me^{\ii m\theta}-\left(r(\wt\theta)\right)^me^{\ii m\wt\theta}}\,\mathrm d\wt\theta,\]
i.e.,
\begin{equation}\label{maineq}
	\mu r(\theta)^2+(1-2\mu)\frac{\g(\theta)}{\g'(\theta)}\left(r(\theta)r'(\theta)-\ii r(\theta)^2\right)=\frac1{2\pi\ii}\pv\int_0^\infty\frac{\g'(\wt\theta)\,\mathrm d\wt\theta}{1-\left(\frac{r(\wt\theta)}{r(\theta)}\right)^m\e^{\ii m(\wt\theta-\theta)}}.
\end{equation}
This is the main equation we want to solve. Note that for fixed $m\in\N$, we have two unknowns $r(\theta), \g(\theta)$ and two equations: the real part of \eqref{maineq} and the imaginary part of \eqref{maineq}.

\subsection{Kaden's algebraic spiral solutions}

We consider the limit $m\rightarrow \infty$. Assume that $r'(\th)<0$, then we formally get the limiting equation of \eqref{maineq}:
\begin{equation}\label{eq_limit}
	\mu r(\theta)^2+(1-2\mu)\frac{\g(\theta)}{\g'(\theta)}\left(r(\theta)r'(\theta)-\ii r(\theta)^2\right)=\frac1{2\pi\ii}\pv\int_\theta^\infty \g'(\wt\th)\,\mathrm d\wt\th=\frac{-\g(\th)}{2\pi\ii}.
\end{equation}
Letting $f=r^2$, the equation \eqref{eq_limit} becomes
\[\mu f+(1-2\mu)\frac\g{\g'}\left(\frac12f'-\ii f\right)=-\frac\g{2\pi\ii},\]
which is equivalent to 
\[\mu f+\frac{1-2\mu}2\frac\g{\g'}f'=0\  (\text{real part})\qquad \text{and} \qquad (2\mu-1)\frac\g{\g'}f=\frac\g{2\pi}\ (\text{imaginary part}).\]
Rearranging them, we obtain
\[\frac{2\mu-1}{2\mu}\frac{f'}f=\frac{\g'}\g\qquad \text{and} \qquad 2\pi(2\mu-1)f=\g'.\]
It follows from $\frac{2\mu-1}{2\mu}\frac{f'}f=\frac{\g'}\g$ that $\g=C_1 f^{\frac{2\mu-1}{2\mu}}$ for some constant $C_1$. Combining this with $2\pi(2\mu-1)f=\g'$ gives that
\[C_1\frac{2\mu-1}{2\mu} f^{-\frac1{2\mu}}f'=2\pi(2\mu-1)f,\]
i.e.,
\[\frac1{2\mu}f^{-1-\frac1{2\mu}}f'=\frac{2\pi}{C_1}\qquad\Longrightarrow\qquad \left(f^{-\frac1{2\mu}}\right)'=-\frac{2\pi}{C_1}.\]
Thus, there exists a constant $C_2$ such that
\[f^{-\frac1{2\mu}}=-\frac{2\pi}{C_1}\th+C_2\qquad\Longrightarrow\qquad f(\th)=\left(-\frac{2\pi}{C_1}\th+C_2\right)^{-2\mu}.\]
Therefore, the solution to \eqref{eq_limit} is given by
\begin{equation}\label{gensol_eq_limit} 
	r(\th)=\left(-\frac{2\pi}{C_1}\th+C_2\right)^{-\mu},\qquad \g(\th)=C_1\left(-\frac{2\pi}{C_1}\th+C_2\right)^{1-2\mu}.
\end{equation}
Choosing $C_1=-2\pi, C_2=0$, we obtain the special solution
\begin{equation}\label{special_sol}
	r_0(\th)=\th^{-\mu},\qquad \g_0(\th)=-2\pi\th^{1-2\mu}.
\end{equation}
This is exactly Kaden's algebraic spiral solution \cite{Kaden,MS1973}.

\subsection{Existence results for self-similar Birkhoff-Rott}
We seek solutions to \eqref{maineq} near $(r_0, \g_0)$. We denote the perturbation still by $(r,\g)$. Then $(r,\g)$ solves the following nonlinear equation (thus $(r+r_0,\g+\g_0)$ solves \eqref{maineq}):
\begin{equation}\label{Eq.perturbation}
	\mathcal N[r,\g]=\mathcal I_m[r,\g],
\end{equation}
where
\begin{align}
		\mathcal N[r,\g](\theta):=\mu&\left(r(\theta)+r_0(\theta)\right)^2+\frac{\g(\th)+\g_0(\th)}{2\pi\ii}\label{Eq.N}\\
		+(1-2\mu)&\frac{\g(\th)+\g_0(\th)}{\g'(\th)+\g_0'(\th)}\left[\left(r(\th)+r_0(\th)\right)\left(r'(\th)+r_0'(\th)\right)-\ii\left(r(\th)+r_0(\th)\right)^2\right],\nonumber\\
		\mathcal I_m[r,\g](\th):=\frac1{2\pi\ii}&\left[\pv\int_0^\infty\frac{\left(\g'(\wt\th)+\g_0'(\wt\th)\right)\,\mathrm d\wt\th}{1-\left(\frac{r(\wt\th)+r_0(\wt\th)}{r(\th)+r_0(\th)}\right)^m\mathrm e^{\ii m(\wt\th-\th)}}-\int_\th^\infty\left(\g'(\wt\th)+\g_0'(\wt\th)\right)\,\mathrm d\wt\th\right].\label{Eq.I_m}
\end{align}

It is noteworthy that the nonlinear operator $\mathcal{N}$ is independent of the parameter $m$, which endows it with a more straightforward analytical framework compared to $\mathcal{I}_m$. The analysis of $\mathcal{I}_m$ is complicated due to the presence of a principal value integral. Furthermore, we observe that $\mathcal{N}[0,0] = 0$, a conclusion that follows naturally from the fact that $(r_0, \gamma_0)$ is a solution to the limiting equation \eqref{eq_limit}.

Our approach to solving equation \eqref{Eq.perturbation} relies on an application of Banach's fixed point theorem in some appropriate Banach spaces $X$ and $Y$. The first goal is to establish the invertibility of the nonlinear operator $\mathcal{N}$ when mapping a small ball within $X$ to $Y$. By achieving this, the equation \eqref{Eq.perturbation} can be reformulated into a fixed-point form\footnote{\label{Footnote.1}In the definitions provided by \eqref{Eq.N} and \eqref{Eq.I_m}, the operators $\mathcal{N}$ and $\mathcal{I}_m$ are initially formulated to map a pair of real-valued functions $(r, \gamma)$ to a complex-valued function of the structure $f + \ii g$. Throughout the subsequent development, we reinterpret these operators as mappings from $(r, \gamma)$ to another pair $(f, g)$, where $f$ is the real part and $g$ the imaginary part. This alternative perspective is adopted starting from \eqref{Eq.fixed-point}.}
\begin{equation}\label{Eq.fixed-point}
	(r,\g)=(\mathcal N^{-1}\circ\mathcal I_m)[r,\g]=:\mathcal F_m[r,\g].
\end{equation}
Then we establish the required estimates to apply the classical Banach's fixed point theorem.

Motivated by nature of the singular integral in $\mathcal I_m$ and \cite{EJ2023}, we are going to work in some weighted H\"older-type spaces. Let $\alpha\in(0,1)$. For each $k\in[0, +\infty)$, we define
\begin{equation}\label{Eq.G^k_def}
	\mathscr G^k:=\{f\in C(\R_+;\C): \|f\|_{k}<+\infty\},\quad \mathcal G^k:=\{f\in \mathscr G^k: \text{$f$ is real-valued}\},
\end{equation}
where the norm is given by
\begin{equation}\label{Eq.G^k_norm}
	\|f\|_{k}:=\|\langle\theta\rangle^\alpha f\|_{L^\infty}+[f]_k,\quad [f]_k:=\sup_{0<\th_2<\th_1<\th_2+1}\langle\th_1\rangle^{k+\alpha}\frac{|f(\th_1)-f(\th_2)|}{|\th_1-\th_2|^\alpha}.
\end{equation}
Note that if $f\in \mathscr G^k$ for some $k\in[0,+\infty)$, then the limit $f(0+):=\lim_{\th\to0+} f(\th)$ exists in $\C$. We denote
\begin{equation}\label{Eq.G^0_0}
	\mathscr G_0^k:=\left\{f\in\mathscr G^k: f(0+)=0\right\}, \quad \mathcal G_0^k:=\left\{f\in\mathcal G^k: f(0+)=0\right\}.
\end{equation}
Our working functional spaces are as follows
\begin{align}
	X:&=\left\{(r,\g): \th^\mu r\in\cG^1, \th^{\mu+1}r'\in\cG^0, \th^{2\mu-1}\g\in\cG^1,\th^{2\mu}\g'\in\cG^1\right\},\\
	Y:&=\left\{(f,g): \th^{2\mu}f\in \cG_0^0, \th^{2\mu-1}g\in\cG^1\right\},\label{Eq.Y}
\end{align}
with the norms given by
\begin{align}
	&\|(r,\g)\|_{X}:=\|r\|_{X_r}+\|\g\|_{X_\g},\quad \|(f,g)\|_Y:=\|\th^{2\mu}f\|_0+\|\th^{2\mu-1}g\|_1,\\
	&\|r\|_{X_r}:=\|\th^\mu r\|_1+\|\th^{\mu+1}r'\|_0,\quad \|\g\|_{X_\g}:=\|\th^{2\mu-1}\g\|_1+\|\th^{2\mu}\g'\|_1.
\end{align}

Equipped with these functional spaces, we are now in a position to present the existence result for \eqref{Eq.perturbation} as follows.

\begin{theorem}\label{Thm.main}
	Let $\mu>1/2$ and $\al\in(0,1)$. Then there exists $m_*>0$ such that for all $m>m_*$, the equation \eqref{Eq.perturbation} has a solution $(r_m,\g_m)\in X$, which is unique in the closed ball $\overline{B_X(1/m)}$, and moreover, there is a constant $C>0$ independent of $m>m_*$ such that
	\begin{equation}\label{Eq.solution_bound}
		\|(r_m,\g_m)\|_X\leq \frac C{m^2}<\frac12,\quad\forall\ m>m_*.
	\end{equation}
\end{theorem}

\subsection{Proof of Theorem \ref{Thm.main}}
The proof of Theorem \ref{Thm.main} is based on Banach's fixed point theorem and the following two propositions.

\begin{proposition}[Invertibility of $\cN$]\label{Prop.N}
	There exist $\varepsilon_0\in(0,1)$, $\delta_0\in(0,1)$ and a constant $C^*>1$ such that 
	\begin{itemize}
		\item $\mathcal N$ is a $C^\infty$ map from $B_X(\varepsilon_0)$ to $Y$;
		\item for any $y\in B_Y(\delta_0)$, there is a unique $x\in B_X(\varepsilon_0)$ such that $\mathcal Nx=y$, and 
		\begin{equation}\label{Eq.N^-1_bound0}
			\|x\|_X\leq C^*\|y\|_Y;
		\end{equation}
		\item the map $\mathcal N^{-1}: B_Y(\delta_0)\to B_X(\varepsilon_0)$ is Lipschitz continuous:
		\begin{equation}\label{Eq.N^-1_lip}
			\left\|\mathcal N^{-1}y_1-\mathcal N^{-1}y_2\right\|_X\leq C^*\|y_1-y_2\|_Y,\quad\forall\ y_1, y_2\in B_Y(\delta_0).
		\end{equation}
	\end{itemize}
\end{proposition}


\begin{proposition}[Contraction of $\mathcal I_m$]\label{Prop.I}
	There exist $\varepsilon_1\in (0,1)$, $m_1\in\Z_+$ and a constant $C_*>1$ such that for all $m>m_1$ there hold
	\begin{align}
		&\left\|\mathcal I_m[0,0]\right\|_Y\leq \frac{C_*}{m^2},\label{Eq.I00}\\
		\left\|\mathcal I_m[r_1,\g_1]-\mathcal I_m[r_1, \g_2]\right\|_Y&\leq \frac{C_*}{m}\|(r_1, \g_1)-(r_2, \g_2)\|_X,\quad\forall\ (r_1,\g_1), (r_2,\g_2)\in B_X(\varepsilon_1).\label{Eq.I_Lip}
	\end{align}
\end{proposition}

Now we prove Theorem \ref{Thm.main} by using Proposition \ref{Prop.N} and Proposition \ref{Prop.I}.

\begin{proof}[Proof of Theorem \ref{Thm.main}]
	Let $$m_*:=2C^*C_*+\frac2{\min\{\varepsilon_0, \varepsilon_1\}}+\sqrt{3C_*/\delta_0}+m_1.$$ For $m>m_*$, we let $\varepsilon_m:=1/m\in(0,1)$, then $\varepsilon_m<\min\{\varepsilon_0, \varepsilon_1\}$. For $(r,\g)\in \overline{B_X(\varepsilon_m)}$, by \eqref{Eq.N^-1_bound0}, \eqref{Eq.I00} and \eqref{Eq.I_Lip}, we have
	\begin{align*}
&\|\mathcal I_m[r,\g]\|_Y\leq C_*\left(\frac1{m^2}+\frac1m\|(r,\g)\|_X\right)\leq C_*\frac2{m^2}\leq\frac 2{3}\delta_0<\delta_0,\\
		&\left\|\mathcal F_m[r,\g]\right\|_X=\|\mathcal N^{-1}\mathcal I_m[r,\g]\|_X\leq C^*\|\mathcal I_m[r,\g]\|_Y\leq C^*C_*\frac2{m^2}<\frac{m_*}{m^2}<\frac1m=\varepsilon_m.
	\end{align*}
	For $(r_1,\g_1), (r_2, \g_2)\in \overline{B_X(\varepsilon_m)}$, by \eqref{Eq.N^-1_lip} and \eqref{Eq.I_Lip}, we have
	\begin{align}
		\left\|\mathcal F_m[r_1,\g_1]-\mathcal F_m[r_2,\g_2]\right\|_X&\leq C^*\left\|\mathcal I_m[r_1,\g_1]-\mathcal I_m[r_1, \g_2]\right\|_Y\leq \frac{C^*C_*}{m}\|(r_1, \g_1)-(r_2, \g_2)\|_X\nonumber\\
		&<\frac12\|(r_1, \g_1)-(r_2, \g_2)\|_X.\label{Eq.contraction}
	\end{align}
	As a consequence, for each $m>m_*$, the nonlinear map
	\[\mathcal F_m: \overline{B_X(\varepsilon_m)}\to \overline{B_X(\varepsilon_m)}\]
	is a contraction. Then Banach's fixed point theorem implies that \eqref{Eq.fixed-point} has a unique solution $(r_m,\g_m)\in \overline{B_X(\varepsilon_m)}$, hence \eqref{Eq.perturbation} has a unique solution $(r_m,\g_m)\in \overline{B_X(\varepsilon_m)}$.
	
	Finally, we prove \eqref{Eq.solution_bound}. By \eqref{Eq.fixed-point} and \eqref{Eq.contraction}, we have
	\begin{align*}
		\|(r_m,\g_m)\|_X&=\left\|\mathcal F_m[r_m,\g_m]\right\|_X\leq \left\|\mathcal F_m[r_m,\g_m]-\mathcal F_m[0,0]\right\|_X+\|\mathcal F_m[0,0]\|_X\\
		&<\frac12\|(r_m,\g_m)\|_X+\|\mathcal F_m[0,0]\|_X.
	\end{align*}
	Then $\|(r_m,\g_m)\|_X\leq 2\|\mathcal F_m[0,0]\|_X$. Using \eqref{Eq.N^-1_bound0} and \eqref{Eq.I00} gives that
	\[\|(r_m,\g_m)\|_X\leq 2C^*\frac{C_*}{m^2},\quad \forall\ m>m_*.\]
	This completes the proof.
\end{proof}

\subsection{Sketch of the proof: Proposition \ref{Prop.N}}\label{Subsec.sketch_N}
The proof of Proposition \ref{Prop.N} can be found in Section \ref{Sec.N}. Here we briefly sketch the proof.

The first step is to show that $\mathcal N$ is a $C^\infty$ map from a small ball near $(0,0)\in X$ to $Y$. Recalling the definition of $\mathcal N$ in \eqref{Eq.N} and footnote \ref{Footnote.1}, we establish some Banach-algebraic properties of our functional spaces $\mathcal G^k$ and $\mathcal G^k\oplus \R$ (see Section \ref{Sec.spaces}), and then we use abstract properties of Banach algebras to conclude the $C^\infty$ regularity of the nonlinear map $\mathcal N$ (see Subsection \ref{Subsec.nonlinear_smooth}).

To show the invertibility of $\mathcal N$ in a small ball near $(0,0)\in X$, it is natural to consider the linearization of $\mathcal N$ near $(0,0)$. Our main goal is to prove the invertibility of the linearized operator $\mathcal L: X\to Y$. A direct computation gives that (see Lemma \ref{Lem.L-expression})
\begin{equation*}
	\mathcal L[r,\g](\th)=\begin{pmatrix}
		\mu\th^{-\mu}r(\th)+\th^{1-\mu}r'(\th)+\frac\mu{2\pi}\th^{-1}\g(\th)+\frac\mu{2\pi(2\mu-1)}\g'(\th)\\
		-2\th^{1-\mu}r(\th)+\frac{\th\g'(\th)}{2\pi(2\mu-1)}
	\end{pmatrix}.
\end{equation*}
Using some ODE techniques, we can compute the inversion of $\mathcal L$ explicitly (see Lemma \ref{Lem.inversion_L}). Then one can check that $\mathcal L: X\to Y$ is bijective. 

The invertibility of $\mathcal N$ is based on an application of Banach's fixed point theorem. Given $y\in Y$, we would like to solve the equation $\mathcal Nx=y$ for $x$. Let $\mathcal N=\mathcal L+\mathcal N_0$, then the equation $\mathcal Nx=y$ is equivalent to $x=\mathcal L^{-1}[y-\mathcal N_0x]$. Applying Banach's fixed point theorem on the map $\mathcal L^{-1}y-\mathcal L^{-1}\circ\mathcal N_0$ gives the unique solution of $\mathcal Nx=y$. The fixed-point formulation also implies the Lipschitz continuity of $\mathcal N^{-1}$. See Subsection \ref{Subsec.invert_nonlinear}. 

\subsection{Sketch of the proof: Proposition \ref{Prop.I}}\label{Subsec.sketch_I} 
The proof of Proposition \ref{Prop.I} is the most difficult part of this paper. Here we briefly sketch the proof and readers can find the full details in Section \ref{Sec.I}. 

Fundamentally, we are going to demonstrate the boundedness of a singular integral operator within specific weighted H\"older spaces. Proving such estimates typically poses significant challenges, and when feasible, the proof process tends to be intricate and of high complexity. In this paper, we aim to present our proof in a manner that enhances its comprehensibility.

The starting point of our idea is to rewrite the singular integral operator $\mathcal I_m$ in the form of a series:
\begin{equation}\label{Eq.I_series}
	\mathcal I_m[r,\g](\th)=-(2\mu-1)\ii\sum_{n=1}^\infty \cQ_{mn}[r,\g](\th),
\end{equation}
where $\cQ_n:=-\cQ_{n,1}+\cQ_{n,2}$, and
\begin{align}
	\cQ_{n,1}[r,\g](\th):&=\int_0^\th \wt\th^{-2\mu}\left(1+\frac{\wt\th^{2\mu}\g'(\wt\th)}{2\pi(2\mu-1)}\right)\left(\frac{\wt\th}{\th}\right)^{\mu n}\left(\frac{1+\wt\th^\mu r(\wt\th)}{1+\th^\mu r(\th)}\right)^{-n}\e^{\ii n(\th-\wt\th)}\,\mathrm d\wt\th,\label{Eq.Q_n1}\\
	\cQ_{n,2}[r,\g](\th):&=\int_\th^\infty \wt\th^{-2\mu}\left(1+\frac{\wt\th^{2\mu}\g'(\wt\th)}{2\pi(2\mu-1)}\right)\left(\frac{\wt\th}{\th}\right)^{-\mu n}\left(\frac{1+\wt\th^\mu r(\wt\th)}{1+\th^\mu r(\th)}\right)^n\e^{\ii n(\wt\th-\th)}\,\mathrm d\wt\th.\label{Eq.Q_n2}
\end{align}
See Subsection \ref{Subsec.reformulation}.  The significance of the series \eqref{Eq.I_series} resides in its ability to transform the estimation of weighted H\"older norms into that of weighted $W^{1,\infty}$ norms, making it more tractable. Additionally, the explicit forms given by equations \eqref{Eq.Q_n1} and \eqref{Eq.Q_n2} facilitate direct differentiation with respect to $\theta$.

Given $(r,\g)\in X$, to estimate the weighted H\"older norm of $\mathcal I_m[r,\g]$, it suffices to prove some corresponding weighted $C^0$ and $C^1$ estimates on each $\mathcal Q_n[r,\g]$. To be more precise, let $$\mathcal P_n[r,\g](\th):=(-\mu+\ii\th)\th^{2\mu-1}\mathcal Q_n[r,\g](\th),$$
then we are left to prove the following
\begin{proposition}\label{Prop.P_n_est}
	There exist $\varepsilon_1\in(0,1)$, $N\in\N_+$ and a constant $C>0$ such that for each $n>N$, we can decompose $\cP_n$ as
	\begin{equation}\label{Eq.P_n_decompose0}
		\cP_n=\cP_n^{(1)}+\cP_n^{(2)},
	\end{equation}
	with the estimates
	\begin{align}
		\left|\cP_n^{(1)}[r,\g](\th)\right|&\leq \frac{C}{n^{1+\alpha}}\frac{\th^\alpha}{\langle\th\rangle^{2\alpha}}\left(\frac1n+\|(r,\g)\|_X\right),\label{Eq.P_n^1est}\\
		\left|\cP_n^{(1)}[r_1,\g_1](\th)-\cP_n^{(1)}[r_2,\g_2](\th)\right|&\leq \frac{C}{n^{1+\alpha}}\frac{\th^\alpha}{\langle\th\rangle^{2\alpha}}\left\|(r_1,\g_1)-(r_2,\g_2)\right\|_X,\label{Eq.P_n^1est_diff}\\
		\left\|\cP_n^{(2)}[r,\g]\right\|_{1}&\leq \frac{C}{n^2},\label{Eq.P_n^2est}\\
		\left\|\cP_n^{(2)}[r_1,\g_1]-\cP_n^{(2)}[r_2,\g_2]\right\|_{1}&\leq \frac{C}{n^2}\left\|(r_1,\g_1)-(r_2,\g_2)\right\|_X,\label{Eq.P_n^2est_diff}\\
		\left|\left(\cP_n[r,\g]\right)'(\th)\right|&\leq \frac C{n^\alpha}\th^{\alpha-1}\langle\th\rangle^{1-2\al}\left(\frac1n+\|(r,\g)\|_X\right),\label{Eq.C^1est}\\
		\left|\left(\cP_n[r_1,\g_1]\right)'(\th)-\left(\cP_n[r_2,\g_2]\right)'(\th)\right|&\leq \frac C{n^\alpha}\th^{\alpha-1}\langle\th\rangle^{1-2\al}\left\|(r_1,\g_1)-(r_2,\g_2)\right\|_X,\label{Eq.C^1est_diff}
	\end{align}
	for all $(r,\g), (r_1,\g_1), (r_2,\g_2)\in B_X(\varepsilon_1)$ and all $\th\in(0,+\infty)$.
\end{proposition}

Proposition \ref{Prop.I} is a direct consequence of Proposition \ref{Prop.P_n_est}, along with a standard summation argument (see Subsection \ref{Subsec.Proof-contraction}). 

A crucial observation is that $\mathcal{P}_n[r, \gamma]$ can be decomposed into two components. The first part, denoted by $\mathcal{P}_n^{(1)}$, corresponds to the derivative $\left( \mathcal{P}_n[r, \gamma] \right)'$, while the second part, $\mathcal{P}_n^{(2)}$, is directly related to $(r,\g)$ in a local manner (see \eqref{Eq.P_n_decompose}). Consequently, estimating $\mathcal{P}_n^{(2)}$ is straightforward (see Subsection \ref{Subsec.P_n^2est}), and the $C^0$ bound for $\mathcal{P}_n^{(1)}$ can be inferred from that of $\left( \mathcal{P}_n[r, \gamma] \right)'$ (see Proposition \ref{Prop.P_n,1,2}). Therefore, in Proposition \ref{Prop.P_n_est}, the most demanding tasks are the proofs of the $C^1$ estimates in  \eqref{Eq.C^1est} and \eqref{Eq.C^1est_diff}, see Subsection \ref{Subsec.P_n-W^1,infty} for full details.

\section{Properties of the weighted H\"older spaces}\label{Sec.spaces}

Building upon the preliminary discussion in Subsection \ref{Subsec.sketch_N}, our initial task is to establish certain Banach algebraic properties inherent to our weighted H\"older spaces. These properties are important for two main purposes. Firstly, they are helpful in demonstrating the regularity of the nonlinear mapping $\mathcal{N}$. Secondly, they play a crucial role in the proof of Proposition \ref{Prop.P_n_est}. This is especially important due to the high nonlinearity exhibited in \eqref{Eq.Q_n,1} and \eqref{Eq.Q_n,2}.

Recall the definition of our weighted H\"older spaces in \eqref{Eq.G^k_def}, \eqref{Eq.G^k_norm} and \eqref{Eq.G^0_0}.

\begin{lemma}\label{Lem.seminorm}
	For each $k\in[0,+\infty)$ and $\alpha\in(0,1)$, there exists a constant $C=C(k,\alpha)>0$ such that
	\begin{equation}\label{Eq.seminorm}
		[f]_k\leq C\left\|\langle\th\rangle^{k+2\al-1}\th^{1-\al}f'\right\|_{L^\infty},\quad\forall\ f\in C^1(\R_+;\C).
	\end{equation}
\end{lemma}
\begin{proof}
	For simplicity, we assume without loss of generality that $\left\|\langle\th\rangle^{k+2\al-1}\th^{1-\al}f'\right\|_{L^\infty}=1$. Let $0<\th_2<\th_1<\th_2+1$. Then
	\begin{align*}
		\langle\th_1\rangle^{k+\al}\frac{|f(\th_1)-f(\th_2)|}{|\th_1-\th_2|^\al}&\leq \frac{\langle\th_1\rangle^{k+\al}}{|\th_1-\th_2|^\al}\int_{\th_2}^{\th_1}\left|f'(\wt\th)\right|\,\mathrm d\wt\th\leq \frac{\langle\th_1\rangle^{k+\al}}{|\th_1-\th_2|^\al}\int_{\th_2}^{\th_1}\langle\wt\th\rangle^{1-2\al-k}\wt\th^{\al-1}\,\mathrm d\wt\th\\
		&\lesssim \frac{\langle\th_1\rangle^{k+\al}\langle\th_1\rangle^{1-2\al-k}}{|\th_1-\th_2|^\al}\int_{\th_2}^{\th_1}\wt\th^{\al-1}\,\mathrm d\wt\th\lesssim\frac{\langle\th_1\rangle^{1-\al}}{|\th_1-\th_2|^\al}\left|\th_1^\al-\th_2^\al\right|.
	\end{align*}
	Let $\phi_0(t):=(t^\al-1)/(t-1)^\al$ for $t>1$. Then $\lim_{t\to1+}\phi_0(t)=0$ and $\lim_{t\to+\infty}\phi_0(t)=1$, hence $0\leq\phi_0\in L^\infty((1,+\infty))$. 
	If $\th_1\in(0,2]$, then 
	\[\frac{\langle\th_1\rangle^{1-\al}}{|\th_1-\th_2|^\al}\left|\th_1^\al-\th_2^\al\right|\leq \langle2\rangle^{1-\al}\phi_0(\th_1/\th_2)\lesssim 1.\]
	If $\th_1\in(2,+\infty)$, then there exists $\th_3\in(\th_2,\th_1)$ such that 
	\[\frac{\langle\th_1\rangle^{1-\al}}{|\th_1-\th_2|^\al}\left|\th_1^\al-\th_2^\al\right|=\frac{\langle\th_1\rangle^{1-\al}}{|\th_1-\th_2|^\al}\alpha\th_3^{\al-1}|\th_1-\th_2|;\]
	since $\th_1>2$, we have $\th_2>1$ and $\th_1<\th_2+1<2\th_2$, hence $\th_1\sim\th_2$ and $\th_3\sim\th_1$, thus (also using $|\th_1-\th_2|<1$ and $\th_1>2$)
	\[\frac{\langle\th_1\rangle^{1-\al}}{|\th_1-\th_2|^\al}\left|\th_1^\al-\th_2^\al\right|\lesssim \frac{\langle\th_1\rangle^{1-\al}}{\th_1^{1-\al}}|\th_1-\th_2|^{1-\al}\lesssim \frac{\langle\th_1\rangle^{1-\al}}{\th_1^{1-\al}}\lesssim 1.\]
	Then \eqref{Eq.seminorm} follows from the definition of the semi-norm $[\cdot]_k$ in \eqref{Eq.G^k_norm}.
\end{proof}

\begin{lemma}\label{Lem.cG^0_0}
	Let $f\in C^1(\R_+;\R)\cap L^\infty(\R_+)$. If the limit $L:=\lim_{\th\to0+}\th f'(\th)\in\R$ exists, then $L=0$.
\end{lemma}
\begin{proof}
	Suppose that $L\neq 0$. Without loss of generality we assume $L>0$, then there exists $\th_0>0$ such that $\th f'(\th)>L/2$ for all $\th\in(0, 2\th_0)$. Hence, for $\th\in(0,\th_0)$, we have
	\begin{align*}
		f(\th)=f(\th_0)-\int_\th^{\th_0}f'(\wt\th)\,\mathrm d\wt\th\leq f(\th_0)-\int_\th^{\th_0}\frac L{2\wt\th}\,\mathrm d\wt\th=f(\th_0)-\frac L2\ln(\th_0/\th).
	\end{align*}
	Hence, $\lim_{\th\to0+} f(\th)=-\infty$, which contradicts with our assumption $f\in L^\infty(\R_+)$.
\end{proof}

A direct corollary of Lemma \ref{Lem.cG^0_0} is the following

\begin{corollary}\label{Cor.cG^0_0}
	If $(r,\g)\in X$, then $\th(\th^\mu r)'\in\cG^0_0$ and $\th(\th^{2\mu-1}\g)'\in\cG^1_0$.
\end{corollary}
\begin{proof}
	Since $(r,\g)\in X$, we have $\th^{\mu}r\in\cG^1$ and $\th^{\mu+1}r'\in\cG^0$, hence,
	\[\th(\th^\mu r)'=\th^{\mu+1}r'+\mu\th^{\mu}r\in\cG^0.\]
	Thus, $\th^{\mu}r\in L^\infty(\R_+)$ and the limit $\lim_{\th\to0+}\th(\th^\mu r)'(\th)$ exists. Then by Lemma \ref{Lem.cG^0_0}, we know that $\lim_{\th\to0+}\th(\th^\mu r)'(\th)=0$, hence $\th(\th^\mu r)'\in\cG^0_0$ by the definition of $\cG^0_0$ in \eqref{Eq.G^0_0}. Similarly, one shows that $\th(\th^{2\mu-1}\g)'\in\cG^1_0$.
\end{proof}

\begin{lemma}[Algebra property]\label{Lem.algebra}
	Let $k_2\geq k_1\geq 0$ and $\alpha\in (0,1)$. If $f_1\in \mathscr G^{k_1}$ and $f_2\in \mathscr G^{k_2}$, then $f_1f_2\in\mathscr G^{k_1}$ and
	\[\|f_1f_2\|_{k_1}\leq \|f_1\|_{k_1}\|f_2\|_{k_2}.\]
\end{lemma}
\begin{proof}
	It is obvious that $\left\|\langle\th\rangle^\al f_1f_2\right\|_{L^\infty}\leq \left\|\langle\th\rangle^\al f_1\right\|_{L^\infty}\left\|f_2\right\|_{L^\infty}\leq \left\|\langle\th\rangle^\al f_1\right\|_{L^\infty}\left\|\langle\th\rangle^\al f_2\right\|_{L^\infty}.$ Let $0<\th_2<\th_1<\th_2+1$, then
	\begin{align*}
		&\langle\th_1\rangle^{k_1+\al}\frac{\left|f_1(\th_1)f_2(\th_1)-f_1(\th_2)f_2(\th_2)\right|}{|\th_1-\th_2|^\al}\\
		\leq\ & \langle\th_1\rangle^{k_1+\al}\frac{\left|f_1(\th_1)-f_1(\th_2)\right|}{|\th_1-\th_2|^\al}|f_2(\th_1)|+|f_1(\th_2)|\langle\th_1\rangle^{k_1+\al}\frac{\left|f_2(\th_1)-f_2(\th_2)\right|}{|\th_1-\th_2|^\al}\\
		\leq\ & \langle\th_1\rangle^{k_1+\al}\frac{\left|f_1(\th_1)-f_1(\th_2)\right|}{|\th_1-\th_2|^\al}|f_2(\th_1)|+|f_1(\th_2)|\langle\th_1\rangle^{k_2+\al}\frac{\left|f_2(\th_1)-f_2(\th_2)\right|}{|\th_1-\th_2|^\al},
	\end{align*}
	because of $k_1\leq k_2$. Hence, $[f_1f_2]_{k_1}\leq [f_1]_{k_1}\|f_2\|_{L^\infty}+\|f_1\|_{L^\infty}[f_2]_{k_2}$. Therefore,
	\begin{align*}
		\|f_1f_2\|_{k_1}&=\left\|\langle\th\rangle^\al f_1f_2\right\|_{L^\infty}+[f_1f_2]_{k_1}\\
		&\leq \left\|\langle\th\rangle^\al f_1\right\|_{L^\infty}\left\|\langle\th\rangle^\al f_2\right\|_{L^\infty}+[f_1]_{k_1}\left\|\langle\th\rangle^\al f_2\right\|_{L^\infty}+\left\|\langle\th\rangle^\al f_1\right\|_{L^\infty}[f_2]_{k_2}\\
		&\leq \left(\left\|\langle\th\rangle^\al f_1\right\|_{L^\infty}+[f_1]_{k_1}\right)\left(\left\|\langle\th\rangle^\al f_2\right\|_{L^\infty}+[f_2]_{k_2}\right)=\|f_1\|_{k_1}\|f_2\|_{k_2}.
	\end{align*}
	This completes the proof.
\end{proof}

Given $k\in[0,+\infty)$ and $\alpha\in(0,1)$, noting that $\mathscr G^k\cap\C=\{0\}$, where $\C$ denotes the space consisting of all complex-valued constant functions on $\R_+$, we define the Banach space
\begin{equation}
	\mathscr G^k_*:=\mathscr G^k\oplus \C,
\end{equation}
with the norm $\|\cdot\|_{*,k}$ given by
\begin{equation*}
	\|f\|_{*,k}:=\|\wt f\|_k+|c|,\quad\text{where }f=\wt f+c\in\mathscr G^k_*\text{ with }\wt f\in \mathscr G^k, c\in\C.
\end{equation*}
Given $f_1, f_2\in \mathscr G^k_*$, we write $f_1=\wt{f_1}+c_1, f_2=\wt{f_2}+c_2$ with $\wt{f_1}, \wt{f_2}\in\mathscr G^k$ and $c_1, c_2\in\C$. Then by Lemma \ref{Lem.algebra}, we have
\begin{align*}
	\|f_1f_2\|_{*,k}&=\left\|\wt{f_1}\wt{f_2}+c_1\wt{f_2}+c_2\wt{f_1}\right\|_k+|c_1c_2|\\
	&\leq \left\|\wt{f_1}\right\|_k\left\|\wt{f_2}\right\|_k+|c_1|\left\|\wt{f_2}\right\|_k+|c_2|\left\|\wt{f_1}\right\|_k+|c_1||c_2|\\
	&=\left(\left\|\wt{f_1}\right\|_k+|c_1|\right)\left(\left\|\wt{f_2}\right\|_k+|c_2|\right)=\|f_1\|_{*,k}\|f_2\|_{*,k}.
\end{align*}
Hence, $\mathscr G^k_*$ is a unital Banach algebra. Similar to \eqref{Eq.G^k_def}, we define $$\cG^k_*:=\{f\in \mathscr G^k_*: f \text{ is real-valued}\}=\cG^k\oplus\R.$$ Then $\cG^k_*$ is also a unital Banach algebra.

\section{Invertibility of $\mathcal N$}\label{Sec.N}

In this section, we prove Proposition \ref{Prop.N}. As illustrated in Subsection \ref{Subsec.sketch_N}, the first step is to prove the $C^\infty$ regularity of $\cN$ (see Subsection \ref{Subsec.nonlinear_smooth}), using the Banach algebraic properties established in the previous section. The second step is to prove the invertibility of $\cL=\cN'[0]:X\to Y$, based on the exact solvability of the linear equation $\cL[r,\g]=(f,g)$, see Subsection \ref{Subsec.invert_linear}. Finally, the invertibility of $\cN$ follows from the invertibility of $\cL=\cN'[0]:X\to Y$ and an application of Banach's fixed point theorem, see Subsection \ref{Subsec.invert_nonlinear}.

\subsection{Regularity of nonlinear map $\mathcal N$}\label{Subsec.nonlinear_smooth}

In this subsection, we show the first item of Proposition \ref{Prop.N}.

Recall the definition of $\mathcal N$ in \eqref{Eq.N} and footnote \ref{Footnote.1}. Using \eqref{special_sol}, we rewrite $\mathcal N$ in the form of $\mathcal N[r,\g]=(\cN_1[r,\g], \cN_2[r,\g])$, where
\begin{align*}
	\cN_1[r,\g](\th):&=
		\mu\left(\th^{-\mu}+r(\th)\right)^2\\
		&\quad+(1-2\mu)\frac{-2\pi\th^{1-2\mu}+\g(\th)}{2\pi(2\mu-1)\th^{-2\mu}+\g'(\th)}\left(\th^{-\mu}+r(\th)\right)\left(-\mu\th^{-\mu-1}+r'(\th)\right),\\
	\cN_2[r,\g](\th):&=(2\mu-1)\frac{-2\pi\th^{1-2\mu}+\g(\th)}{2\pi(2\mu-1)\th^{-2\mu}+\g'(\th)}\left(\th^{-\mu}+r(\th)\right)^2-\frac{-2\pi\th^{1-2\mu}+\g(\th)}{2\pi}.
\end{align*}

\begin{lemma}\label{Lem.N_1}
	There exists $\varepsilon_0\in(0,1)$ such that
	\[(r,\g)\mapsto \th^{2\mu}\mathcal N_1[r,\g]\in C^\infty(B_X(\varepsilon_0); \cG_0^0).\]
\end{lemma}
\begin{proof}
	We compute that
	\begin{align*}
		\th^{2\mu}\cN_1[r,\g](\th)&=\mu\left(1+\th^\mu r(\th)\right)^2\\
		&\quad+(1-2\mu)\frac{-2\pi+\th^{2\mu-1}\g(\th)}{2\pi(2\mu-1)+\th^{2\mu}\g'(\th)}\left(1+\th^\mu r(\th)\right)\left(-\mu+\th^{\mu+1}r'(\th)\right).
	\end{align*}
	Due to the trivial embedding $\cG^1\hookrightarrow\cG^0$ and  the fact that $\cG^0_*, \cG^1_*$ are Banach algebras, the following maps are all smooth:
	\begin{align}
		(r,\g)\in X\mapsto r\in X_r\mapsto \th^\mu r\in \cG^1\mapsto (1+\th^\mu r)^2\in \cG^1_*,\nonumber\\
		(r,\g)\in X\mapsto r\in X_r\mapsto (\th^\mu r, \th^{\mu+1}r')\in \cG^1\times\cG^0\mapsto \left(1+\th^\mu r(\th)\right)\left(-\mu+\th^{\mu+1}r'(\th)\right)\in\cG^0_*,\nonumber\\
		(r,\g)\in B_X(\varepsilon_0)\mapsto \g\in B_{X_\g}(\varepsilon_0)\mapsto \th^{2\mu}\g'\in B_{\cG^1}(\varepsilon_0)\mapsto \frac1{2\pi(2\mu-1)+\th^{2\mu}\g'}\in\cG^1_*,\label{Eq.smooth_reci}\\
		(r,\g)\in X\mapsto \g\in X_\g\mapsto \th^{2\mu-1}\g\in \cG^1\mapsto -2\pi+\th^{2\mu-1}\g\in \cG^1_*,\nonumber
	\end{align}
	for some small $\varepsilon_0\in(0,1)$, where in \eqref{Eq.smooth_reci} we have used $\mu>1/2$, and we also used the fact that Taylor expansion of $x\mapsto x^{-1}$ around $x=2\pi(2\mu-1)\neq 0$ gives a smooth map from $B_{\cG^1_*}(2\pi(2\mu-1), \varepsilon_0)$ to $\cG^1_*$ for some small $\varepsilon_0\in(0,1)$. As a consequence, 
	\[(r,\g)\mapsto \th^{2\mu}\mathcal N_1[r,\g]\in C^\infty(B_X(\varepsilon_0); \cG^0_*).\]
	
	Moreover, given $(r,\g)\in B_X(\varepsilon_0)$, we have $\langle\th\rangle^\al(\th^\mu r, \th^{\mu+1}r', \th^{2\mu-1}\g, \th^{2\mu}\g')\in L^\infty(\R_+)$, hence,
	\[\lim_{\th\to+\infty}\th^\mu r(\th)=\lim_{\th\to+\infty}\th^{\mu+1}r'(\th)=\lim_{\th\to+\infty}\th^{2\mu-1}\g(\th)=\lim_{\th\to+\infty}\th^{2\mu}\g'(\th)=0.\]
	Thus,
	\[\lim_{\th\to+\infty}\th^{2\mu}\cN_1[r,\g](\th)=\mu+(1-2\mu)\frac{-2\pi}{2\pi(2\mu-1)}(-\mu)=0.\]
	Therefore, $\th^{2\mu}\mathcal N_1[r,\g]\in \cG^0$ and
	\[(r,\g)\mapsto \th^{2\mu}\mathcal N_1[r,\g]\in C^\infty(B_X(\varepsilon_0); \cG^0).\]
	
	Finally, we check that $\th^{2\mu}\mathcal N_1[r,\g]$ lies in the smaller set $\cG^0_0$. Note that
	\begin{align}
		&\th^{2\mu}\cN_1[r,\g](\th)=\left(1+\th^\mu r(\th)\right)\times\nonumber\\
		&\qquad \left[\mu\left(1+\th^\mu r(\th)\right)+(1-2\mu)\frac{-2\pi+\th^{2\mu-1}\g(\th)}{2\pi(2\mu-1)+\th^{2\mu}\g'(\th)}\left(-\mu+\th^{\mu+1}r'(\th)\right)\right]\nonumber\\
		=&\left(1+\th^\mu r(\th)\right)\times\nonumber\\
		&\qquad\left[\th(\th^\mu r)'+\mu-\th^{\mu+1}r'(\th)+(1-2\mu)\frac{-2\pi+\th^{2\mu-1}\g(\th)}{2\pi(2\mu-1)+\th^{2\mu}\g'(\th)}\left(-\mu+\th^{\mu+1}r'(\th)\right)\right]\nonumber\\
		=&\left(1+\th^\mu r(\th)\right)\left[\th(\th^\mu r)'(\th)+\left(1+(2\mu-1)\frac{-2\pi+\th^{2\mu-1}\g(\th)}{2\pi(2\mu-1)+\th^{2\mu}\g'(\th)}\right)\left(\mu-\th^{\mu+1}r'(\th)\right)\right]\nonumber\\
		=&\left(1+\th^\mu r(\th)\right)\left[\th(\th^\mu r)'(\th)+\frac{\th(\th^{2\mu-1}\g)'(\th)}{2\pi(2\mu-1)+\th^{2\mu}\g'(\th)}\left(\mu-\th^{\mu+1}r'(\th)\right)\right].\label{Eq.N_1}
	\end{align}
	It follows from Corollary \ref{Cor.cG^0_0} that $\lim_{\th\to0+}\th(\th^\mu r)'(\th)=0$ and $\lim_{\th\to0+}\th(\th^{2\mu-1}\g)'(\th)=0$. Then $\lim_{\th\to0+}\th^{2\mu}\cN_1[r,\g](\th)=0$, and thus $\th^{2\mu}\mathcal N_1[r,\g]\in\cG^0_0$.
\end{proof}

\begin{lemma}\label{Lem.N_2}
	There exists $\varepsilon_0\in(0,1)$ such that
	\[(r,\g)\mapsto \th^{2\mu-1}\cN_2[r,\g]\in C^\infty(B_X(\varepsilon_0); \cG^1).\]
\end{lemma}
\begin{proof}
	We compute that
	\begin{equation}\label{Eq.N_2}
		\th^{2\mu-1}\cN_2[r,\g](\th)=(2\mu-1)\frac{-2\pi+\th^{2\mu-1}\g(\th)}{2\pi(2\mu-1)+\th^{2\mu}\g'(\th)}\left(1+\th^\mu r(\th)\right)^2-\frac{-2\pi+\th^{2\mu-1}\g(\th)}{2\pi}.
	\end{equation}
	The remaining proof is very similar to the proof of Lemma \ref{Lem.N_1}: use the theory of Banach algebras to show that $\th^{2\mu-1}\cN_2\in C^\infty(B_X(\varepsilon_0);\cG^1_*)$, then consider the limit $\th\to+\infty$ to obtain $\th^{2\mu-1}\cN_2[r,\g]\in \cG^1$. We omit the details here.
\end{proof}

By Lemma \ref{Lem.N_1}, Lemma \ref{Lem.N_2} and the definition of $Y$ in \eqref{Eq.Y}, we obtain
\begin{equation}\label{Eq.N_smooth}
	\cN\in C^\infty(B_X(\varepsilon_0); Y).
\end{equation}
This proves the first item of Proposition \ref{Prop.N}.

\subsection{Invertibility of the linearized operator}\label{Subsec.invert_linear}
We denote the linearization of $\cN$ near $(0,0)$ by $\mathcal L$, which is to say, $\cL$ is the Fr\'echet derivative of $\cN$ at $(0,0)\in X$. Hence, $\cL$ is a bounded linear operator from $X$ to $Y$.

Our main proposition of this subsection is the following.
\begin{proposition}\label{Prop.invert_linear}
    The bounded linear operator $\cL: X\to Y$ is bijective, and the inversion $\cL^{-1}: Y\to X$ is a bounded linear operator, with
    \begin{equation*}
        \norm{\cL^{-1}[f,g]}_X\leq C\norm{(f,g)}_Y,\quad \forall\ (f,g)\in Y,
    \end{equation*}
    for some constant $C=C(\al,\mu)>0$.
\end{proposition}

Proposition \ref{Prop.invert_linear} follows directly from the following Lemma \ref{Lem.M_bound} and Lemma \ref{Lem.inversion_L}.
Let us first compute $\cL$.

\begin{lemma}\label{Lem.L-expression}
	For any $(r,\g)\in X$, there holds
	\begin{equation}\label{Eq.L_expression}
		\mathcal L[r,\g](\th)=\begin{pmatrix}
			\mu\th^{-\mu}r(\th)+\th^{1-\mu}r'(\th)+\frac\mu{2\pi}\th^{-1}\g(\th)+\frac\mu{2\pi(2\mu-1)}\g'(\th)\\
			-2\th^{1-\mu}r(\th)+\frac{\th\g'(\th)}{2\pi(2\mu-1)}
		\end{pmatrix}.
	\end{equation}
\end{lemma}
\begin{proof}
	By definition and $\cN[0,0]=0$, we have
	\begin{equation}\label{Eq.L_Gautex}
		\cL[r,\g](\th)=\lim_{t\to0}\frac{\cN[tr,t\g](\th)}{t}=\frac{\mathrm d}{\mathrm dt}\Big|_{t=0}\cN[tr,t\g](\th).
	\end{equation}
	We write $\cL=(\cL_1, \cL_2)$. By \eqref{Eq.N_1} and \eqref{Eq.L_Gautex}, we have
	\begin{align}
		\th^{2\mu}\cL_1[r,\g](\th)=\frac{\mathrm d}{\mathrm dt}\Big|_{t=0}\th^{2\mu}\cN_1[tr,t\g](\th)=\th(\th^\mu r)'(\th)+\frac{\mu}{2\pi(2\mu-1)}\th(\th^{2\mu-1}\g)'(\th).\label{Eq.L_1}
	\end{align}
	By \eqref{Eq.N_2} and \eqref{Eq.L_Gautex}, we have
	\begin{align}
		&\th^{2\mu-1}\cL_2[r,\g](\th)=\frac{\mathrm d}{\mathrm dt}\Big|_{t=0}\th^{2\mu-1}\cN_2[tr,t\g](\th)\nonumber\\
		=\ &(2\mu-1)\frac{\th^{2\mu-1}\g(\th)}{2\pi(2\mu-1)}+(2\mu-1)(-2\pi)\frac{-\th^{2\mu}\g'(\th)}{\left(2\pi(2\mu-1)\right)^2}\nonumber\\
		&\qquad\qquad+(2\mu-1)\frac{-2\pi}{2\pi(2\mu-1)}2\th^\mu r(\th)-\frac1{2\pi}\th^{2\mu-1}\g(\th)\nonumber\\
		=\ &-2\th^\mu r(\th)+\frac1{2\pi(2\mu-1)}\th^{2\mu}\g'(\th).\label{Eq.L_2}
	\end{align}
	Obviously, \eqref{Eq.L_expression} follows from \eqref{Eq.L_1} and \eqref{Eq.L_2}.
\end{proof}

Recall that our aim in this subsection is to show the invertibility of $\cL: X\to Y$. Now we define a linear operator $\cM:Y\to X$, which turns out to be the inversion of $\cL$.

Given $(f,g)\in Y$, we define
\begin{align}
	\cM_2[f,g](\th)&:=2\pi(2\mu-1)\th^{-2\mu}\left[-2\int_0^\th\int_{\wt\th}^\infty \xi^{2\mu-1}f(\xi)\,\mathrm d\xi\,\mathrm d\wt\th+\int_0^\th\wt\th^{2\mu-1}g(\wt\th)\,\mathrm d\wt\th\right],\label{Eq.M_2}\\
	\cM_1[f,g]&(\th):=-\frac\mu{2\pi(2\mu-1)}\th^{\mu-1}\cM_2[f,g](\th)-\th^{-\mu}\int_\th^\infty\wt\th^{2\mu-1}f(\wt\th)\,\mathrm d\wt\th,\label{Eq.M_1}\\
	&\cM[f,g](\th):=\left(\cM_1[f,g](\th), \cM_2[f,g](\th)\right),\label{Eq.M_def}
\end{align}
for all $\th\in\R_+$. The rest of this subsection is devoted to showing that
\begin{equation}
	\cM: Y\to X \text{ is bounded and } \cM\circ\cL=\mathrm{Id}_X,\ \ \cL\circ\cM=\mathrm{Id}_Y.
\end{equation}
At this stage, we have not checked that the integrals in \eqref{Eq.M_2} and \eqref{Eq.M_1} are absolutely convergent. We will deal with this issue in the process of showing the boundedness of $\cM$.

We define two linear operators $\cA$ and $\cB$ by
	\begin{equation}\label{Eq.AB_operator}
		\cA[f](\th):=\int_\th^\infty\frac{f(\wt\th)}{\wt\th}\,\mathrm d\wt\th,\quad \cB[g](\th):=\frac1\th\int_0^\th g(\wt\th)\,\mathrm d\wt\th,\quad\forall\ \th\in\R_+,
	\end{equation}
	for all $f\in \cG^0_0$ and $g\in\cG^1$.  Then we can write
\begin{align}
	\frac{\cM_2[f,g]}{2\pi(2\mu-1)}&=-2\th^{1-2\mu}(\cB\circ\cA)\left[\th^{2\mu}f\right]+\th^{1-2\mu}\cB\left[\th^{2\mu-1}g\right],\label{Eq.M_2_operator}\\
	\cM_1[f,g]&=-\frac{\mu}{2\pi(2\mu-1)}\th^{\mu-1}\cM_2[f,g]-\th^{-\mu}\cA\left[\th^{2\mu}f\right].\label{Eq.M_1_operator}
\end{align}

We have the following bounds for  linear operators $\cA$ and $\cB$.
\begin{lemma}\label{Lem.AB_bound}
	 There exists a constant $C=C(\al)>0$ such that
	\begin{equation*}
		\|\cA[f]\|_1\leq C\|f\|_0,\quad \|\cB[g]\|_1\leq C\|g\|_1,\quad \forall\ f\in\cG_0^0, \ \forall\ g\in\cG^1.
	\end{equation*}
\end{lemma}
\begin{proof}
	By the definition of $\cG_0^0$, we have
	\begin{equation}\label{Eq.3.17}
		|f(\th)|\leq \|f\|_0\frac{\min\{\th^\al, 1\}}{\langle\th\rangle^\al}\lesssim\|f\|_0\frac{\th^\al}{\langle\th\rangle^{2\al}},\quad\forall\ \th\in\R_+,\ \forall\ f\in\cG_0^0.
	\end{equation}
	Hence,
	\begin{equation}\label{Eq.Af_infty_1}
		\|\cA[f]\|_{L^\infty}\leq\|\th^{-1}f\|_{L^1}\lesssim\|f\|_0,\quad\forall\ f\in\cG_0^0.
	\end{equation}
	Note that \eqref{Eq.Af_infty_1} implies the operator $\cA$ is well-defined.  If $\th>1$, we get by \eqref{Eq.3.17} that
	\begin{equation}\label{Eq.Af_infty_2}
		|\cA[f](\th)|\leq \|f\|_0\int_\th^\infty \wt\th^{-1}\langle\wt\th\rangle^{-\al}\,\mathrm d\wt\th\lesssim \th^{-\al}\|f\|_0,\quad \forall\ \th>1,\ \forall\ f\in\cG^0_0.
	\end{equation}
	Combining \eqref{Eq.Af_infty_1} and \eqref{Eq.Af_infty_2} gives that $\|\langle\th\rangle^\al\cA[f]\|_{L^\infty}\lesssim\|f\|_0$ for all $f\in\cG_0^0$. Moreover, by Lemma \ref{Lem.seminorm} and \eqref{Eq.3.17}, we have
	\begin{align*}
		\left[\cA[f]\right]_1\lesssim \left\|\langle\th\rangle^{2\al}\th^{1-\al}\left(\cA[f]\right)'\right\|_{L^\infty}\lesssim \left\|\langle\th\rangle^{2\al}\th^{1-\al}\th^{-1}f\right\|_{L^\infty}\lesssim\|f\|_0, \quad\forall\ f\in\cG_0^0.
	\end{align*}
	Therefore, there exists a constant $C=C(\al)>0$ such that
	\begin{equation*}
		\|\cA[f]\|_1\leq C\|f\|_0,\quad \forall\ f\in\cG_0^0.
	\end{equation*}
	
	For $g\in\cG^1$, we  have 
	\begin{equation}\label{Eq.Bg_infty_1}
		\left\|\cB[g]\right\|_{L^\infty}\leq \|g\|_{L^\infty}\leq \|\langle\th\rangle^\al g\|_{L^\infty}\leq \|g\|_1.
	\end{equation}
	Hence, the operator $\cB$ is well-defined. If $\th\ge1$, then
	\begin{equation}\label{Eq.Bg_infty_2}
		\left|\cB[g](\th)\right|\leq  \frac{\|\langle\th\rangle^\al g\|_{L^\infty}}\th\int_0^\th\wt\th^{-\al}\,\mathrm d\wt\th=\frac{\|\langle\th\rangle^\al g\|_{L^\infty}}{(1-\al)\th^\al}.
	\end{equation}
	Combining \eqref{Eq.Bg_infty_1} and \eqref{Eq.Bg_infty_2} gives that $\|\langle\th\rangle^\al \cB[g]\|_{L^\infty}\lesssim\|g\|_1$. On the other hand, for $g\in\cG^1$ and $\th\in\R_+$, we compute that
	\begin{align}
		\left(\cB[g]\right)'(\th)&=-\frac1{\th^2}\int_0^\th g(\wt\th)\,\mathrm d\wt\th+\frac{g(\th)}\th=-\frac1\th\cB[g](\th)+\frac{g(\th)}{\th}\label{Eq.Bg'1}\\
		&=-\frac1{\th^2}\int_0^\th \left(g(\wt\th)-g(0+)\right)\,\mathrm d\wt\th+\frac{g(\th)-g(0+)}\th,\label{Eq.Bg'2}
	\end{align}
	where $g(0+):=\lim_{\th\to0+}g(\th)$. By the definition of $\cG^1$, for $\th\in(0,1)$, there holds $$\frac{|g(\th)-g(0+)|}{\th^\al}\leq[g]_1\leq \|g\|_1,$$
	hence it follows from \eqref{Eq.Bg'2} that
	\begin{equation}\label{Eq.Bg'_infty_1}
		\left|\th^{1-\al}\left(\cB[g]\right)'(\th)\right|\leq \th^{-1-\al}\|g\|_1\int_0^\th\wt\th^\al\,\mathrm d\wt\th+\|g\|_1=\left(\frac1{1+\al}+1\right)\|g\|_1
	\end{equation}
	for all $g\in\cG^1$ and $\th\in(0,1)$. As for $\th\geq 1$, by \eqref{Eq.Bg'1} and \eqref{Eq.Bg_infty_2}, we have
	\begin{equation}\label{Eq.Bg'_infty_2}
		\left|\langle\th\rangle^{2\al}\th^{1-\al}\left(\cB[g]\right)'(\th)\right|\lesssim \th^{2\al}\th^{1-\al}\th^{-1}\left(\left|\cB[g](\th)\right|+|g(\th)|\right)\lesssim\|\langle\th\rangle^\al g\|_{L^\infty}\lesssim \|g\|_1.
	\end{equation}
	Finally, by Lemma \ref{Lem.seminorm}, \eqref{Eq.Bg'_infty_1} and \eqref{Eq.Bg'_infty_2}, we obtain
	\begin{equation*}
		\left[\cB[g]\right]_1\lesssim\left\|\langle\th\rangle^{2\al}\th^{1-\al}\left(\cB[g]\right)'\right\|_{L^\infty}\lesssim\|g\|_1,\quad \forall\ g\in\cG^1.
	\end{equation*}
	This completes the proof.
\end{proof}

\begin{lemma}\label{Lem.M_bound}
	There exists a constant $C=C(\al,\mu)>0$ such that
	\begin{equation*}
		\left\|\cM[f,g]\right\|_X\leq C\|(f,g)\|_Y,\quad\forall\ (f,g)\in Y.
	\end{equation*}
\end{lemma}
\begin{proof}
	Given $(f,g)\in Y$, we let $\g=\cM_2[f,g]$ and $r=\cM_1[f,g]$. It suffices to show that
	\begin{align}
		\|r\|_{X_r}=\left\|\th^\mu r\right\|_1+\left\|\th^{\mu+1}r'\right\|_0\leq C\|(f,g)\|_Y,\label{Eq.r_bound}\\
		\|\g\|_{X_\g}=\left\|\th^{2\mu-1}\g\right\|_1+\left\|\th^{2\mu}\g'\right\|_1\leq C\|(f,g)\|_Y,\label{Eq.g_bound}
	\end{align}
	for some constant $C=C(\al,\mu)>0$. We first prove \eqref{Eq.g_bound}. By \eqref{Eq.M_2_operator} and Lemma \ref{Lem.AB_bound}, 
	\[\left\|\th^{2\mu-1}\g\right\|_1\lesssim\left\|\cA\left[\th^{2\mu}f\right]\right\|_1+\norm{\th^{2\mu-1}g}_1\lesssim \norm{\th^{2\mu}f}_0+\norm{\th^{2\mu-1}g}_1\lesssim\|(f,g)\|_Y.\]
	On the other hand, it follows from \eqref{Eq.M_2_operator} and \eqref{Eq.Bg'1} that
	\begin{align*}
		\frac{\g'}{2\pi(2\mu-1)}&=-2(1-2\mu)\th^{-2\mu}(\cB\circ\cA)\left[\th^{2\mu}f\right]+(1-2\mu)\th^{-2\mu}\cB\left[\th^{2\mu-1}g\right]\\
		&\quad -2\th^{1-2\mu}\left(-\frac1\th(\cB\circ\cA)\left[\th^{2\mu}f\right]+\frac1\th\cA\left[\th^{2\mu}f\right]\right)\\
		&\quad+\th^{1-2\mu}\left(-\frac1\th \cB\left[\th^{2\mu-1}g\right]+\frac1\th\th^{2\mu-1}g\right)\\
		&=4\mu\th^{-2\mu}(\cB\circ\cA)\left[\th^{2\mu}f\right]-2\mu\th^{-2\mu}\cB\left[\th^{2\mu-1}g\right]-2\th^{-2\mu}\cA\left[\th^{2\mu}f\right]+\th^{-2\mu}\th^{2\mu-1}g.
	\end{align*}
	Thus, by Lemma \ref{Lem.AB_bound}, we have
	\begin{equation*}
		\norm{\th^{2\mu}\g'}_1\lesssim \norm{\th^{2\mu}f}_0+\norm{\th^{2\mu-1}g}_1\lesssim\norm{(f,g)}_Y.
	\end{equation*}
	This proves \eqref{Eq.g_bound}. Next we show \eqref{Eq.r_bound}. Using \eqref{Eq.M_1_operator} gives that
	\begin{equation}\label{Eq.r_g_relation}
		r=-\frac\mu{2\pi(2\mu-1)}\th^{\mu-1}\g-\th^{-\mu}\cA\left[\th^{2\mu}f\right].
	\end{equation}
	By \eqref{Eq.r_g_relation}, \eqref{Eq.g_bound} and Lemma \ref{Lem.AB_bound}, we have
	\[\norm{\th^\mu r}_1\lesssim\norm{\th^{2\mu-1}\g}_1+\norm{\cA\left[\th^{2\mu}f\right]}_1\lesssim\norm{(f,g)}_Y.\]
	Finally, it follows from \eqref{Eq.r_g_relation} and \eqref{Eq.AB_operator} that
	\begin{align*}
		r'=-\frac{\mu(\mu-1)}{2\pi(2\mu-1)}\th^{\mu-2}\g-\frac\mu{2\pi(2\mu-1)}\th^{\mu-1}\g'+\mu\th^{-\mu-1}\cA\left[\th^{2\mu}f\right]+\th^{-\mu-1}\th^{2\mu}f.
	\end{align*}
	Then by the embedding $\cG^1\hookrightarrow\cG^0$, \eqref{Eq.g_bound} and Lemma \ref{Lem.AB_bound}, we have
	\begin{align*}
		\norm{\th^{\mu+1}r'}_0\lesssim\norm{\th^{2\mu-1}\g}_0+\norm{\th^{2\mu}\g'}_0+\norm{\cA\left[\th^{2\mu}f\right]}_0+\norm{\th^{2\mu}f}_0\lesssim\norm{(f,g)}_Y.
	\end{align*}
	This completes the proof of \eqref{Eq.r_bound}.
\end{proof}

To conclude this subsection, we check that $\cM$ is the inversion of $\cL$.

\begin{lemma}\label{Lem.inversion_L}
	We have $\cM\circ\cL=\mathrm{Id}_X$ and $ \cL\circ\cM=\mathrm{Id}_Y$, where $\mathrm{Id}_X$ and $\mathrm{Id}_Y$ denote the identity maps on $X$ and $Y$, respectively.
\end{lemma}
\begin{proof}
	Let $(r,\g)\in X$ and $(f,g)=\cL[r,\g]\in Y$. Then by \eqref{Eq.L_1} and \eqref{Eq.L_2}, we have
	\begin{align}
		\th^{2\mu}f=\th\left(\th^\mu r\right)'+\frac{\mu}{2\pi(2\mu-1)}\th\left(\th^{2\mu-1}\g\right)',\quad 
		\th^{2\mu-1}g=-2\th^\mu r+\frac1{2\pi(2\mu-1)}\th^{2\mu}\g'.
	\end{align}
	Hence,
	\begin{equation}\label{Eq.Af}
		\cA\left[\th^{2\mu}f\right]=-\th^\mu r-\frac{\mu}{2\pi(2\mu-1)}\th^{2\mu-1}\g,
	\end{equation}
	and then \eqref{Eq.M_2_operator} implies that
	\begin{align*}
		\frac{\th^{2\mu-1}\cM_2[f,g]}{2\pi(2\mu-1)}&=-2\left(\cB\circ\cA\right)\left[\th^{2\mu}f\right]+\cB\left[\th^{2\mu-1}g\right]\\
		&=-2\cB\left[-\th^\mu r-\frac{\mu}{2\pi(2\mu-1)}\th^{2\mu-1}\g\right]+\cB\left[-2\th^\mu r+\frac1{2\pi(2\mu-1)}\th^{2\mu}\g'\right]\\
		&=\frac1{2\pi(2\mu-1)}\cB\left[2\mu\th^{2\mu-1}\g+\th^{2\mu}\g'\right]=\frac1{2\pi(2\mu-1)}\cB\left[\left(\th^{2\mu}\g\right)'\right]\\
		&=\frac1{2\pi(2\mu-1)}\frac1\th\th^{2\mu}\g=\frac{\th^{2\mu-1}\g}{2\pi(2\mu-1)},
	\end{align*}
	where in the last line we have used the definition of $\cB$ and the fact that $\lim_{\th\to0+}\th^{2\mu}\g(\th)=0$, since $\th^{2\mu-1}\g\in L^\infty(\R_+)$ (recalling the definition of $X$). So, $\cM_2[f,g]=\g$. As for $\cM_1[f,g]$, by \eqref{Eq.M_1_operator}, $\cM_2[f,g]=\g$ and \eqref{Eq.Af}, we have
	\begin{align*}
		\cM_1[f,g]=-\frac{\mu}{2\pi(2\mu-1)}\th^{\mu-1}\g-\th^{-\mu}\left(-\th^\mu r-\frac{\mu}{2\pi(2\mu-1)}\th^{2\mu-1}\g\right)=r.
	\end{align*}
	Hence, $\cM[f,g]=(r,\g)$, i.e., $\left(\cM\circ\cL\right)[r,\g]=(r,\g)$ for all $(r,\g)\in X$. 
	
	Next, we prove $\cL\circ\cM=\mathrm{Id}_Y$. It is observed from the definition \eqref{Eq.AB_operator} that both $\cA$ and $\cB$ are injective linear operators. This implies the injectivity of $\cM$. Indeed, assume that $\cM[f,g]=0$ for some $(f,g)\in Y$, then by \eqref{Eq.M_1_operator} we know that $\cA\left[\th^{2\mu}f\right]=0$, hence $f=0$, then by \eqref{Eq.M_2_operator} we have $\cB\left[\th^{2\mu-1}g\right]=0$, thus $g=0$. Now, with $\cM\circ\cL=\mathrm{Id}_X$ and the injectivity of $\cM$ at hand, we can deduce that $\cL\circ\cM=\mathrm{Id}_Y$. Indeed, let $(f,g)\in Y$ and we denote $(r,\g):=\cM[f,g]$, $(\wt f, \wt g):=\cL[r,\g]=\left(\cL\circ\cM\right)[f,g]$, then by $\cM\circ\cL=\mathrm{Id}_X$, we have $$\cM\left[\wt f, \wt g\right]=\left(\left(\cM\circ\cL\right)\circ\cM\right)[f,g]=\cM[f,g].$$
	Hence, by the injectivity of $\cM$, we obtain $(\wt f, \wt g)=(f,g)$. 
\end{proof}

\subsection{Invertibility of nonlinear map $\cN$}\label{Subsec.invert_nonlinear}
This subsection is devoted to the proof of the last two items of Proposition \ref{Prop.N}, using Subsection \ref{Subsec.nonlinear_smooth}, Proposition \ref{Prop.invert_linear} and Banach's fixed point theorem. Our main goal in this subsection is to establish the following abstract functional analysis proposition, from which we obtain the last two items of Proposition \ref{Prop.N}.

\begin{proposition}
	Let $X$ and $Y$ be two Banach spaces and let $\cN: B_X(\varepsilon_0)\to Y$ be a $C^2$ map for some $\varepsilon_0>0$. We denote $\cL=\cN'[0_X]$ the Fr\'echet derivative of $\cN$ at $0_X$. Assume that $\cL: X\to Y$ is a bijection, then there exist $\varepsilon_0'>0$ and $\delta_0>0$ such that $\cN: \overline{B_X[\varepsilon_0'/2]}\to Y$ is injective and the inversion $\cN^{-1}: B_Y(\delta_0)\to \overline{B_X(\varepsilon_0'/2)}$ satisfies
	\begin{align}
		\norm{\cN^{-1}y}_X&\leq C^*\|y\|_Y,\quad \forall\ y\in B_Y(\delta_0),\label{Eq.N^-1_bound}\\
		\norm{\cN^{-1}y_1-\cN^{-1}y_2}_X&\leq C^*\|y_1-y_2\|_Y,\quad \forall\ y_1, y_2\in B_Y(\delta_0),\label{Eq.N^-1_Lip}
	\end{align}
	for some constant $C^*>0$.
\end{proposition}

\begin{proof}
	We may assume without loss of generality that $\varepsilon_0>0$ is small enough such that $\mathcal N'': B_X(\varepsilon_0/2)\to L(X\times X, Y)$ is bounded, where we recall that $L(X\times X, Y)$ denotes the set of all bounded linear operators from $X\times X$ to $Y$. Let $\cN_1:=\cN-\cL$, then for all $x\in B_X(\varepsilon_0/2)$, we have
	\begin{align*}
		\cN_1x=\cN x-\cN'[0_X][x]=\left(\int_0^1(\cN'[tx]-\cN'[0_X])\,\mathrm dt\right)[x]=\int_0^1\,\mathrm dt\int_0^t\cN''[sx][x,x]\,\mathrm ds,
	\end{align*}
	hence, using the boundedness of $\cN''$ on $B_X(\varepsilon_0/2)$, we obtain
	\begin{align}\label{Eq.N_1bound}
		\norm{\cN_1x}_Y\leq 2C_1\int_0^1t\|x\|_X\,\mathrm dt\|x\|_X=C_1\|x\|_X^2,\quad \forall\ x\in B_X(\varepsilon_0/2),
	\end{align}
	for some constant $C_1>1$. By adjusting $C_1>1$ to be larger if necessary, we can assume that $\|\cL^{-1}\|\leq C_1$ and $C_1>2\varepsilon_0^{-1/2}$. Moreover, note that $\cN_1'[0_X]=0$, hence the boundedness of $\cN_1''$ on $B_X(\varepsilon_0/2)$ implies that $\norm{\cN_1'[x]}_{L(X,Y)}\lesssim\|x\|_X$ for $x\in B_X(\varepsilon_0/2)$. Then for $x_1, x_2\in B_X(\varepsilon_0/2)$, we have
    \begin{align}
        \norm{\cN_1x_1-\cN_1x_2}_Y&\leq\norm{\int_0^1\cN_1'[tx_1+(1-t)x_2][x_1-x_2]\,\mathrm dt}_Y\nonumber\\
        &\leq \frac{C_1}2\left(\|x_1\|_X+\|x_2\|_X\right)\|x_1-x_2\|_X,\label{Eq.N_1-N_2est}
    \end{align}
	by adjusting $C_1>1$ to be larger if necessary.
	Let
	\[\delta_0:=\frac1{4C_1^3}>0,\quad \varepsilon_0':=\frac1{C_1^2}\in\left(0, \varepsilon_0/2\right).\]
	We are going to prove that
	\begin{equation}\label{Eq.claim_invert_N}
		\text{for any $y\in B_{Y}(\delta_0)$, there exists a unique $x\in \overline{B_X(\varepsilon_0'/2)}$ such that $\cN x=y$.}
	\end{equation}
	Given $y\in B_Y(\delta_0)$, the equation $\cN x=y$ is equivalent to $\cL x=y-\cN_1x$, which is further equivalent to $x=\cL^{-1}y-\cL^{-1}\left[\cN_1x\right]$. This motivates us to consider the nonlinear map 
	\[\cF_y: \overline{B_X(\varepsilon_0'/2)}\to X,\qquad \cF_yx:= \cL^{-1}y-\cL^{-1}\left[\cN_1x\right]\quad\forall\ x\in \overline{B_X(\varepsilon_0'/2)}.\]
	For each $y\in B_Y(\delta_0)$ and $x\in \overline{B_X(\varepsilon_0'/2)}$, we have
	\begin{align*}
		\norm{\cF_yx}_X&\leq C_1\|y\|_Y+C_1\norm{\cN_1x}_Y\leq C_1\delta_0+C_1^2\|x\|_X^2\\
        &\leq C_1\delta_0+C_1^2\left(\frac{\varepsilon_0'}2\right)^2=\frac1{4C_1^2}+\frac{1}{4C_1^2}=\frac{\varepsilon_0'}2,
	\end{align*}
	Hence, $\cF_y$ is a map from $\overline{B_X(\varepsilon_0'/2)}$ to $\overline{B_X(\varepsilon_0'/2)}$. On the other hand, for each $x_1,x_2\in\overline{B_X(\varepsilon_0'/2)}$, by \eqref{Eq.N_1-N_2est} we have
	\begin{align}
		\norm{\cF_yx_1-\cF_yx_2}_X&\leq C_1\norm{\cN_1x_1-\cN_1x_2}_Y\leq \frac{C_1^2}2\left(\|x_1\|_X+\|x_2\|_X\right)\|x_1-x_2\|_X\nonumber\\
		&\leq \frac{C_1^2}2\varepsilon_0'\|x_1-x_2\|_X=\frac12 \|x_1-x_2\|_X.\label{Eq.F_ycontraction}
	\end{align}
	Hence, the map $\cF_y: \overline{B_X(\varepsilon_0'/2)} \to \overline{B_X(\varepsilon_0'/2)}$ is a contraction. By Banach's fixed point theorem, there is a unique $x\in \overline{B_X(\varepsilon_0'/2)}$ such that $\cF_yx=x$, or equivalently, $\cN x=y$. This proves our claim \eqref{Eq.claim_invert_N}.
	
	Now, we can define the inversion map $\cN^{-1}$ from $B_Y(\delta_0)$ to $\overline{B_X(\varepsilon_0'/2)}$ according to \eqref{Eq.claim_invert_N}. Given $y\in B_Y(\delta_0)$, we let $x:=\cN^{-1}y$. Then by \eqref{Eq.F_ycontraction}, we have
	\begin{align*}
		\|x\|_X=\|\cF_yx\|_X\leq \|\cF_yx-\cF_y[0_X]\|_X+\norm{\cF_y[0_X]}_X\leq \frac12\|x\|_X+\norm{\cL^{-1}y}_X.
	\end{align*}
	Hence, $\|\cN^{-1}y\|_X=\|x\|_X\leq 2\norm{\cL^{-1}y}_X\leq 2C_1\|y\|$, which proves \eqref{Eq.N^-1_bound}. Finally, for $y_1, y_2\in B_Y(\delta_0)$, we let $x_1=\cN^{-1}y_1, x_2=\cN^{-1}y_2$. Then by \eqref{Eq.F_ycontraction}, we have
	\begin{align*}
		\norm{x_1-x_2}_X&=\norm{\cF_{y_1}x_1-\cF_{y_2}x_2}_X\leq \norm{\cF_{y_1}x_1-\cF_{y_1}x_2}_X+\norm{\cF_{y_1}x_2-\cF_{y_2}x_2}_X\\
		&\leq \frac12\|x_1-x_2\|_X+\norm{\cL^{-1}y_1-\cL^{-1}y_2}_X.
	\end{align*}
	Thus, $\norm{\cN^{-1}y_1-\cN^{-1}y_2}_X=\norm{x_1-x_2}_X\leq 2 \norm{\cL^{-1}y_1-\cL^{-1}y_2}_X\leq 2C_1\|y_1-y_2\|_Y$, which proves \eqref{Eq.N^-1_Lip}.
\end{proof}

\section{Contraction of $\mathcal I_m$}\label{Sec.I}
In this section, we prove Proposition \ref{Prop.I}. As we mentioned in Subsection \ref{Subsec.sketch_I}, the first step is to rewrite the complicated principle value integral defining $\cI_m$ in the form of a series \eqref{Eq.I_series}, which is our main purpose in Subsection \ref{Subsec.reformulation}. In Subsection \ref{Subsec.decomposition-P_n}, we cleverly decompose $\cP_n$ into two parts $\cP_n=\cP_n^{(1)}+\cP_n^{(2)}$ (i.e. \eqref{Eq.P_n_decompose0}), as detailed in \eqref{Eq.P_n_decompose}. It is noted that $\cP_n^{(2)}$ is related to $(r,\g)$ in a local manner, hence the weighted H\"older estimate of $\cP_n^{(2)}$ is straightforward, as discussed in Subsection \ref{Subsec.P_n^2est}. The estimate of $\cP_n^{(1)}$is much more involved due to the highly nonlocal and nonlinear expression of $(\cP_n[r,\g])'$ (see \eqref{Eq.P_n,1'_express}). To streamline the proof, we express it as a composition of two nonlinear operators (see \eqref{Eq.P_n,1'=T_n}), and then we prove the properties of these two operators separately. Full details can be found in Subsection \ref{Subsec.P_n-W^1,infty}. Finally, a routine summation argument implies the contraction of $\mathcal I_m$, thus completing the proof of Proposition \ref{Prop.I}. This is detailed in Subsection \ref{Subsec.Proof-contraction}.

\subsection{Reformulation of $\cI_m$}\label{Subsec.reformulation}
We first recall the definition \eqref{Eq.I_m}:
\begin{equation}\label{Eq.I}
	\mathcal I_m[r,\g](\th):=\frac1{2\pi\ii}\left[\pv\int_0^\infty\frac{\left(\g'(\wt\th)+\g_0'(\wt\th)\right)\,\mathrm d\wt\th}{1-\left(\frac{r(\wt\th)+r_0(\wt\th)}{r(\th)+r_0(\th)}\right)^m\mathrm e^{\ii m(\wt\th-\th)}}-\int_\th^\infty\left(\g'(\wt\th)+\g_0'(\wt\th)\right)\,\mathrm d\wt\th\right].
\end{equation}
We also recall the limiting Kaden's solution \eqref{special_sol}: 
\begin{equation}\label{Eq.special_sol}
	r_0(\th)=\th^{-\mu},\quad\g_0(\th)=-2\pi\th^{1-2\mu}.
\end{equation}
Using the elementary identity
\[\frac1{1-r\e^{\ii\th}}=\begin{cases}
	\sum_{n=0}^\infty r^n\e^{\ii n\th},&\text{if }r\in(0,1)\text{ and }\th\in\R,\\
	-\sum_{n=1}^\infty r^{-n}\e^{-\ii n\th}, &\text{if }r\in(1,+\infty)\text{ and }\th\in\R,
\end{cases}\]
and the fact that $\th\mapsto r(\th)+r_0(\th)$ is strictly decreasing if $\norm{\th^{\mu+1}r'}_{L^\infty}\leq \mu/2$,
we obtain
\begin{align*}
	&\frac{1}{1-\left(\frac{r(\wt\th)+r_0(\wt\th)}{r(\th)+r_0(\th)}\right)^m\mathrm e^{\ii m(\wt\th-\th)}}\\=\ &\sum_{n=0}^\infty \left(\frac{r(\wt\th)+r_0(\wt\th)}{r(\th)+r_0(\th)}\right)^{mn}\e^{\ii mn(\wt\th-\th)}\mathbf{1}_{\wt\th>\th}-\sum_{n=1}^\infty \left(\frac{r(\wt\th)+r_0(\wt\th)}{r(\th)+r_0(\th)}\right)^{-mn}\e^{-\ii mn(\wt\th-\th)}\mathbf{1}_{\wt\th<\th}\\
	=\ &\sum_{n=0}^\infty\left(\frac{\wt\th}\th\right)^{-\mu mn}\left(\frac{R(\wt\th)}{R(\th)}\right)^{mn}\e^{\ii mn(\wt\th-\th)}\mathbf{1}_{\wt\th>\th}-\sum_{n=1}^\infty \left(\frac{\wt\th}\th\right)^{\mu mn}\left(\frac{R(\wt\th)}{R(\th)}\right)^{-mn}\e^{-\ii mn(\wt\th-\th)}\mathbf{1}_{\wt\th<\th},
\end{align*}
where
\begin{equation}\label{Eq.Rr}
	R(\th)=R[r](\th):=1+\th^\mu r(\th).
\end{equation}
Hence we have, at least formally, 
\begin{align}\label{Eq.I_series_proof}
	\cI_m[r,\g](\th)=\frac{2\mu-1}{\ii}\sum_{n=1}^\infty \cQ_{mn}[r,\g](\th),
\end{align}
where $\cQ_n:=-\cQ_{n,1}+\cQ_{n,2}$, and
\begin{align}
	\cQ_{n,1}[r,\g](\th):&=\int_0^\th \wt\th^{-2\mu}\Gamma[\g](\wt\th)\left(\frac{\wt\th}\th\right)^{\mu n}\left(\frac{R[r](\wt\th)}{R[r](\th)}\right)^{-n}\e^{\ii n(\th-\wt\th)}\,\mathrm d\wt\th,\label{Eq.Q_n,1}\\
	\cQ_{n,2}[r,\g](\th):&=\int_\th^\infty \wt\th^{-2\mu}\Gamma[\g](\wt\th)\left(\frac{\wt\th}\th\right)^{-\mu n}\left(\frac{R[r](\wt\th)}{R[r](\th)}\right)^{n}\e^{\ii n(\wt\th-\th)}\,\mathrm d\wt\th,\label{Eq.Q_n,2}\\
	\Gamma[\g]:&=1+\frac{\th^{2\mu}\g'}{2\pi(2\mu-1)},\label{Eq.Gamma_gamma}
\end{align}
for $n\in\Z_+$ and $\th>0$. The rigorous verification of \eqref{Eq.I_series_proof} is non-trivial, which relies on the estimates of $\cQ_{n}[r,\g]$ in Subsections \ref{Subsec.decomposition-P_n}--\ref{Subsec.P_n-W^1,infty}, along with some necessary adaptations. We omit the details here and emphasize that Subsections \ref{Subsec.decomposition-P_n}--\ref{Subsec.P_n-W^1,infty} are independent of the decomposition \eqref{Eq.I_series_proof}, since we will only estimate $\cQ_{n}[r,\g]$ in these subsections. We also note that the right-hand side of \eqref{Eq.I_series_proof} turns out to be absolutely convergent, as shown by \eqref{Eq.P_n} and Proposition \ref{Prop.P_n_est}.

\subsection{Decomposition of $\cP_n$}\label{Subsec.decomposition-P_n}
Recall the definition of $\cP_n$:
\begin{equation}\label{Eq.P_n}
	\cP_{n}[r,\g](\th):=(-\mu+\ii\th)\th^{2\mu-1}\cQ_n[r,\g](\th).
\end{equation}
The main purpose of this subsection is to establish the decomposition of $\cP_n$ in \eqref{Eq.P_n_decompose0}, which plays a crucial role in the estimate of $\cP_n$ in $W^{1,\infty}$.

Let
\begin{align*}
	\cP_{n,1}[r,\g](\th):=(-\mu+\ii\th)\th^{2\mu-1}\cQ_{n,1}[r,\g](\th),\quad \cP_{n,2}[r,\g](\th):=(-\mu+\ii\th)\th^{2\mu-1}\cQ_{n,2}[r,\g](\th),
\end{align*}
then $\cP_n=-\cP_{n,1}+\cP_{n,2}$. By \eqref{Eq.Q_n,1} and \eqref{Eq.Q_n,2}, we can express $\cP_{n,1}$ and $\cP_{n,2}$ in the following form:
\begin{align}
	\cP_{n,1}[r,\g](\th)=\int_0^\th\frac{-\mu+\ii\wt\th}{\wt\th}\Gamma[\g](\wt\th)\e^{nh_n[r](\th)-nh_n[r](\wt\th)}\,\mathrm d\wt\th,\label{Eq.P_n,1}\\
	\cP_{n,2}[r,\g](\th)=\int_\th^\infty\frac{-\mu+\ii\wt\th}{\wt\th}\Gamma[\g](\wt\th)\e^{n\phi_n[r](\wt\th)-n\phi_n[r](\th)}\,\mathrm d\wt\th,\label{Eq.P_n,2}
\end{align}
where
\begin{align}
	h_n[r](\th):&=\frac1n\ln(\mu-\ii\th)+\left(-\mu+\frac{2\mu-1}{n}\right)\ln\th+\ln R[r](\th)+\ii\th,\label{Eq.h_n}\\
	\phi_n[r](\th):&=-\frac1n\ln(\mu-\ii\th)+\left(-\mu-\frac{2\mu-1}{n}\right)\ln\th+\ln R[r](\th)+\ii\th,\label{Eq.phi_n}
\end{align}
for all $n\in\Z_+$ and $\th>0$. Here $\ln$ is the principal branch of the logarithmic function on the right half complex plane.

As explained in Subsection \ref{Subsec.sketch_I}, our analytical framework is based on a crucial observation: the $W^{1,\infty}$ estimate of  $\cP_n$ can be fundamentally reduced to establishing  $C^0$ estimates of the derivative $(\cP_n[r,\g])'$. We now demonstrate this point.

A direct computation gives 
\begin{align}
	\left(\cP_{n,1}[r,\g]\right)'(\th)&=\frac{-\mu+\ii\th}{\th}\Gamma[\g](\th)+n\left(h_n[r]\right)'(\th)\cP_{n,1}[r,\g](\th),\label{Eq.P_n,1'}\\
	\left(\cP_{n,2}[r,\g]\right)'(\th)&=-\frac{-\mu+\ii\th}{\th}\Gamma[\g](\th)-n\left(\phi_n[r]\right)'(\th)\cP_{n,2}[r,\g](\th).\label{Eq.P_n,2'}
\end{align}
It follows from $\cP_n=-\cP_{n,1}+\cP_{n,2}$, \eqref{Eq.P_n,1'} and \eqref{Eq.P_n,2'} that
\begin{equation}\label{Eq.P_n_decompose}
	\cP_n[r,\g]=\underbrace{-\frac{\left(\cP_{n,1}[r,\g]\right)'}{n\left(h_n[r]\right)'}-\frac{\left(\cP_{n,2}[r,\g]\right)'}{n\left(\phi_n[r]\right)'}}_{\cP_n^{(1)}[r,\g]}+\underbrace{\frac{(-\mu+\ii\th)\Gamma[\g]}{n\th}\left(\frac1{\left(h_n[r]\right)'}-\frac1{\left(\phi_n[r]\right)'}\right)}_{\cP_n^{(2)}[r,\g]}.
\end{equation}
This gives the decomposition of $\cP_n$ in \eqref{Eq.P_n_decompose0}.

\subsection{Estimate $\cP_n^{(2)}$ in H\"older spaces}\label{Subsec.P_n^2est}
Now we handle the easier part $\cP_n^{(2)}$. Our main goal in this subsection is to prove \eqref{Eq.P_n^2est} and \eqref{Eq.P_n^2est_diff}.

By \eqref{Eq.P_n_decompose}, \eqref{Eq.h_n} and \eqref{Eq.phi_n}, we have
\begin{align}
	\cP_n^{(2)}[r,\g](\th)&=\frac{2(-\mu+\ii\th)\Gamma[\g](\th)}{n^2\th \left(h_n[r]\right)'(\th)\left(\phi_n[r]\right)'(\th)}\left(\frac\ii{\mu-\ii\th}-\frac{2\mu-1}{\th}\right)\nonumber\\
	&=\frac1{n^2}\frac{2\Gamma[\g](\th)K_0(\th)}{\cH_n[r](\th)\Phi_n[r](\th)},\label{Eq.P_n^2}
\end{align}
where
\begin{align*}
	K_0(\th):=\frac\th{-\mu+\ii\th}\left(\frac\ii{\mu-\ii\th}-\frac{2\mu-1}{\th}\right), \ \cH_n[r](\th):=\frac{\th\left(h_n[r]\right)'(\th)}{-\mu+\ii\th},\ \Phi_n[r](\th):=\frac{\th\left(\phi_n[r]\right)'(\th)}{-\mu+\ii\th}
\end{align*}
for $n\in\Z_+$ and $\th>0$. Using Lemma \ref{Lem.seminorm}, one checks that $K_0\in\mathscr G^1$ and $(-\mu+\ii\th)^{-1}\in \mathscr G^1$. Moreover, the following inequality is useful in the subsequent analysis: 
\begin{equation}\label{4.16}
	\norm{\frac{f}{-\mu+\ii\th}}_1\lesssim \|f\|_0,\quad\forall\ f\in\mathscr G^0.
\end{equation}
The proof of \eqref{4.16} is similar to the proof of Lemma \ref{Lem.algebra}; see also the proof of Lemma \ref{Lem.I_bound_by_J}.

\begin{lemma}\label{Lem.H_n_Phi_n}
	The maps $\cH_n: B_{X_r}(1/2)\to \mathscr G^1_*$ and $\Phi_n: B_{X_r}(1/2)\to \mathscr G^1_*$ are smooth with uniform $C^1$ bounds with respect to $n\in\Z_+$, and
	\begin{equation*}
		\norm{\cH_n[r]-1}_1+\norm{\Phi_n[r]-1}_1\leq C\left(1/n+\|r\|_{X_r}\right),\quad\forall\ n\in\Z_+,\ \forall\ r\in B_{X_r}(1/2),
	\end{equation*}
	where $C>0$ is a constant depending only on $\alpha,\mu$.
\end{lemma}
\begin{proof}
	A direct computation using \eqref{Eq.h_n} and \eqref{Eq.Rr} yields that
	\begin{align}
		\cH_n[r](\th)&=\frac\th{-\mu+\ii\th}\left(\frac1n\frac{-\ii}{\mu-\ii\th}+\left(-\mu+\frac{2\mu-1}{n}\right)\frac1\th+\frac{R[r]'(\th)}{R[r](\th)}+\ii\right)\nonumber\\
		&=\frac\th{-\mu+\ii\th}\left(\frac{-\mu+\ii\th}{\th}+\frac{\th^\mu r'(\th)+\mu \th^{\mu-1}r(\th)}{1+\th^\mu r(\th)}-\frac1n\left(\frac\ii{\mu-\ii\th}-\frac{2\mu-1}{\th}\right)\right)\nonumber\\
		&=1+\frac{\th^{\mu+1}r'(\th)+\mu\th^\mu r(\th)}{(-\mu+\ii\th)\left(1+\th^\mu r(\th)\right)}-\frac1nK_0(\th).\label{Eq.H_n}
	\end{align}
    Due to the facts that $\cG^1_*$ is a unital Banach algebra, $\cG_*^1\subset\cG^0_*$ and $\cG_0$ is an ideal of $\cG^0_*$, the following maps are all smooth:
    \begin{align*}
        r\mapsto (1+\th^\mu r)^{-1}: &\ B_{X_r}(1/2)\to \cG^1_*,\quad r\mapsto \frac{\th^{\mu+1}r'+\mu\th^\mu r}{1+\th^\mu r}: B_{X_r}(1/2)\to \cG^0,\\
        &r\mapsto \frac{\th^{\mu+1}r'(\th)+\mu\th^\mu r(\th)}{(-\mu+\ii\th)\left(1+\th^\mu r(\th)\right)}: B_{X_r}(1/2)\to \mathscr G^1,
    \end{align*}
    where in the last step we have used \eqref{4.16}. Hence, using $K_0\in\mathscr G^1$, we know that the map $\cH_n: B_{X_r}(1/2)\to \mathscr G^1_*$ is smooth and 
		\begin{align*}
		\norm{\cH_n[r]-1}_1\lesssim \norm{(1+\th^\mu r)^{-1}}_{*,1}\left(\norm{\th^{\mu+1}r'}_0+\norm{\th^\mu r}_1\right)+\frac1n\norm{K_0}_1\lesssim 1/n+\|r\|_{X_r}.
	\end{align*}
    Note that only the third term in \eqref{Eq.H_n} depends on $n$, which does not depend on $r$. Hence, there exists a constant $C>0$ independent of $n\in \Z_+$ and $r\in B_{X_r}(1/2)$ such that
    \[\norm{\cH_n[r]}_{*,1}+\norm{\cH_n'[r]}_{L(X_r, \mathscr G^1_*)}\leq C,\quad \forall\ n\in\Z_+,\ \forall\ r\in B_{X_r}(1/2).\]
	The estimate of $\Phi_n[r]$ is similar and is based on the expression
	\[\Phi_n[r](\th)=1+\frac{\th^{\mu+1}r'(\th)+\mu\th^\mu r(\th)}{(-\mu+\ii\th)\left(1+\th^\mu r(\th)\right)}+\frac1nK_0(\th).\]
    This completes the proof.
\end{proof}

\begin{lemma}\label{Lem.wtP_n^2smoooth}
	There exist $\varepsilon_1\in(0,1)$ and $N\in\Z_+$ such that the nonlinear map
	\[(r,\g)\mapsto \wt{\cP_n^{(2)}}[r,\g]:=\frac{\G[\g]K_0}{\cH_n[r]\Phi_n[r]}: B_X(\varepsilon_1)\to \mathscr G^1\]
	is smooth for all $n>N$ with uniform $C^1$ bound with respect to $n>N$.
\end{lemma}
\begin{proof}
	By Lemma \ref{Lem.H_n_Phi_n}, taking small enough $\varepsilon_1<1/2$ and large enough $N\in \Z_+$, we have
	\[\norm{\cH_n[r]-1}_1+\norm{\Phi_n[r]-1}_1\leq 1/2,\quad\forall\ r\in B_{X_r}(\varepsilon_1),\ \forall\ n>N.\]
	Then by Lemma \ref{Lem.H_n_Phi_n}, $K_0\in\mathscr G^1$, $\Gamma-1\in L(X_\g, \cG^1_*)$ , the fact that $\cG^1_*$ is a unital Banach algebra and $\mathscr G^1\subset \mathscr G^1_*$ is a sub-algebra, we know that $\wt{\cP_n^{(2)}}: B_X(\varepsilon_1)\to \mathscr G^1$ is a smooth map for all $n>N$. Thanks to Lemma \ref{Lem.H_n_Phi_n}, the uniform $C^1$ bounds of $\cH_n$ and $\Phi_n$ on $B_{X_r}(1/2)$ imply that
	\begin{equation}\label{Eq.wtP_n^2_C^1}
		\norm{\wt{\cP_n^{(2)}}[r,\g]}_1+\norm{\left(\wt{\cP_n^{(2)}}\right)'[r,\g]}_{L(X, \mathscr G^1)}\leq C,\quad \forall\ (r,\g)\in B_X(\varepsilon_1),\ \forall\ n>N
	\end{equation}
	for some constant $C>0$ independent of $(r,\g)\in B_X(\varepsilon_1)$ and $n>N$.
\end{proof}

Now, \eqref{Eq.P_n^2est} and \eqref{Eq.P_n^2est_diff} are direct consequences of Lemma \ref{Lem.wtP_n^2smoooth}. Namely, we have
\begin{corollary}
	There exist $\varepsilon_1\in(0,1)$, $N\in\Z_+$ and a constant $C=C(\al,\mu)>0$ such that
	\begin{equation}\label{Eq.P_n^2est_all}
		\norm{\cP_n^{(2)}[r,\g]}_1\leq\frac C{n^2},\quad \norm{\cP_n^{(2)}[r_1,\g_1]-\cP_n^{(2)}[r_2,\g_2]}_1\leq \frac C{n^2}\norm{(r_1, \g_1)-(r_2,\g_2)}_X
	\end{equation}
	for all $(r,\g), (r_1, \g_1), (r_2,\g_2)\in B_X(\varepsilon_1)$.
\end{corollary}
\begin{proof}
	It follows from \eqref{Eq.P_n^2} that
	\[\cP_n^{(2)}=\frac{2}{n^2}\wt{\cP_n^{(2)}},\quad\forall\ n\in\Z_+.\]
	Therefore, \eqref{Eq.wtP_n^2_C^1} implies \eqref{Eq.P_n^2est_all}.
\end{proof}

\subsection{Estimate $\cP_n$ in $W^{1,\infty}$}\label{Subsec.P_n-W^1,infty}
This subsection is devoted to the proof of \eqref{Eq.P_n^1est}, \eqref{Eq.P_n^1est_diff}, \eqref{Eq.C^1est} and \eqref{Eq.C^1est_diff}. The main result of this subsection is the following

\begin{proposition}\label{Prop.P_n,1,2}
	There exist $\varepsilon_1\in(0,1)$, $N\in\Z_+$ and a constant $C=C(\al,\mu)>0$ such that
	\begin{align}
		\left|\left(\cP_{n,1}[r,\g]\right)'(\th)\right|&\leq \frac C{n^\alpha}\th^{\alpha-1}\langle\th\rangle^{1-2\al}\left(\frac1n+\|(r,\g)\|_X\right),\label{Eq.C^1est_1}\\
		\left|\left(\cP_{n,1}[r_1,\g_1]\right)'(\th)-\left(\cP_{n,1}[r_2,\g_2]\right)'(\th)\right|&\leq \frac C{n^\alpha}\th^{\alpha-1}\langle\th\rangle^{1-2\al}\left\|(r_1,\g_1)-(r_2,\g_2)\right\|_X,\label{Eq.C^1est_diff_1}\\
		\left|\left(\cP_{n,2}[r,\g]\right)'(\th)\right|&\leq \frac C{n^\alpha}\th^{\alpha-1}\langle\th\rangle^{1-2\al}\left(\frac1n+\|(r,\g)\|_X\right),\label{Eq.C^1est_2}\\
		\left|\left(\cP_{n,2}[r_1,\g_1]\right)'(\th)-\left(\cP_{n,2}[r_2,\g_2]\right)'(\th)\right|&\leq \frac C{n^\alpha}\th^{\alpha-1}\langle\th\rangle^{1-2\al}\left\|(r_1,\g_1)-(r_2,\g_2)\right\|_X,\label{Eq.C^1est_diff_2}
	\end{align}
	for all $(r,\g), (r_1, \g_1), (r_2,\g_2)\in B_X(\varepsilon_1)$, $\th>0$ and $n>N$.
\end{proposition}

We first show that Proposition \ref{Prop.P_n,1,2} implies \eqref{Eq.P_n^1est}, \eqref{Eq.P_n^1est_diff}, \eqref{Eq.C^1est} and \eqref{Eq.C^1est_diff}.

\begin{proof}[Proof of \eqref{Eq.C^1est} and \eqref{Eq.C^1est_diff}]
	Recall that $\cP_n=-\cP_{n,1}+\cP_{n,2}$. We see that \eqref{Eq.C^1est} and \eqref{Eq.C^1est_diff} follow directly from Proposition \ref{Prop.P_n,1,2}.
\end{proof}

\begin{proof}[Proof of \eqref{Eq.P_n^1est} and \eqref{Eq.P_n^1est_diff}]
	By the definition of $\cP_n^{(1)}$ in \eqref{Eq.P_n_decompose} and the definitions of $\cH_n$, $\Phi_n$, we obtain
	\begin{equation}\label{Eq.P_n^1}
		-\cP_n^{(1)}[r,\g]=\frac{\th\left(\cP_{n,1}[r,\g]\right)'}{n(-\mu+\ii\th)\cH_n[r]}+\frac{\th\left(\cP_{n,2}[r,\g]\right)'}{n(-\mu+\ii\th)\Phi_n[r]}.
	\end{equation}
	By Lemma \ref{Lem.H_n_Phi_n}, there exist $\varepsilon_1\in(0,1)$ and $N\in\Z_+$ such that
	\[r\mapsto \frac1{\cH_n[r]}-1: B_{X_r}(\varepsilon_1)\to \mathscr G^1, \quad r\mapsto \frac1{\Phi_n[r]}-1: B_{X_r}(\varepsilon_1)\to \mathscr G^1,\quad n>N\]
	are smooth maps with uniform $C^1$ bounds. In particular,
	\begin{equation}\label{Eq.1/H_Lip}
		\norm{\frac1{\cH_n[r_1]}-\frac1{\cH_n[r_2]}}_1+\norm{\frac1{\Phi_n[r_1]}-\frac1{\Phi_n[r_2]}}_1\leq C\norm{r_1-r_2}_{X_r}
	\end{equation}
	for all $r_1, r_2\in B_{X_r}(\varepsilon_1)$ and $n>N$, where $C>0$ is a constant depending only on $\al,\mu$.
	As a consequence, it follows from \eqref{Eq.P_n^1} and Proposition \ref{Prop.P_n,1,2} that
	\begin{align*}
		\left|\cP_n^{(1)}[r,\g](\th)\right|&\lesssim\frac\th{n\langle\th\rangle}\left|\left(\cP_{n,1}[r,\g]\right)'(\th)\right|+\frac\th{n\langle\th\rangle}\left|\left(\cP_{n,2}[r,\g]\right)'(\th)\right|\\
		&\lesssim \frac1{n^{1+\al}}\th^\al\langle\th\rangle^{-2\al}\left(\frac1n+\norm{(r,\g)}_X\right)
	\end{align*}
	for all $(r, \g)\in B_X(\varepsilon_1)$, $\th>0$ and $n>N$; and also (using \eqref{Eq.1/H_Lip})
	\begin{align*}
		&\left|\cP_n^{(1)}[r_1,\g_1](\th)-\cP_n^{(1)}[r_2,\g_2](\th)\right|\\
		\lesssim&\ \frac\th{n\langle\th\rangle}\left|\frac1{\cH_n[r_1](\th)}\right|\left|\left(\cP_{n,1}[r_1,\g_1]\right)'(\th)-\left(\cP_{n,1}[r_2,\g_2]\right)'(\th)\right|\\
		&\qquad+\frac\th{n\langle\th\rangle}\left|\left(\cP_{n,1}[r_2,\g_2]\right)'(\th)\right|\left|\frac1{\cH_n[r_1](\th)}-\frac1{\cH_n[r_2](\th)}\right|\\
		&\qquad+\frac\th{n\langle\th\rangle}\left|\frac1{\Phi_n[r_1](\th)}\right|\left|\left(\cP_{n,2}[r_1,\g_1]\right)'(\th)-\left(\cP_{n,2}[r_2,\g_2]\right)'(\th)\right|\\
		&\qquad+\frac\th{n\langle\th\rangle}\left|\left(\cP_{n,2}[r_2,\g_2]\right)'(\th)\right|\left|\frac1{\Phi_n[r_1](\th)}-\frac1{\Phi_n[r_2](\th)}\right|\\
		\lesssim&\ \frac1{n^{1+\al}}\th^\al\langle\th\rangle^{-2\al}\norm{(r_1,\g_1)-(r_2,\g_2)}_X
	\end{align*}
	for all $(r_1,\g_1), (r_2,\g_2)\in B_X(\varepsilon_1)$, $\th>0$ and $n>N$.
\end{proof}

The rest of this subsection is devoted to the proof of Proposition \ref{Prop.P_n,1,2}. We will analyze $\cP_{n,1}$ in full details, hence proving \eqref{Eq.C^1est_1} and \eqref{Eq.C^1est_diff_1}. After that, we will briefly sketch the proof of \eqref{Eq.C^1est_2} and \eqref{Eq.C^1est_diff_2}, whose proof shares a similar spirit to the analysis of $\cP_{n,1}$.

By the definition \eqref{Eq.P_n,1}, we have
\begin{align*}
	\left(\cP_{n,1}[r,\g]\right)'(\th)=\frac{-\mu+\ii\th}{\th}\Gamma[\g](\th)+n\left(h_n[r]\right)'(\th)\int_0^\th \frac{-\mu+\ii\wt\th}{\wt\th}\Gamma[\g](\wt\th)\e^{nh_n[r](\th)-nh_n[r](\wt\th)}\,\mathrm d\wt\th.
\end{align*}
It follows from \eqref{Eq.h_n} that $\lim_{\th\to0+}\operatorname{Re}h_n[r](\th)=+\infty$, hence,
\begin{align*}
	n\int_0^\th \e^{nh_n[r](\th)-nh_n[r](\wt\th)}\left(h_n[r]\right)'(\wt\th)\,\mathrm d\wt\th= \e^{nh_n[r](\th)-nh_n[r](\wt\th)}\Big|^{\wt\th=0}_{\wt\th=\th}=-1.
\end{align*}
Therefore, we have
\begin{align}
	&\left(\cP_{n,1}[r,\g]\right)'(\th)=n\int_0^\th \left(\frac{-\mu+\ii\wt\th}{\wt\th}\Gamma[\g](\wt\th)\left(h_n[r]\right)'(\th)-\frac{-\mu+\ii\th}{\th}\Gamma[\g](\th) \left(h_n[r]\right)'(\wt\th)\right) \times\nonumber\\
	&\qquad\qquad\qquad\qquad\qquad\times\e^{nh_n[r](\th)-nh_n[r](\wt\th)}\,\mathrm d\wt\th\nonumber\\
	&\quad=n\int_0^\th\left(\cD_n[r,\g](\wt\th)-\cD_n[r,\g](\th)\right) \left(h_n[r]\right)'(\th) \left(h_n[r]\right)'(\wt\th) \e^{nh_n[r](\th)-nh_n[r](\wt\th)}\,\mathrm d\wt\th,\label{Eq.P_n,1'_express}
\end{align}
where
\begin{equation}
	\cD_n[r,\g](\th):=\frac{(-\mu+\ii\th)\Gamma[\g](\th)}{\th \left(h_n[r]\right)'(\th)}-1=\frac{\Gamma[\g](\th)}{\cH_n[r](\th)}-1.
\end{equation}

By Lemma \ref{Lem.H_n_Phi_n}, \eqref{Eq.Gamma_gamma} and Banach algebraic properties, there exist $\varepsilon_1\in(0,1)$ and $N\in\Z_+$ such that the nonlinear map $\cD_n: B_X(\varepsilon_1)\to \mathscr G^1$ is smooth for all $n>N$ with a uniform $C^1$ bound with respect to $n>N$, and
\begin{equation}\label{Eq.D_n_G^1bound}
	\norm{\cD_n[r,\g]}_1\leq C\left(1/n+\norm{(r,\g)}_X\right),\quad\forall\ n>N,\ \forall\ (r,\g)\in B_X(\varepsilon_1),
\end{equation}
where $C=C(\mu,\al)>0$ is a constant independent of $n>N$ and $(r,\g)\in B_X(\varepsilon_1)$.

This motivates us to consider the following nonlinear map
\begin{equation*}
	\cT_n[F,r](\th):=n \left(h_n[r]\right)'(\th)\int_0^\th \left(F(\wt\th)-F(\th)\right)\left(h_n[r]\right)'(\wt\th) \e^{nh_n[r](\th)-nh_n[r](\wt\th)}\,\mathrm d\wt\th
\end{equation*}
for $F\in\mathscr G^1$ and $r\in X_r$ near $0$. We note that $\cT_n$ is linear in $F$ and it follows from \eqref{Eq.P_n,1'_express} that
\begin{equation}\label{Eq.P_n,1'=T_n}
	\left(\cP_{n,1}[r,\g]\right)'=\cT_n\left[\cD_n[r,\g],r\right].
\end{equation}
In this way, we reduce the proof of \eqref{Eq.C^1est_1} and \eqref{Eq.C^1est_diff_1} to the analysis of $\cT_n$.

We start by the estimate in the case $\th\in(0,1]$. 

\begin{lemma}\label{Lem.T_nF,r_th_small}
	There exist $\varepsilon_1\in(0,1)$, $N\in\Z_+$ and a constant $C=C(\mu,\al)>0$ such that
	\begin{align}
		\left|\cT_n[F,r](\th)\right|&\leq C\frac{\th^{\al-1}}{n^\al}[F]_1,\label{Eq.T_nF,r_th_small}\\
		\left|\cT_n[F,r_1](\th)-\cT_n[F, r_2](\th)\right|&\leq C \frac{\th^{\al-1}}{n^\al}[F]_1\norm{r_1-r_2}_{X_r},\label{Eq.T_nF,r_diff_th_small}
	\end{align} 
	for all $r, r_1, r_2\in B_{X_r}(\varepsilon_1)$, $n>N$, $F\in\mathscr G^1$ and $\th\in(0,1]$.
\end{lemma}
\begin{proof}
	By definition, we have
	\begin{align}
		\cT_n[F,r](\th)&=n\left(h_n[r]\right)'(\th)\int_0^\th \left(F(\wt\th)-F(\th)\right) \left(h_n[r]\right)'(\wt\th)\frac{\mu-\ii\th}{\mu-\ii\wt\th}\left(\frac{\wt\th}{\th}\right)^{\mu n-2\mu+1}\label{Eq.T_nF,r_express}\\
		&\qquad\qquad\times\left(\frac{R[r](\th)}{R[r](\wt\th)}\right)^n\e^{\ii n(\th-\wt\th)}\,\mathrm d\wt\th.\nonumber
	\end{align}
	By Lemma \ref{Lem.H_n_Phi_n}, there exist $\varepsilon_1\in(0,1)$ and $N\in\Z_+$ that
    \begin{equation}\label{Eq.h_n'_bound}
        \frac\th{-\mu+\ii\th}\left|\left(h_n[r]\right)'(\th)\right|=\left|\cH_n[r](\th)\right|\sim1\Longrightarrow \left| \left(h_n[r]\right)'(\th)\right|\sim\frac{\langle\th\rangle}{\th},\quad\forall\ \th>0
    \end{equation}
	for all $n>N$ and $r\in B_{X_r}(\varepsilon_1)$.
	Using  \eqref{Eq.T_nF,r_express} and \eqref{Eq.h_n'_bound}, for all $\th\in(0,1]$, there holds
	\begin{align}
		\left|\cT_n[F,r](\th)\right|&\lesssim n\frac{\langle\th\rangle}{\th}\int_0^\th \frac1{\langle\th\rangle^{1+\al}}\left|\wt\th-\th\right|^\al[F]_1\frac{\langle\wt\th\rangle}{\wt\th}\frac{\langle\th\rangle}{\langle\wt\th\rangle} \left(\frac{\wt\th}{\th}\right)^{\mu n-2\mu+1} \left(\frac{R[r](\th)}{R[r](\wt\th)}\right)^n\,\mathrm d\wt\th\nonumber\\
		&\lesssim n\frac{\langle\th\rangle^{1-\al}}{\th}[F]_1\int_0^\th \left|\wt\th-\th\right|^\al\frac1{\th} \left(\frac{\wt\th}{\th}\right)^{\mu n-2\mu} \left(\frac{R[r](\th)}{R[r](\wt\th)}\right)^n\,\mathrm d\wt\th\nonumber\\
		&\lesssim n \frac{\langle\th\rangle^{1-\al}}{\th^{1-\al}}[F]_1\int_0^1(1-t)^\al t^{\mu n-2\mu}\left(\frac{R[r](\th)}{R[r](\th t)}\right)^n\,\mathrm dt\nonumber\\
		&\lesssim n \frac{\langle\th\rangle^{1-\al}}{\th^{1-\al}}[F]_1\int_0^1|\ln t|^\al t^{\mu n-2\mu}\left(\frac{R[r](\th)}{R[r](\th t)}\right)^n\,\mathrm dt,\label{Eq.T_nF,r_4.36}
	\end{align}
	where in the last inequality we have used the fact $1-t\leq|\ln t|$ for all $t\in(0,1)$. Now we consider a change of variable defined by $s:=t^\mu R[r](\th)/R[r](\th t)\in(0,1)$ for $t\in(0,1)$. Let $\varepsilon_1=\varepsilon_1(\al, \mu)\in(0,1)$ be small enough. For $r\in B_{X_r}(\varepsilon_1)$, we note that $s\sim t^\mu$, $t\sim s^{1/\mu}$ and
	\begin{align*}
		\frac{\mathrm ds}{\mathrm dt}=t^{\mu-1}\frac{R[r](\th)}{R[r](\th t)}\left(\mu-\frac{\th t(R[r])'(\th t)}{R[r](\th t)}\right)\sim t^{\mu-1}\sim s^{1-\frac 1\mu},
	\end{align*}
	which in particular implies that the map $t\mapsto s$ is strictly increasing (for $t\in\R_+$). On the other hand, we also have
	\begin{align*}
		t^{2\mu}\leq t^{2\mu}\frac{R[r](\th)}{t^\mu R[r](\th t)}=s=t^{\mu/2}t^{\mu/2}\frac{R[r](\th)}{R[r](\th t)}\leq t^{\mu/2},
	\end{align*}
	where we have used the facts that $t\mapsto \frac{R[r](\th)}{t^\mu R[r](\th t)}$ is strictly decreasing and $t\mapsto t^{\mu/2}\frac{R[r](\th)}{R[r](\th t)}$ is strictly increasing on $t\in(0,1]$ for $r\in B_{X_r}(\varepsilon_1)$. So $|\ln t|\sim |\ln s|$. As a consequence, 
	\begin{align*}
		&\int_0^1|\ln t|^\al t^{\mu n-2\mu}\left(\frac{R[r](\th)}{R[r](\th t)}\right)^n\,\mathrm dt\lesssim \int_0^1 |\ln s|^\al s^n s^{-2} s^{\frac1\mu-1}\,\mathrm ds\\
		\lesssim &\ \int_0^\infty \tau^\al \e^{-(n-2+1/\mu)\tau}\,\mathrm d\tau\lesssim \frac1{n^{1+\al}}
	\end{align*}
	for all $n\geq 2$ and $r\in B_{X_r}(\varepsilon_1)$. This, combining with \eqref{Eq.T_nF,r_4.36}, proves \eqref{Eq.T_nF,r_th_small}. In particular, for future usage, we have proved that
	\begin{align}\label{Eq.future_th_small}
		n\frac{\langle\th\rangle}{\th}\int_0^\th\left|F(\wt\th)-F(\th)\right| \frac{\langle\wt\th\rangle}{\wt\th}\frac{\langle\th\rangle}{\langle\wt\th\rangle} \left(\frac{\wt\th}{\th}\right)^{\mu n-2\mu+1} \left(\frac{R[r](\th)}{R[r](\wt\th)}\right)^n\,\mathrm d\wt\th\lesssim \frac{1}{n^\al}\frac{\langle\th\rangle^{1-\al}}{\th^{1-\al}}[F]_1
	\end{align}
	for all $r\in B_{X_r}(\varepsilon_1)$, $n>N$, $F\in\mathscr G^1$ and $\th\in(0,1]$.
	Next we show \eqref{Eq.T_nF,r_diff_th_small}. Using \eqref{Eq.T_nF,r_express}, a direct subtraction gives that
\begin{align}\label{Eq.DeltaT_n}
	&\cT_n[F, r_1](\th)-\cT_n[F, r_2](\th)=\Delta_{n,1}+\Delta_{n,2},
\end{align}
where
\begin{align*}
\Delta_{n,1}:&=\frac{\left(h_n[r_1]\right)'(\th)-\left(h_n[r_2]\right)'(\th)}{\left(h_n[r_1]\right)'(\th)}\cT_n[F, r_1](\th)\\
	&\quad+n \left(h_n[r_2]\right)'(\th)\int_0^\th\left(F(\wt\th)-F(\th)\right)\left(\left(h_n[r_1]\right)'(\wt\th)-\left(h_n[r_2]\right)'(\wt\th)\right) \frac{\mu-\ii\th}{\mu-\ii\wt\th}\\		&\qquad\qquad\qquad\qquad\qquad\times \left(\frac{\wt\th}{\th}\right)^{\mu n-2\mu+1}\left(\frac{R[r_1](\th)}{R[r_1](\wt\th)}\right)^n\e^{\ii n(\th-\wt\th)}\,\mathrm d\wt\th,\\
	\Delta_{n,2}:&=n \left(h_n[r_2]\right)'(\th)\int_0^\th\left(F(\wt\th)-F(\th)\right)\left(h_n[r_2]\right)'(\wt\th) \frac{\mu-\ii\th}{\mu-\ii\wt\th}\left(\frac{\wt\th}{\th}\right)^{\mu n-2\mu+1}\\
	&\qquad\qquad\qquad\qquad\times\left(\left(\frac{R[r_1](\th)}{R[r_1](\wt\th)}\right)^n-\left(\frac{R[r_2](\th)}{R[r_2](\wt\th)}\right)^n\right) \e^{\ii n(\th-\wt\th)}\,\mathrm d\wt\th.
\end{align*}	
By Lemma \ref{Lem.H_n_Phi_n}, the map $\cH_n-1: B_{X_r}(1/2)\to \mathscr G^1$ has a uniform $C^1$ bound with respect to $n\in\Z_+$, hence $\norm{\cH_n[r_1]-\cH_n[r_2]}_1\lesssim \|r_1-r_2\|_{X_r}$ for all $r_1, r_2\in B_{X_r}(1/2)$ and $n\in\Z_+$. Then, using the definition of $\cH_n$, we get
	\begin{equation}\label{Eq.h_n'diff}
		\left|\left(h_n[r_1]\right)'(\th)-\left(h_n[r_2]\right)'(\th)\right|\lesssim \frac{|-\mu+\ii\th|}{\th}\norm{r_1-r_2}_{X_r}\lesssim \frac{\langle\th\rangle}{\th} \norm{r_1-r_2}_{X_r}.
	\end{equation}
	It follows from \eqref{Eq.h_n'_bound}, \eqref{Eq.h_n'diff}, \eqref{Eq.T_nF,r_th_small} and \eqref{Eq.future_th_small} that
	\begin{align}
		\left|\Delta_{n,1}\right|&\lesssim \norm{r_1-r_2}_{X_r}\left|\cT_n[F, r_1](\th)\right|\nonumber\\
		&\qquad\qquad+ n\frac{\langle\th\rangle}{\th}\norm{r_1-r_2}_{X_r}\int_0^\th\left|F(\wt\th)-F(\th)\right| \frac{\langle\wt\th\rangle}{\wt\th}\frac{\langle\th\rangle}{\langle\wt\th\rangle} \left(\frac{\wt\th}{\th}\right)^{\mu n-2\mu+1} \left(\frac{R[r_1](\th)}{R[r_1](\wt\th)}\right)^n\,\mathrm d\wt\th\nonumber\\
		&\lesssim \frac{\th^{\al-1}}{n^\al}[F]_1 \norm{r_1-r_2}_{X_r}.\label{Eq.Delta_n,1est}
	\end{align}
	To estimate $\Delta_{n,2}$, we use the identity $\e^{nX}-\e^{nY}=n(X-Y)\int_0^1\e^{nsX+n(1-s)Y}\,\mathrm ds$ for all $X, Y\in \R$ to obtain
	\begin{align}\label{Eq.R_diff}
		\left(\frac{R[r_1](\th)}{R[r_1](\wt\th)}\right)^n-\left(\frac{R[r_2](\th)}{R[r_2](\wt\th)}\right)^n=n\left(\ln \frac{R[r_1](\th)}{R[r_1](\wt\th)}-\ln\frac{R[r_2](\th)}{R[r_2](\wt\th)}\right)\int_0^1\left(\frac{R_{(s)}(\th)}{R_{(s)}(\wt\th)}\right)^n\,\mathrm ds,
	\end{align}
	where $R_{(s)}:=R[r_1]^sR[r_2]^{1-s}$ for $s\in(0,1)$. Hence, by \eqref{Eq.h_n'_bound}, we have
	\begin{align}
		&\left|\Delta_{n,2}\right|\lesssim n\frac{\langle\th\rangle}{\th}\int_0^\th\frac{\left|\wt\th-\th\right|^\al}{\langle\th\rangle^{1+\al}}[F]_1\frac{\langle\wt\th\rangle}{\wt\th}\frac{\langle\th\rangle}{\langle\wt\th\rangle}\left(\frac{\wt\th}{\th}\right)^{\mu n-2\mu+1}\nonumber\\
		&\qquad\qquad\qquad\times n \left|\ln \frac{R[r_1](\th)}{R[r_1](\wt\th)}-\ln\frac{R[r_2](\th)}{R[r_2](\wt\th)}\right|\int_0^1\left(\frac{R_{(s)}(\th)}{R_{(s)}(\wt\th)}\right)^n\,\mathrm ds\,\mathrm d\wt\th\nonumber\\
		&\lesssim n^2\frac{\langle\th\rangle^{1-\al}}{\th^{1-\al}}[F]_1\int_0^1(1-t)^\al t^{\mu n-2\mu}\left|\ln \frac{R[r_1](\th)}{R[r_2](\th)}-\ln\frac{R[r_1](\th t)}{R[r_2](\th t)}\right|\int_0^1\left(\frac{R_{(s)}(\th)}{R_{(s)}(\th t)}\right)^n\,\mathrm ds\,\mathrm dt\nonumber\\
		& \lesssim n^2\frac{\langle\th\rangle^{1-\al}}{\th^{1-\al}}[F]_1\int_0^1|\ln t|^\al t^{\mu n-2\mu}\left|\ln \frac{R[r_1](\th)}{R[r_2](\th)}-\ln\frac{R[r_1](\th t)}{R[r_2](\th t)}\right|\int_0^1\left(\frac{R_{(s)}(\th)}{R_{(s)}(\th t)}\right)^n\,\mathrm ds\,\mathrm dt,\label{Eq.Delta_n,2}
	\end{align}
	where in the last step we used $1-t\leq|\ln t|$ for $t\in(0,1)$. By \eqref{Eq.Rr} and Banach algebraic properties, we know that the nonlinear map
	$r\mapsto \frac{\th \left(R[r]\right)'}{R[r]}: B_{X_r}(1/2)\to \cG^0$
	is smooth, thus,
	\begin{align}
		\left|\ln \frac{R[r_1](\th)}{R[r_2](\th)}-\ln\frac{R[r_1](\th t)}{R[r_2](\th t)}\right|&=\left|\ln\frac{R[r_1](\th\tau)}{R[r_2](\th\tau)}\Bigg|^{\tau=1}_{\tau=t}\right|=\left|\int_t^1\frac{\mathrm d}{\mathrm d\tau} \ln\frac{R[r_1](\th\tau)}{R[r_2](\th\tau)}\,\mathrm d\tau\right|\nonumber\\
		&=\left|\int_t^1\left(\frac{\th\left(R[r_1]\right)'(\th\tau)}{R[r_1](\th\tau)}-\frac{\th\left(R[r_2]\right)'(\th\tau)}{R[r_2](\th\tau)}\right)\,\mathrm d\tau\right|\nonumber\\
		&\lesssim\int_t^1\frac1\tau\,\mathrm d\tau\norm{r_1-r_2}_{X_r}\lesssim |\ln t| \norm{r_1-r_2}_{X_r}\label{Eq.ln_diff}
	\end{align}
	for all $\th\in(0,1]$, $t\in(0,1)$ and $r_1, r_2\in B_{X_r}(1/2)$. Plugging \eqref{Eq.ln_diff} into \eqref{Eq.Delta_n,2} gives that
	\begin{align}\label{Eq.Delta_n,2_est}
		\left|\Delta_{n,2}\right|\lesssim n^2\frac{\langle\th\rangle^{1-\al}}{\th^{1-\al}}[F]_1\norm{r_1-r_2}_{X_r}\int_0^1\left(\int_0^1|\ln t|^{1+\al} t^{\mu n-2\mu}\left(\frac{R_{(s)}(\th)}{R_{(s)}(\th t)}\right)^n\,\mathrm dt\right)\,\mathrm ds.
	\end{align}
	Now, we use an argument similar to the proof of \eqref{Eq.future_th_small}.  Specifically, we make a change of variables $\sigma:=t^\mu R_{(s)}(\th)/{R_{(s)}(\th t)}\in(0,1)$ for $t\in(0,1)$. 
	Then, we have $\sigma\sim t^\mu$, $t\sim\sigma^{1/\mu}$, $\mathrm d\sigma/\mathrm dt\sim t^{\mu-1}\sim \sigma^{1-1/\mu}$, and $|\ln t|\sim |\ln\sigma|$, with all these estimates being uniform with respect to $s\in (0,1)$. Thus, 
	\begin{align*}
		\int_0^1|\ln t|^{1+\al} t^{\mu n-2\mu}\left(\frac{R_{(s)}(\th)}{R_{(s)}(\th t)}\right)^n\,\mathrm dt&\lesssim \int_0^1|\ln\sigma|^{1+\al}\sigma^{n-3+1/\mu}\,\mathrm d\sigma\lesssim \int_0^\infty \tau^{1+\al}\e^{-(n-2+1/\mu)\tau}\,\mathrm d\tau\\
		&\lesssim n^{-2-\al}
	\end{align*}
	for all $\th\in\R_+$, $n\geq 4$ and $r_1, r_2\in B_{X_r}(\varepsilon_1)$, uniformly with respect to $s\in(0,1)$. Hence, it follows from \eqref{Eq.Delta_n,2_est} that
	\begin{align}\label{Eq.Delta_n,2est}
		\left|\Delta_{n,2}\right|\lesssim\frac1{n^\al}\frac{\langle\th\rangle^{1-\al}}{\th^{1-\al}}[F]_1\norm{r_1-r_2}_{X_r}\lesssim \frac{\th^{\al-1}}{n^\al}[F]_1 \norm{r_1-r_2}_{X_r}
	\end{align}
	for all $r_1, r_2\in B_{X_r}(\varepsilon_1)$, $n>N$, $F\in\mathscr G^1$ and $\th\in(0,1]$.
	Combining \eqref{Eq.DeltaT_n}, \eqref{Eq.Delta_n,1est} and \eqref{Eq.Delta_n,2est} gives \eqref{Eq.T_nF,r_diff_th_small}.
\end{proof}

To estimate $\cT_n$ for $\th>1$, we need the following decomposition lemma.

\begin{lemma}\label{Lem.decompose_G^1}
	For any $F\in\mathscr G^1$ and $n\in\Z_+$, we can find functions $F_0\in C(\R_+; \C)$ and $F_1\in C^1(\R_+; \C)$ such that $F=F_0+F_1$ and
	\begin{equation*}
		\norm{\langle\th\rangle^{1+\al}F_0}_{L^\infty}\leq \frac{n^{-\al}}{\al+1}[F]_1,\quad \norm{\langle\th\rangle^{1+\al}F_1'}_{L^\infty}\leq n^{1-\al}[F]_1,\quad \|\langle\th\rangle^\al F_1\|_{L^\infty}\leq \|\langle\th\rangle^\al F\|_{L^\infty}.
	\end{equation*}
\end{lemma}
\begin{proof}
	We define
	\[F_1(\th):=n\int_\th^{\th+\frac1n}F(\wt\th)\,\mathrm d\wt\th,\quad F_0(\th):=F(\th)-F_1(\th),\quad\forall\ \th>0.\]
	Then $F_1\in C^1(\R_+; \C)$, $\|\langle\th\rangle^\al F_1\|_{L^\infty}\leq \|\langle\th\rangle^\al F\|_{L^\infty}$  and $F_1'(\th)=n\left(F(\th+1/n)-F(\th)\right)$ for $\th\in\R_+$. Hence,
	\begin{equation*}
		\left|F_1'(\th)\right|\leq n[F]_1\langle\th+1/n\rangle^{-1-\al}n^{-\al}\leq n^{1-\al}[F]_1\langle\th\rangle^{-1-\al},\quad\forall\ \th>0.
	\end{equation*}
	We also have 
	\begin{align*}
		|F_0(\th)|&=n\left|\int_\th^{\th+\frac1n}\left(F(\th)-F(\wt\th)\right)\,\mathrm d\wt\th\right|\leq n \int_\th^{\th+\frac1n}\langle\wt\th\rangle^{-1-\al}[F]_1\left|\wt\th-\th\right|^\al\,\mathrm d\wt\th\\
		&\leq n\langle\th\rangle^{-1-\al}[F]_1 \int_\th^{\th+\frac1n} \left|\wt\th-\th\right|^\al\,\mathrm d\wt\th=\frac{n^{-\al}\langle\th\rangle^{-1-\al}}{\al+1}[F]_1,\quad\forall\ \th>0.
	\end{align*}
	This completes the proof.
\end{proof}

\begin{lemma}\label{Lem.T_nF,r_th_large1}
	There exist $\varepsilon_1\in(0,1)$, $N\in\Z_+$ and a constant $C=C(\mu,\al)>0$ such that
	\begin{align}
		\left|\cT_n[F,r](\th)\right|&\leq \frac{C}{n\langle\th\rangle^\al}\norm{\langle\th\rangle^{1+\al}F'}_{L^\infty}, \label{Eq.T_nF,r_th_large1}\\
		\left|\cT_n[F, r_1](\th)-\cT_n[F, r_2](\th)\right|&\leq \frac{C}{n\langle\th\rangle^\al}\norm{\langle\th\rangle^{1+\al}F'}_{L^\infty}\norm{r_1-r_2}_{X_r},\label{Eq.T_nF,r_diff_th_large1}
	\end{align}
	for all $r, r_1, r_2\in B_{X_r}(\varepsilon_1)$, $n>N$, $\th>1$ and all $F\in C^1(\R_+; \C)$ satisfying $\langle\th\rangle^{1+\al}F'\in L^\infty$.
\end{lemma}
\begin{proof}
	Note that $\cT_n$ is linear in $F$, then we can assume without loss of generality that $\norm{\langle\th\rangle^{1+\al}F'}_{L^\infty}=1$.
	Using integration by parts and $\lim_{\th\to0+}\operatorname{Re}h_n[r](\th)=+\infty$, we obtain
	\begin{align}
		\cT_n[F,r](\th)&=\left(h_n[r]\right)'(\th)\int_0^\th F'(\wt\th)\e^{nh_n[r](\th)-nh_n[r](\wt\th)}\,\mathrm d\wt\th\nonumber\\
		&=\left(h_n[r]\right)'(\th)\int_0^\th F'(\wt\th)\frac{\mu-\ii\th}{\mu-\ii\wt\th}\left(\frac{\wt\th}{\th}\right)^{\mu n-2\mu+1}\left(\frac{R[r](\th)}{R[r](\wt\th)}\right)^n\e^{\ii n(\th-\wt\th)}\,\mathrm d\wt\th.\label{Eq.T_nF,r_IBP}
	\end{align}
	By \eqref{Eq.h_n'_bound}, we have
	\begin{align}
		\left|\cT_n[F,r](\th)\right|&\lesssim\frac{\langle\th\rangle}{\th}\int_0^\th \langle\wt\th\rangle^{-1-\al}\frac{\langle\th\rangle}{\langle\wt\th\rangle} \left(\frac{\wt\th}{\th}\right)^{\mu n-2\mu+1}\left(\frac{R[r](\th)}{R[r](\wt\th)}\right)^n\,\mathrm d\wt\th\nonumber\\
		&\lesssim \langle\th\rangle^2\int_0^1\langle\th t\rangle^{-2-\al}t^{\mu n-2\mu+1} \left(\frac{R[r](\th)}{R[r](\th t)}\right)^n\,\mathrm dt\nonumber\\
		&\lesssim\frac{\langle\th\rangle^2}{\th^{2+\al}}\int_0^1 t^{\mu n-2\mu-1-\al} \left(\frac{R[r](\th)}{R[r](\th t)}\right)^n\,\mathrm dt.\label{Eq.T_nF,r_4.49}
	\end{align}
	Then we use an argument similar to the proof of \eqref{Eq.future_th_small}, i.e., make a change of variables $s:=t^\mu R[r](\th)/R[r](\th t)\in(0,1)$ for $t\in(0,1)$. Then we have $s\sim t^\mu$, $t\sim s^{1/\mu}$, $\mathrm ds/\mathrm dt\sim t^{\mu-1}\sim s^{1-1/\mu}$. Thus,
	\begin{align}
		\int_0^1 t^{\mu n-2\mu-1-\al} \left(\frac{R[r](\th)}{R[r](\th t)}\right)^n\,\mathrm dt&\lesssim \int_0^1s^{-(2\mu+1+\al)/\mu}s^ns^{1/\mu-1}\,\mathrm ds\nonumber\\
		&\lesssim \int_0^1s^{n-\frac{3\mu+\al}{\mu}}\,\mathrm ds\lesssim\frac1n\label{Eq.power_est4.50}
	\end{align} 
	for all $\th\in\R_+$, $n\geq 5$ and $r\in B_{X_r}(\varepsilon_1)$. Therefore, \eqref{Eq.T_nF,r_th_large1} follows from \eqref{Eq.T_nF,r_4.49} by noting that $\th>1$. 
	
	Next we prove \eqref{Eq.T_nF,r_diff_th_large1}. Recall that we assume $\norm{\langle\th\rangle^{1+\al}F'}_{L^\infty}=1$. Using \eqref{Eq.T_nF,r_IBP} gives that $\cT_n[F,r_1](\th)-\cT_n[F,r_2](\th)=\Delta_{n,1}^{(1)}+\Delta_{n,2}^{(1)}$, where
	\begin{align*}
		\Delta_{n,1}^{(1)}:&= \frac{\left(h_n[r_1]\right)'(\th)-\left(h_n[r_2]\right)'(\th)}{\left(h_n[r_1]\right)'(\th)}\cT_n[F, r_1](\th),\\
		\Delta_{n,2}^{(1)}:&=\left(h_n[r_2]\right)'(\th)\int_0^\th F'(\wt\th)\frac{\mu-\ii\th}{\mu-\ii\wt\th}\left(\frac{\wt\th}\th\right)^{\mu n-2\mu+1} \\
		&\qquad\qquad\times\left(\left(\frac{R[r_1](\th)}{R[r_1](\wt\th)}\right)^n-\left(\frac{R[r_2](\th)}{R[r_2](\wt\th)}\right)^n\right) \e^{\ii n(\th-\wt\th)}\,\mathrm d\wt\th.
	\end{align*}
	It follows from \eqref{Eq.h_n'_bound}, \eqref{Eq.h_n'diff} and \eqref{Eq.T_nF,r_th_large1} that $\left|\Delta_{n,1}^{(1)}\right|\lesssim \frac1{n\langle\th\rangle^\al}\norm{r_1-r_2}_{X_r}$ for $\th>1$, $n\geq N$ and $r_1, r_2\in B_{X_r}(\varepsilon_1)$. As for $\Delta_{n,2}^{(1)}$, by \eqref{Eq.h_n'_bound}, \eqref{Eq.R_diff}, \eqref{Eq.ln_diff}, we have
	\begin{align*}
		\left|\Delta_{n,2}^{(1)}\right|&\lesssim \frac{\langle\th\rangle}{\th}\int_0^\th \langle\wt\th\rangle^{-1-\al}\frac{\langle\th\rangle}{\langle\wt\th\rangle} \left(\frac{\wt\th}\th\right)^{\mu n-2\mu+1}n \left|\ln \frac{R[r_1](\th)}{R[r_1](\wt\th)}-\ln\frac{R[r_2](\th)}{R[r_2](\wt\th)}\right|\int_0^1\left(\frac{R_{(s)}(\th)}{R_{(s)}(\wt\th)}\right)^n\,\mathrm ds\,\mathrm d\wt\th\\
		&\lesssim n\langle\th\rangle^2\int_0^1\frac{1}{\langle\th t\rangle^{2+\al}}t^{\mu n-2\mu+1} \left|\ln \frac{R[r_1](\th)}{R[r_2](\th)}-\ln\frac{R[r_1](\th t)}{R[r_2](\th t)}\right|\int_0^1\left(\frac{R_{(s)}(\th)}{R_{(s)}(\th t)}\right)^n\,\mathrm ds\,\mathrm dt\\
		&\lesssim n\frac{\langle\th\rangle^2\norm{r_1-r_2}_{X_r}}{\th^{2+\al}}\int_0^1t^{\mu n-2\mu-1-\al}|\ln t| \int_0^1\left(\frac{R_{(s)}(\th)}{R_{(s)}(\th t)}\right)^n\,\mathrm ds\,\mathrm dt,
	\end{align*}
	where $R_{(s)}:=R[r_1]^sR[r_2]^{1-s}$ for $s\in(0,1)$. Using an argument similar to the proof of \eqref{Eq.Delta_n,2est}, we can show that
	\begin{align}\label{Eq.power_est4.51}
		\int_0^1 t^{\mu n-2\mu-1-\al}|\ln t| \left(\frac{R_{(s)}(\th)}{R_{(s)}(\th t)}\right)^n\,\mathrm dt\lesssim \int_0^1|\ln\sigma|\sigma^{n-\frac{3\mu+\al}{\mu}}\,\mathrm d\sigma\lesssim \frac{1}{n^2}
	\end{align}
	for all $\th\in\R_+$, $n\geq 5$ and $r_1, r_2\in B_{X_r}(\varepsilon_1)$, uniformly with respect to $s\in(0,1)$. Thus, $\left|\Delta_{n,2}^{(1)}\right|\lesssim \frac1{n\langle\th\rangle^\al}\norm{r_1-r_2}_{X_r}$ for $\th>1$, $n\geq N$ and $r_1, r_2\in B_{X_r}(\varepsilon_1)$. Combining this with the estimate of $\Delta_{n,1}^{(1)}$ gives \eqref{Eq.T_nF,r_diff_th_large1}.
\end{proof}

\begin{lemma}\label{Lem.T_nF,r_th_large2}
	There exist $\varepsilon_1\in(0,1)$, $N\in\Z_+$ and a constant $C=C(\mu,\al)>0$ such that
	\begin{align}
		\left|\cT_n[F,r](\th)\right|&\leq \frac{C}{\langle\th\rangle^\al}\norm{\langle\th\rangle^{1+\al}F}_{L^\infty},\label{Eq.T_nF,r_th_large2}\\
		\left|\cT_n[F, r_1](\th)-\cT_n[F, r_2](\th)\right|&\leq \frac{C}{\langle\th\rangle^\al}\norm{\langle\th\rangle^{1+\al}F}_{L^\infty}\norm{r_1-r_2}_{X_r},\label{Eq.T_nF,r_diff_th_large2}
	\end{align}
	for all $r, r_1, r_2\in B_{X_r}(\varepsilon_1)$, $n>N$, $\th>1$ and all $F\in C(\R_+; \C)$ satisfying $\langle\th\rangle^{1+\al}F\in L^\infty$.
\end{lemma}
\begin{proof}
	Note that $\cT_n$ is linear in $F$, then we can assume without loss of generality that $\norm{\langle\th\rangle^{1+\al}F}_{L^\infty}=1$. 
	By \eqref{Eq.T_nF,r_express} and \eqref{Eq.h_n'_bound}, we have
	\begin{align}
		\left|\cT_n[F,r](\th)\right|&\lesssim n\frac{\langle\th\rangle}{\th}\int_0^\th\left(\frac1{\langle\wt\th\rangle^{1+\al}}+\frac1{\langle\th\rangle^{1+\al}}\right)\frac{\langle\wt\th\rangle}{\wt\th}\frac{\langle\th\rangle}{\langle\wt\th\rangle}\left(\frac{\wt\th}{\th}\right)^{\mu n-2\mu+1}\left(\frac{R[r](\th)}{R[r](\wt\th)}\right)^n\,\mathrm d\wt\th\nonumber\\
		&\lesssim n\frac{\langle\th\rangle^2}{\th^2}\int_0^\th \frac1{\langle\wt\th\rangle^{1+\al}}\left(\frac{\wt\th}{\th}\right)^{\mu n-2\mu}\left(\frac{R[r](\th)}{R[r](\wt\th)}\right)^n\,\mathrm d\wt\th\nonumber\\
		&\lesssim n\frac{\langle\th\rangle^2}{\th}\int_0^1\frac1{\langle\th t\rangle^{1+\al}}t^{\mu n-2\mu}\left(\frac{R[r](\th)}{R[r](\th t)}\right)^n\,\mathrm dt\nonumber\\
		&\lesssim n\frac{\langle\th\rangle^2}{\th^{2+\al}}\int_0^1t^{\mu n-2\mu-1-\al}\left(\frac{R[r](\th)}{R[r](\th t)}\right)^n\,\mathrm dt.\label{Eq.T_nF,r_4.53}
	\end{align}
	Then \eqref{Eq.T_nF,r_th_large2} follows from \eqref{Eq.T_nF,r_4.53} and \eqref{Eq.power_est4.50} by noting that $\th>1$. The proof of \eqref{Eq.T_nF,r_diff_th_large2} is similar to the proof of \eqref{Eq.T_nF,r_diff_th_small}, and is based on \eqref{Eq.DeltaT_n}. It follows from \eqref{Eq.h_n'_bound}, \eqref{Eq.h_n'diff}, \eqref{Eq.T_nF,r_th_small}, \eqref{Eq.T_nF,r_4.53} and \eqref{Eq.power_est4.50} that
	\begin{align}
		\left|\Delta_{n,1}\right|&\lesssim \norm{r_1-r_2}_{X_r}\left|\cT_n[F, r_1](\th)\right|\nonumber\\
		&\qquad\qquad+ n\frac{\langle\th\rangle}{\th}\norm{r_1-r_2}_{X_r}\int_0^\th\left|F(\wt\th)-F(\th)\right| \frac{\langle\wt\th\rangle}{\wt\th}\frac{\langle\th\rangle}{\langle\wt\th\rangle} \left(\frac{\wt\th}{\th}\right)^{\mu n-2\mu+1} \left(\frac{R[r_1](\th)}{R[r_1](\wt\th)}\right)^n\,\mathrm d\wt\th\nonumber\\
		&\lesssim \frac{1}{\langle\th\rangle^\al} \norm{r_1-r_2}_{X_r}\nonumber
	\end{align}
	for all $r, r_1, r_2\in B_{X_r}(\varepsilon_1)$, $n>N$ and $\th>1$. Using the definition of $\Delta_{n,2}$, \eqref{Eq.h_n'_bound}, \eqref{Eq.R_diff} and \eqref{Eq.ln_diff}, we obtain
	\begin{align*}
		\left|\Delta_{n,2}\right|&\lesssim n\frac{\langle\th\rangle}{\th}\int_0^\th\frac{1}{\langle\wt\th\rangle^{1+\al}}\frac{\langle\wt\th\rangle}{\wt\th}\frac{\langle\th\rangle}{\langle\wt\th\rangle}\left(\frac{\wt\th}{\th}\right)^{\mu n-2\mu+1} \\
		&\qquad\qquad\times n\left|\ln \frac{R[r_1](\th)}{R[r_1](\wt\th)}-\ln\frac{R[r_2](\th)}{R[r_2](\wt\th)}\right|\int_0^1\left(\frac{R_{(s)}(\th)}{R_{(s)}(\wt\th)}\right)^n\,\mathrm ds\,\mathrm d\wt\th\\
		&\lesssim n^2\frac{\langle\th\rangle^2}{\th}\int_0^1\frac{t^{\mu n-2\mu}}{\langle\th t\rangle^{1+\al}}\left|\ln \frac{R[r_1](\th)}{R[r_2](\th)}-\ln\frac{R[r_1](\th t)}{R[r_2](\th t)}\right|\int_0^1\left(\frac{R_{(s)}(\th)}{R_{(s)}(\th t)}\right)^n\,\mathrm ds\,\mathrm dt\\
		&\lesssim n^2\frac{\langle\th\rangle^2\norm{r_1-r_2}_{X_r}}{\th^{2+\al}}\int_0^1t^{\mu n-2\mu-1-\al}|\ln t| \int_0^1\left(\frac{R_{(s)}(\th)}{R_{(s)}(\th t)}\right)^n\,\mathrm ds\,\mathrm dt\\
		&\lesssim \frac1{\langle\th\rangle^\al}\norm{r_1-r_2}_{X_r},
	\end{align*}
	where in the last step we have used \eqref{Eq.power_est4.51}. Combining this with \eqref{Eq.DeltaT_n} and the estimate of $\Delta_{n,1}$ gives \eqref{Eq.T_nF,r_diff_th_large2}.
\end{proof}

Putting Lemma \ref{Lem.decompose_G^1}, Lemma \ref{Lem.T_nF,r_th_large1} and Lemma \ref{Lem.T_nF,r_th_large2} altogether, we complete the estimate of $\cT_n[F,r]$ for $\th>1$ and $F\in\mathscr G^1$, as stated in the following lemma.
\begin{lemma}\label{Lem.T_nF,r_th_large}
	There exist $\varepsilon_1\in(0,1)$, $N\in\Z_+$ and a constant $C=C(\mu,\al)>0$ such that
	\begin{align}
		\left|\cT_n[F,r](\th)\right|&\leq C\frac{1}{n^\al\langle\th\rangle^\al}[F]_1,\label{Eq.T_nF,r_th_large}\\
		\left|\cT_n[F,r_1](\th)-\cT_n[F, r_2](\th)\right|&\leq C \frac{1}{n^\al\langle\th\rangle^\al}[F]_1\norm{r_1-r_2}_{X_r},\label{Eq.T_nF,r_diff_th_large}
	\end{align} 
	for all $r, r_1, r_2\in B_{X_r}(\varepsilon_1)$, $n>N$, $F\in\mathscr G^1$ and $\th>1$.
\end{lemma}
\begin{proof}
	Given $F\in\mathscr G^1$ and $n\in\Z_+$, by Lemma \ref{Lem.decompose_G^1}, there exist $F_0\in C(\R_+; \C)$ and $F_1\in C^1(\R_+; \C)$ such that
	\begin{equation*}
		\norm{\langle\th\rangle^{1+\al}F_0}_{L^\infty}\leq \frac{n^{-\al}}{\al+1}[F]_1,\qquad \norm{\langle\th\rangle^{1+\al}F_1'}_{L^\infty}\leq n^{1-\al}[F]_1.
	\end{equation*}
	Note that $\cT_n$ is linear in $F$, hence it follows from \eqref{Eq.T_nF,r_diff_th_large1} and \eqref{Eq.T_nF,r_diff_th_large2} that
	\begin{align*}
		\left|\cT_n[F, r](\th)\right|&\lesssim \left|\cT_n[F_1, r](\th)\right|+\left|\cT_n[F_0, r](\th)\right|\\
		&\lesssim \frac1{n\langle\th\rangle^\al}\norm{\langle\th\rangle^{1+\al}F_1'}_{L^\infty}+\frac1{\langle\th\rangle^\al}\norm{\langle\th\rangle^{1+\al}F_0}_{L^\infty}\lesssim \frac1{n^\al\langle\th\rangle^\al}[F]_1
	\end{align*}
	for all $r\in B_{X_r}(\varepsilon_1)$, $n>N$, $F\in\mathscr G^1$ and $\th>1$.
	This proves \eqref{Eq.T_nF,r_th_large}. To show \eqref{Eq.T_nF,r_diff_th_large}, using \eqref{Eq.T_nF,r_diff_th_large1} and \eqref{Eq.T_nF,r_diff_th_large2}, we obtain
	\begin{align*}
		&\left|\cT_n[F,r_1](\th)-\cT_n[F, r_2](\th)\right|\lesssim \left|\cT_n[F_1,r_1](\th)-\cT_n[F_1, r_2](\th)\right|+\left|\cT_n[F_0,r_1](\th)-\cT_n[F_0, r_2](\th)\right|\\
		\lesssim&\  \frac{1}{n\langle\th\rangle^\al}\norm{\langle\th\rangle^{1+\al}F_1'}_{L^\infty}\norm{r_1-r_2}_{X_r}+\frac{1}{\langle\th\rangle^\al}\norm{\langle\th\rangle^{1+\al}F_0}_{L^\infty}\norm{r_1-r_2}_{X_r}\lesssim
		 \frac{[F]_1}{n^\al\langle\th\rangle^\al}\norm{r_1-r_2}_{X_r}
	\end{align*}
	for all $r_1, r_2\in B_{X_r}(\varepsilon_1)$, $n>N$, $F\in\mathscr G^1$ and $\th>1$.
\end{proof}

Now, we are ready to prove \eqref{Eq.C^1est_1} and \eqref{Eq.C^1est_diff_1}.

\begin{proof}[Proof of \eqref{Eq.C^1est_1} and \eqref{Eq.C^1est_diff_1}]
	Combining Lemma \ref{Lem.T_nF,r_th_small} and Lemma \ref{Lem.T_nF,r_th_large}, we obtain
	\begin{align*}
		\left|\cT_n[F,r](\th)\right|&\leq C\frac{\th^{\al-1}\langle\th\rangle^{1-2\al}}{n^\al}[F]_1,\\
		\left|\cT_n[F,r_1](\th)-\cT_n[F, r_2](\th)\right|&\leq C \frac{\th^{\al-1}\langle\th\rangle^{1-2\al}}{n^\al}[F]_1\norm{r_1-r_2}_{X_r},
	\end{align*}
	for all $r\in B_{X_r}(\varepsilon_1)$, $n>N$, $F\in\mathscr G^1$ and $\th\in\R_+$. Therefore, \eqref{Eq.C^1est_1} follows directly from \eqref{Eq.D_n_G^1bound} and \eqref{Eq.P_n,1'=T_n}. By \eqref{Eq.D_n_G^1bound}, \eqref{Eq.P_n,1'=T_n} and the fact that $\cT_n$ is linear in $F$, we have
	\begin{align*}
		&\left|\left(\cP_{n,1}[r_1,\g_1]\right)'(\th)-\left(\cP_{n,1}[r_2,\g_2]\right)'(\th)\right|=\left|\cT_n\left[\cD_n[r_1, \g_1], r_1\right](\th)-\cT_n\left[\cD_n[r_2, \g_2], r_2\right](\th)\right|\\
		\leq &\ \left|\cT_n\left[\cD_n[r_1, \g_1], r_1\right](\th)-\cT_n\left[\cD_n[r_1, \g_1], r_2\right](\th)\right|+\left|\cT_n\left[\cD_n[r_1, \g_1]-\cD_n[r_2, \g_2], r_2\right](\th)\right|\\
		\lesssim &\ \frac{\th^{\al-1}\langle\th\rangle^{1-2\al}}{n^\al}\norm{\cD_n[r_1, \g_1]}_1\norm{r_1-r_2}_{X_r}+\frac{\th^{\al-1}\langle\th\rangle^{1-2\al}}{n^\al}\norm{\cD_n[r_1, \g_1]-\cD_n[r_2, \g_2]}_1\\
		\lesssim &\ \frac{\th^{\al-1}\langle\th\rangle^{1-2\al}}{n^\al}\norm{(r_1,\g_1)-(r_2,\g_2)}_X,
	\end{align*}
	where in the last step we have used the fact that $\cD_n: B_X(\varepsilon_1)\to\mathscr G^1$ has a uniform $C^1$ bound with respect to $n>N$. This proves \eqref{Eq.C^1est_diff_1}.
\end{proof}

To complete the proof of Proposition \ref{Prop.P_n,1,2}, it remains to show \eqref{Eq.C^1est_2} and \eqref{Eq.C^1est_diff_2}. The proof is similar to that of \eqref{Eq.C^1est_1} and \eqref{Eq.C^1est_diff_1}.

By \eqref{Eq.P_n,2} and \eqref{Eq.phi_n}, we have $\lim_{\th\to+\infty}\operatorname{Re}\phi_n[r](\th)=-\infty$, and
\begin{align}
 &\left(\cP_{n,2}[r,\g]\right)'(\th)=n\int_\th^\infty \left(\frac{-\mu+\ii\th}{\th}\Gamma[\g](\th) \left(\phi_n[r]\right)'(\wt\th)-\frac{-\mu+\ii\wt\th}{\wt\th}\Gamma[\g](\wt\th)\left(\phi_n[r]\right)'(\th)\right) \nonumber\\
	&\qquad\qquad\qquad\qquad\qquad\times\e^{n\phi_n[r](\wt\th)-n\phi_n[r](\th)}\,\mathrm d\wt\th\nonumber\\
	&\quad=n\int_\th^\infty\left(\cE_n[r,\g](\th)-\cE_n[r,\g](\wt\th)\right) \left(\phi_n[r]\right)'(\th) \left(\phi_n[r]\right)'(\wt\th) \e^{n\phi_n[r](\wt\th)-n\phi_n[r](\th)}\,\mathrm d\wt\th,\label{Eq.P_n,2'_express}
\end{align}
where
\begin{equation}
	\cE_n[r,\g](\th):=\frac{(-\mu+\ii\th)\Gamma[\g](\th)}{\th \left(\phi_n[r]\right)'(\th)}-1=\frac{\Gamma[\g](\th)}{\Phi_n[r](\th)}-1.
\end{equation}

Using Lemma \ref{Lem.H_n_Phi_n}, \eqref{Eq.Gamma_gamma} and Banach algebraic properties, similar to \eqref{Eq.D_n_G^1bound}, one can show that there exist $\varepsilon_1\in(0,1)$ and $N\in\Z_+$ such that the nonlinear map $\cE_n: B_X(\varepsilon_1)\to \mathscr G^1$ is smooth for all $n>N$ with a uniform $C^1$ bound with respect to $n>N$, and
\begin{equation}\label{Eq.E_n_G^1bound}
	\norm{\cE_n[r,\g]}_1\leq C\left(1/n+\norm{(r,\g)}_X\right),\quad\forall\ n>N,\ \forall\ (r,\g)\in B_X(\varepsilon_1),
\end{equation}
where $C=C(\mu,\al)>0$ is a constant independent of $n>N$ and $(r,\g)\in B_X(\varepsilon_1)$. As a consequence, we have
\begin{equation}\label{Eq.phi_n'_bound}
	\left|\left(\phi_n[r]\right)'(\th)\right|\sim\frac{\langle\th\rangle}{\th},\qquad\forall\ \th>0,
\end{equation}
for all $n>N$ and $r\in B_{X_r}(\varepsilon_1)$.

Hence, it suffices to analyze the properties of $\cS_n[F, r]$ defined by
\begin{align*}
	\cS_n[F,r](\th):=n\left(\phi_n[r]\right)'(\th)\int_\th^\infty\left(F(\th)-F(\wt\th)\right) \left(\phi_n[r]\right)'(\wt\th)\e^{n\phi_n[r](\wt\th)-n\phi_n[r](\th)}\,\mathrm d\wt\th
\end{align*}
for $F\in\mathscr G^1$ and $r\in X_r$ near $0$. The operators $\cS_n$ satisfy the same properties as $\cT_n$, listed in Lemma \ref{Lem.T_nF,r_th_small}, Lemma \ref{Lem.T_nF,r_th_large1} and Lemma \ref{Lem.T_nF,r_th_large2}. Namely, we have the following three lemmas for $\cS_n$.

\begin{lemma}\label{Lem.S_nF,r_th_small}
	There exist $\varepsilon_1\in(0,1)$, $N\in\Z_+$ and a constant $C=C(\mu,\al)>0$ such that
	\begin{align}
		\left|\cS_n[F,r](\th)\right|&\leq C\frac{\th^{\al-1}}{n^\al}\|F\|_1,\label{Eq.S_nF,r_th_small}\\
		\left|\cS_n[F,r_1](\th)-\cS_n[F, r_2](\th)\right|&\leq C \frac{\th^{\al-1}}{n^\al}\|F\|_1\norm{r_1-r_2}_{X_r},\label{Eq.S_nF,r_diff_th_small}
	\end{align} 
	for all $r, r_1, r_2\in B_{X_r}(\varepsilon_1)$, $n>N$, $F\in\mathscr G^1$ and $\th\in(0,1]$.
\end{lemma}

\begin{lemma}\label{Lem.S_nF,r_th_large1}
	There exist $\varepsilon_1\in(0,1)$, $N\in\Z_+$ and a constant $C=C(\mu,\al)>0$ such that
	\begin{align}
		\left|\cS_n[F,r](\th)\right|&\leq \frac{C}{n\langle\th\rangle^\al}\norm{\langle\th\rangle^{1+\al}F'}_{L^\infty}, \label{Eq.S_nF,r_th_large1}\\
		\left|\cS_n[F, r_1](\th)-\cS_n[F, r_2](\th)\right|&\leq \frac{C}{n\langle\th\rangle^\al}\norm{\langle\th\rangle^{1+\al}F'}_{L^\infty}\norm{r_1-r_2}_{X_r},\label{Eq.S_nF,r_diff_th_large1}
	\end{align}
	for all $r, r_1, r_2\in B_{X_r}(\varepsilon_1)$, $n>N$, $\th>1$ and all $F\in C^1(\R_+; \C)$ satisfying $\langle\th\rangle^{1+\al}F'\in L^\infty$.
\end{lemma}

\begin{lemma}\label{Lem.S_nF,r_th_large2}
	There exist $\varepsilon_1\in(0,1)$, $N\in\Z_+$ and a constant $C=C(\mu,\al)>0$ such that
	\begin{align}
		\left|\cS_n[F,r](\th)\right|&\leq \frac{C}{\langle\th\rangle^\al}\norm{\langle\th\rangle^{1+\al}F}_{L^\infty},\label{Eq.S_nF,r_th_large2}\\
		\left|\cS_n[F, r_1](\th)-\cS_n[F, r_2](\th)\right|&\leq \frac{C}{\langle\th\rangle^\al}\norm{\langle\th\rangle^{1+\al}F}_{L^\infty}\norm{r_1-r_2}_{X_r},\label{Eq.S_nF,r_diff_th_large2}
	\end{align}
	for all $r, r_1, r_2\in B_{X_r}(\varepsilon_1)$, $n>N$, $\th>1$ and all $F\in C(\R_+; \C)$ satisfying $\langle\th\rangle^{1+\al}F\in L^\infty$.
\end{lemma}

The proofs of Lemma \ref{Lem.S_nF,r_th_large1} and Lemma \ref{Lem.S_nF,r_th_large2} follow a closely analogous approach to those established for Lemma \ref{Lem.T_nF,r_th_large1} and Lemma \ref{Lem.T_nF,r_th_large2}, respectively. To ensure methodological transparency, we present detailed proofs of  \eqref{Eq.S_nF,r_th_large1} and \eqref{Eq.S_nF,r_th_large2} here. The proofs of \eqref{Eq.S_nF,r_diff_th_large1} and \eqref{Eq.S_nF,r_diff_th_large2}, being structurally parallel to the preceding arguments, are omitted here for conciseness.

\begin{proof}[Proof of \eqref{Eq.S_nF,r_th_large1}]
	We assume without loss of generality that $\norm{\langle\th\rangle^{1+\al}F'}_{L^\infty}=1$. An integration by parts yields
	\begin{align*}
		\cS_n[F, r](\th)&=\left(\phi_n[r]\right)'(\th)\int_\th^\infty F'(\wt\th)\e^{n\phi_n[r](\wt\th)-n\phi_n[r](\th)}\,\mathrm d\wt\th\\
		&=\left(\phi_n[r]\right)'(\th)\int_\th^\infty F'(\wt\th)\frac{\mu-\ii\th}{\mu-\ii\wt\th}\left(\frac{\wt\th}{\th}\right)^{-\mu n-2\mu+1}\left(\frac{R[r](\wt\th)}{R[r](\th)}\right)^n\e^{\ii n(\wt\th-\th)}\,\mathrm d\wt\th.
	\end{align*}
	Hence, by \eqref{Eq.phi_n'_bound} we have
	\begin{align*}
		\left|\cS_n[F, r](\th)\right|&\lesssim \frac{\langle\th\rangle}{\th}\int_\th^\infty \frac{1}{\langle\wt\th\rangle^{1+\al}}\frac{\langle\th\rangle}{\langle\wt\th\rangle}\left(\frac{\wt\th}{\th}\right)^{-\mu n-2\mu+1}\left(\frac{R[r](\wt\th)}{R[r](\th)}\right)^n\,\mathrm d\wt\th\\
		&\lesssim \langle\th\rangle^2\int_1^\infty \frac1{\langle\th t\rangle^{2+\al}}t^{-\mu n-2\mu+1}\left(\frac{R[r](\th t)}{R[r](\th)}\right)^n\,\mathrm dt\\
		&\lesssim \frac{\langle\th\rangle^2}{\th^{2+\al}}\int_1^\infty t^{-\mu n-2\mu-1-\al}\left(\frac{R[r](\th t)}{R[r](\th)}\right)^n\,\mathrm dt.
	\end{align*}
	Then we use an argument similar to the proof of \eqref{Eq.future_th_small}, i.e., make a change of variables $s:=t^{-\mu} R[r](\th t)/R[r](\th)\in(0,1)$ for $t\in(1,+\infty)$. Then we have $s\sim t^{-\mu}$, $t\sim s^{-1/\mu}$, $\mathrm ds/\mathrm dt\sim t^{-\mu-1}\sim s^{1+1/\mu}$. Thus,
	\begin{align}\label{Eq.power_est4.66}
		\int_1^\infty t^{-\mu n-2\mu-1-\al}\left(\frac{R[r](\th t)}{R[r](\th)}\right)^n\,\mathrm dt\lesssim \int_0^1 s^ns^{\frac{2\mu+1+\al}{\mu}}s^{-1-\frac1\mu}\,\mathrm ds\lesssim \frac1n
	\end{align}
	for all $\th\in\R_+$, $n\in\Z_+$ and $r\in B_{X_r}(\varepsilon_1)$. Therefore, \eqref{Eq.S_nF,r_th_large1} follows.
\end{proof}

\begin{proof}[Proof of \eqref{Eq.S_nF,r_th_large2}]
	By the definition of $\cS_n$ and \eqref{Eq.phi_n}, we have
	\begin{align}
		\cS_n[F, r](\th)&=n\left(\phi_n[r]\right)'(\th)\int_\th^\infty\left(F(\th)-F(\wt\th)\right) \left(\phi_n[r]\right)'(\wt\th)\frac{\mu-\ii\th}{\mu-\ii\wt\th}\left(\frac{\wt\th}{\th}\right)^{-\mu n-2\mu+1}\label{Eq.S_nF,r_express}\\
		&\qquad\qquad\times\left(\frac{R[r](\wt\th)}{R[r](\th)}\right)^n\e^{\ii n(\wt\th-\th)}\,\mathrm d\wt\th.\nonumber
	\end{align}
	We assume without loss of generality that $\norm{\langle\th\rangle^{1+\al}F}_{L^\infty}=1$. Hence,
	\begin{align*}
		\left|\cS_n[F, r](\th)\right|&\lesssim n\frac{\langle\th\rangle}{\th}\int_\th^\infty\frac1{\langle\th\rangle^{1+\al}}\frac{\langle\wt\th\rangle}{\wt\th}\frac{\langle\th\rangle}{\langle\wt\th\rangle}\left(\frac{\wt\th}{\th}\right)^{-\mu n-2\mu+1}\left(\frac{R[r](\wt\th)}{R[r](\th)}\right)^n\,\mathrm d\wt\th\\
		&\lesssim n\frac{\langle\th\rangle^{1-\al}}{\th}\int_1^\infty t^{-\mu n-2\mu}\left(\frac{R[r](\th t)}{R[r](\th)}\right)^n\,\mathrm dt.
	\end{align*}
	Similar to \eqref{Eq.power_est4.66}, we have
	\[\int_1^\infty t^{-\mu n-2\mu}\left(\frac{R[r](\th t)}{R[r](\th)}\right)^n\,\mathrm dt\lesssim \int_0^1s^ns^2s^{-1-\frac1\mu}\,\mathrm ds\lesssim\frac1n\]
	for all $\th\in\R_+$, $n\geq 2$ and $r\in B_{X_r}(\varepsilon_1)$. Hence \eqref{Eq.S_nF,r_th_large2} follows.
\end{proof}

The proof of Lemma \ref{Lem.S_nF,r_th_small} is analogous to that of Lemma \ref{Lem.T_nF,r_th_small}, with only minor modifications. For readers' convenience, we provide the proof of Lemma \ref{Lem.S_nF,r_th_small} below.

\begin{proof}[Proof of Lemma \ref{Lem.S_nF,r_th_small}]
	For $F\in\mathscr G^1$, we claim
	\begin{equation}\label{Eq.F_diff_bound}
		\left|F(\th)-F(\wt\th)\right|\leq 2\|F\|_1\left|\th-\wt\th\right|^\al,\quad \forall\ \th, \wt\th>0.
	\end{equation}
	Indeed, we assume without loss of generality that $\wt\th>\th$, if $\left|\th-\wt\th\right|\leq 1$, then 
	\[\left|F(\th)-F(\wt\th)\right|\leq [F]_1 \langle\wt\th\rangle^{-1-\al}\left|\th-\wt\th\right|^\al\leq [F]_1\left|\th-\wt\th\right|^\al;\]
	if $\left|\th-\wt\th\right|>1$, then
	\[\left|F(\th)-F(\wt\th)\right|\leq\left|F(\th)\right|+\left|F(\wt\th)\right|\leq \langle\th\rangle^{-\al}\|F\|_1+\langle\wt\th\rangle^{-\al}\|F\|_1\leq 2\|F\|_1\leq 2\|F\|_1\left|\th-\wt\th\right|^\al.\]
	This proves \eqref{Eq.F_diff_bound}. Therefore, 
	using \eqref{Eq.S_nF,r_express}, \eqref{Eq.phi_n'_bound}, $F\in\mathscr G^1$, \eqref{Eq.F_diff_bound} and $\th\in(0,1]$, we obtain
	\begin{align*}
		\left|\cS_n[F, r](\th)\right|&\lesssim n\frac{\langle\th\rangle}{\th}\int_\th^\infty\left|\th-\wt\th\right|^\al\|F\|_1\frac{\langle\wt\th\rangle}{\wt\th}\frac{\langle\th\rangle}{\langle\wt\th\rangle}\left(\frac{\wt\th}{\th}\right)^{-\mu n-2\mu+1}\left(\frac{R[r](\wt\th)}{R[r](\th)}\right)^n\,\mathrm d\wt\th\\
		&\lesssim n\frac{\langle\th\rangle^2}{\th^2}\|F\|_1\int_\th^\infty\left|\th-\wt\th\right|^\al\left(\frac{\wt\th}{\th}\right)^{-\mu n-2\mu}\left(\frac{R[r](\wt\th)}{R[r](\th)}\right)^n\,\mathrm d\wt\th\\
		&\lesssim n\frac{\langle\th\rangle^2}{\th^{1-\al}}\|F\|_1\int_1^\infty(t-1)^\al t^{-\mu n-2\mu}\left(\frac{R[r](\th t)}{R[r](\th)}\right)^n\,\mathrm dt\\
		&\lesssim n\frac{\langle\th\rangle^2}{\th^{1-\al}}\|F\|_1\int_1^\infty\left(1-\frac1t\right)^\al t^{-\mu n-2\mu+\al}\left(\frac{R[r](\th t)}{R[r](\th)}\right)^n\,\mathrm dt\\
		&\lesssim n\frac{\langle\th\rangle^2}{\th^{1-\al}}\|F\|_1\int_1^\infty|\ln t|^\al t^{-\mu n-2\mu+\al}\left(\frac{R[r](\th t)}{R[r](\th)}\right)^n\,\mathrm dt,
	\end{align*}
	where in the last inequality we have used $|1-1/t|\leq |\ln(1/t)|=\ln t$ for $t>1$. Then we use an argument similar to the proof of \eqref{Eq.future_th_small}, i.e., make a change of variables $s:=t^{-\mu} R[r](\th t)/R[r](\th)\in(0,1)$ for $t\in(1,+\infty)$. Then we have $s\sim t^{-\mu}$, $t\sim s^{-1/\mu}$, $\mathrm ds/\mathrm dt\sim t^{-\mu-1}\sim s^{1+1/\mu}$ and $|\ln t|\sim |\ln s|$. 
	Thus,
	\begin{align*}
		\int_1^{\infty}|\ln t|^\al t^{-\mu n-2\mu+\al}\left(\frac{R[r](\th t)}{R[r](\th)}\right)^n\,\mathrm dt\lesssim \int_0^1|\ln s|^\al s^n s^{\frac{2\mu-\al}{\mu}}s^{-1-\frac1\mu}\,\mathrm ds\lesssim \frac1{n^{1+\al}}
	\end{align*}
	for all $\th\in\R_+$, $n\geq 4$ and $r\in B_{X_r}(\varepsilon_1)$. Hence
	\begin{equation*}
		\left|\cS_n[F, r](\th)\right|\lesssim \frac{\langle\th\rangle^2\th^{\al-1}}{n^\al}\|F\|_1\lesssim \frac{\th^{\al-1}}{n^\al}\|F\|_1, \quad\forall\ n>N, r\in B_{X_r}(\varepsilon_1), F\in\mathscr G^1, \th\in(0,1].
	\end{equation*}
	This proves \eqref{Eq.S_nF,r_th_small}. The proof of \eqref{Eq.S_nF,r_diff_th_small} shares the same idea with the proof of \eqref{Eq.T_nF,r_diff_th_small}, using \eqref{Eq.F_diff_bound}. Hence, we leave the proof of \eqref{Eq.S_nF,r_diff_th_small} to readers.
\end{proof}

Lemma \ref{Lem.S_nF,r_th_large1}, Lemma \ref{Lem.S_nF,r_th_large2} and Lemma \ref{Lem.decompose_G^1} provide the estimates of $\cS_n$ for $\th>1$. Combining these estimates with those for $\th\in (0,1]$, i.e., Lemma \ref{Lem.S_nF,r_th_small}, and using the fact $\left(\cP_{n,2}[r,\g]\right)'=\cS_n[\cE_n[r,\g], r]$ and \eqref{Eq.E_n_G^1bound}, we can prove \eqref{Eq.C^1est_2} and \eqref{Eq.C^1est_diff_2}. The details follow a similar approach to the proof of \eqref{Eq.C^1est_1} and \eqref{Eq.C^1est_diff_1}, and are thus omitted here for brevity.

Now, we have finished the proof of Proposition \ref{Prop.P_n,1,2}.

\subsection{Proof of Proposition \ref{Prop.I}}\label{Subsec.Proof-contraction}
So far, we have proved Proposition \ref{Prop.P_n_est}. The main task in this subsection is to prove Proposition \ref{Prop.I} by using Proposition \ref{Prop.P_n_est}.

We define $\cJ_m:=(-\mu+\ii\th)\th^{2\mu-1}\cI_m$, then the estimates of $\cI_m$ in the space $Y$ are reduced to the estimates of $\cJ_m$ in $\mathscr G^0$, as the following lemma illustrates.

\begin{lemma}\label{Lem.I_bound_by_J}
    Let $\mu>1/2$ and $\al\in(0,1)$. Let $f\in C(\R_+; \C)$ be such that $(-\mu+\emph{\ii}\th)\th^{2\mu-1}f\in\mathscr G^0$, then we have $\th^{2\mu}\operatorname{Re}f\in\cG^0_0$ and $\th^{2\mu-1}\operatorname{Im}f\in \cG^1$, with
    \begin{equation}\label{Eq.I_bound_by_J}
    	\norm{\th^{2\mu}\operatorname{Re}f}_0+\norm{\th^{2\mu-1}\operatorname{Im}f}_1 \leq C \norm{(-\mu+\emph{\ii}\th)\th^{2\mu-1}f}_0,
    \end{equation}
    where $C=C(\mu,\al)>0$ is a constant depending only on $\mu,\al$.
\end{lemma}

\begin{proof}
	Let $g:=(-\mu+\ii\th)\th^{2\mu-1}f\in\mathscr G^1$, then
	\begin{equation}\label{Eq.f_express_g}
		\th^{2\mu}f(\th)=\frac{\th g(\th)}{-\mu+\ii\th},\quad \th^{2\mu-1}f(\th)=\frac{g(\th)}{-\mu+\ii\th},\quad\forall\ \th>0.
	\end{equation}
	Using Lemma \ref{Lem.seminorm}, one readily checks that $\th\mapsto \frac{\th}{-\mu+\ii\th}\in \mathscr G^1\subset\mathscr G^0$, hence 
Lemma \ref{Lem.algebra} implies that $\th^{2\mu}f\in\mathscr G^0$ and $\norm{\th^{2\mu}f}_0\lesssim \|g\|_0$. The fact that $\lim_{\th\to0+} \th^{2\mu}f(\th)=0$ is obvious from \eqref{Eq.f_express_g}. Next we prove that $\th^{2\mu-1}f\in\mathscr G^1$. Clearly, $\norm{\langle\th\rangle^\al \th^{2\mu-1}f}_{L^\infty}\lesssim \norm{\langle\th\rangle^\al g}_{L^\infty}\lesssim\|g\|_0$. It suffices to show that
\begin{equation}\label{Eq.f_G^1_boundby_g}
	\left[\th^{2\mu-1}f\right]_1\lesssim \|g\|_0.
\end{equation}
Indeed, for $0<\th_2<\th_1<\th_2+1$, by \eqref{Eq.f_express_g}, we have
\begin{align*}
	\left|\th_1^{2\mu-1}f(\th_1)-\th_2^{2\mu-1}f(\th_2)\right|&\leq\frac{|g(\th_1)-g(\th_2)|}{|-\mu+\ii\th_1|}+\frac{|g(\th_2)| |\th_1-\th_2|}{|-\mu+\ii\th_1||-\mu+\ii\th_2|}\\
	&\lesssim [g]_0\frac{|\th_1-\th_2|^\al}{\langle\th_1\rangle^{1+\al}}+\norm{\langle\th\rangle^\al g}_{L^\infty}\frac{|\th_1-\th_2|^\al}{\langle\th_1\rangle^{1+\al}},
\end{align*}
where we have used $\langle\th_2\rangle\sim\langle\th_1\rangle$, $|-\mu+\ii\th_2|\gtrsim1$ and $|\th_1-\th_2|\le |\th_1-\th_2|^\al$ since $|\th_1-\th_2|<1$ and $\al\in(0,1)$. Thus,\eqref{Eq.f_G^1_boundby_g} follows.
\end{proof}

By Lemma \ref{Lem.I_bound_by_J} and the definition of $Y$ in \eqref{Eq.Y}, we have
\begin{align}
\norm{\cI_m[0,0]}_Y	&\lesssim\norm{\cJ_m[0,0]}_0,\label{Eq.I_bound_by_J1}\\
\norm{\cI_m[r_1,\g_1]-\cI_m[r_2, \g_2]}_Y&\lesssim \norm{\cJ_m[r_1,\g_1]-\cJ_m[r_2, \g_2]}_0. \label{Eq.I_bound_by_J2}
\end{align}
Therefore, Proposition \ref{Prop.I} is a direct consequence of \eqref{Eq.I_bound_by_J1}, \eqref{Eq.I_bound_by_J2}, and the following proposition concerning  the estimates of $\cJ_m$.

\begin{proposition}\label{Prop.J}
	There exist $\varepsilon_1\in (0,1)$, $m_1\in\Z_+$ and a constant $C_*>1$ such that for all $m>m_1$, we have
	\begin{align}
		&\left\|\mathcal J_m[0,0]\right\|_0\leq \frac{C_*}{m^2},\label{Eq.J00}\\
		\left\|\mathcal J_m[r_1,\g_1]-\mathcal J_m[r_{2}, \g_2]\right\|_0&\leq \frac{C_*}{m}\|(r_1, \g_1)-(r_2, \g_2)\|_X,\ \ \forall\ (r_1,\g_1), (r_2,\g_2)\in B_X(\varepsilon_1).\label{Eq.J_Lip}
	\end{align}
\end{proposition}
\begin{proof}
	By the definition of $\cJ_m$, \eqref{Eq.I_series} and \eqref{Eq.P_n}, we have
	\begin{equation}\label{Eq.J_series}
		\cJ_m[r,\g]=-(2\mu-1)\ii\sum_{n=1}^\infty \cP_{mn}[r,\g].
	\end{equation}
	By Proposition \ref{Prop.P_n_est} and \eqref{Eq.P_n_decompose0}, there exist $\varepsilon_1\in(0,1)$ and $N\in\N_+$ such that for all $n>N$, and all $(r,\g), (r_1, \g_1), (r_2, \g_2)\in B_X(\varepsilon_1)$, we have 
	\begin{align}
		\norm{\langle\th\rangle^\al\cP_n[r,\g]}_{L^\infty}\lesssim \frac1{n^{1+\al}}\left(\frac1n+\norm{(r,\g)}_X\right)+\frac1{n^2}\lesssim \frac1{n^2}+ \frac1{n^{1+\al}}\norm{(r,\g)}_X,\label{Eq.P_nC^0}\\
		\norm{\langle\th\rangle^\al\left(\cP_n[r_1,\g_1]-\cP_n[r_2,\g_2]\right)}_{L^\infty}\lesssim \frac1{n^{1+\al}}\norm{(r_1,\g_1)-(r_2,\g_2)}_X;\label{Eq.DeltaP_nC^0}
	\end{align}
	and moreover, for $0<\th_2<\th_1<\th_2+1$, we have ($\Delta\cP_n:=\cP_n[r_1,\g_1]-\cP_n[r_2,\g_2]$, $\Delta\cP_n^{(1)}:=\cP_n^{(1)}[r_1,\g_1]-\cP_n^{(1)}[r_2,\g_2]$, $\Delta\cP_n^{(2)}:=\cP_n^{(2)}[r_1,\g_1]-\cP_n^{(2)}[r_2,\g_2]$)
	\begin{align}
		&\left|\cP_n[r,\g](\th_1)-\cP_n[r,\g](\th_2)\right|\nonumber\\
		\leq &\ \left|\cP_n^{(1)}[r,\g](\th_1)\right|+ \left|\cP_n^{(1)}[r,\g](\th_2)\right|+ \left|\cP_n^{(2)}[r,\g](\th_1)-\cP_n^{(2)}[r,\g](\th_2)\right|\nonumber\\
		\lesssim &\ \frac1{n^{1+\al}}\frac{\th_1^\al}{\langle\th_1\rangle^{2\al}}\left(\frac1n+\norm{(r,\g)}_X\right)+\frac1{n^2}\frac1{\langle\th_1\rangle^{1+\al}}|\th_1-\th_2|^\al,\label{Eq.P_nG^1_1}\\
		&\left|\cP_n[r,\g](\th_1)-\cP_n[r,\g](\th_2)\right|\leq \int_{\th_2}^{\th_1}\left|\left(\cP_n[r,\g]\right)'(\wt\th)\right|\,\mathrm d\wt\th\nonumber\\
		\lesssim &\ \frac{\langle\th_1\rangle^{1-2\al}}{n^\al}\left(\frac1n+\norm{(r,\g)}_X\right)\int_{\th_2}^{\th_1}\wt\th^{\al-1}\,\mathrm d\wt\th\nonumber\\
		\lesssim &\  \frac1{n^\al}\left(\frac1n+\norm{(r,\g)}_X\right)\langle\th_1\rangle^{1-2\al}\th_1^{\al-1}|\th_1-\th_2|,\label{Eq.P_nG^1_2}\\
		&\left|\Delta \cP_n(\th_1)-\Delta \cP_n(\th_2)\right|\leq \left|\Delta\cP_n^{(1)}(\th_1)\right|+\left|\Delta\cP_n^{(1)}(\th_2)\right|+\left|\Delta \cP_n^{(2)}(\th_1)-\Delta \cP_n^{(2)}(\th_2)\right|\nonumber\\
		\lesssim &\ \frac1{n^{1+\al}}\frac{\th_1^\al}{\langle\th_1\rangle^{2\al}}\norm{(r_1,\g_1)-(r_2,\g_2)}_X+\frac{|\th_1-\th_2|^\al}{\langle\th_1\rangle^{1+\al}}\frac1{n^2}\norm{(r_1,\g_1)-(r_2,\g_2)}_X,\label{Eq.DeltaP_nG^1_1}\\
		&\left|\Delta \cP_n(\th_1)-\Delta \cP_n(\th_2)\right|\leq\int_{\th_2}^{\th_1}\left|\left(\Delta\cP_n\right)'(\wt\th)\right|\,\mathrm d\wt\th\nonumber\\
		\lesssim&\ \frac{\langle\th_1\rangle^{1-2\al}}{n^\al}\norm{(r_1,\g_1)-(r_2,\g_2)}_X\int_{\th_2}^{\th_1}\wt\th^{\al-1}\,\mathrm d\wt\th\nonumber\\
		\lesssim&\ \frac{\langle\th_1\rangle^{1-2\al}}{n^\al}\norm{(r_1,\g_1)-(r_2,\g_2)}_X\th_1^{\al-1}|\th_1-\th_2|,\label{Eq.DeltaP_nG^1_2}
	\end{align}
	where in \eqref{Eq.P_nG^1_2} and \eqref{Eq.DeltaP_nG^1_2} we have used the fact that
	\begin{equation}\label{Eq.claim4.80}
		\left|\th_1^\al-\th_2^\al\right|\leq \th_1^{\al-1}|\th_1-\th_2|,\quad\forall\ 0<\th_2<\th_1,\ \forall\ \al\in(0,1).
	\end{equation}
	Indeed, since $\th_2<\th_1$, \eqref{Eq.claim4.80} is equivalent to $\th_1^\al-\th_2^\al\leq \th_1^{\al-1}(\th_1-\th_2)=\th_1^\al-\th_1^{\al-1}\th_2$, which is further equivalent to $\th_1^{\al-1}\th_2\leq \th_2^\al$, or $\th_1^{\al-1}\leq \th_2^{\al-1}$, which follows directly from $\th_2<\th_1$ and $\al\in(0,1)$. This proves \eqref{Eq.claim4.80}. It follows from \eqref{Eq.P_nG^1_1} and \eqref{Eq.P_nG^1_2} that
	\begin{align}
		\left|\cP_n[r,\g](\th_1)-\cP_n[r,\g](\th_2)\right|&
		\lesssim \min\Bigg\{\frac1{n^{1+\al}}\frac{\th_1^\al}{\langle\th_1\rangle^{2\al}}\left(\frac1n+\norm{(r,\g)}_X\right)+\frac1{n^2}\frac1{\langle\th_1\rangle^{1+\al}}|\th_1-\th_2|^\al,\nonumber\\
		& \frac1{n^\al}\left(\frac1n+\norm{(r,\g)}_X\right)\langle\th_1\rangle^{1-2\al}\th_1^{\al-1}|\th_1-\th_2|\Bigg\}\nonumber\\
		\lesssim \frac1{n^2}\frac1{\langle\th_1\rangle^{1+\al}}&|\th_1-\th_2|^\al+\frac{\th_1^\al}{n^\al\langle\th_1\rangle^{2\al}}\left(\frac1n+\norm{(r,\g)}_X\right)\min\left\{\frac1n, \frac{\langle\th_1\rangle}{\th_1}|\th_1-\th_2|\right\};\label{Eq.P_nG^1}
	\end{align}
	and it follows from \eqref{Eq.DeltaP_nG^1_1} and \eqref{Eq.DeltaP_nG^1_2} that
	\begin{align}
		&\left|\Delta\cP_n(\th_1)-\Delta\cP_n(\th_2)\right|\label{Eq.DeltaP_nG^1}\\
		\lesssim&\ \frac{|\th_1-\th_2|^\al}{\langle\th_1\rangle^{1+\al}}\frac{\norm{(r_1,\g_1)-(r_2,\g_2)}_X}{n^2}+\frac{\th_1^\al\norm{(r_1,\g_1)-(r_2,\g_2)}_X}{n^\al\langle\th_1\rangle^{2\al}}\min\left\{\frac1n, \frac{\langle\th_1\rangle}{\th_1}|\th_1-\th_2|\right\},\nonumber
	\end{align}
	for all $n>N$, $(r,\g), (r_1,\g_1), (r_2,\g_2)\in B_X(\varepsilon_1)$ and $0<\th_2<\th_1<\th_2+1$.
	
	Now we are ready to prove \eqref{Eq.J00} and \eqref{Eq.J_Lip}. Let $m_1:=N$ and $m>m_1$. By \eqref{Eq.J_series} and \eqref{Eq.P_nC^0}, we have
	\begin{align*}
		\norm{\langle\th\rangle^\al\cJ_m[0,0]}_{L^\infty}\lesssim \sum_{n=1}^\infty \norm{\langle\th\rangle^\al\cP_{mn}[0,0]}_{L^\infty}\lesssim \sum_{n=1}^\infty\frac1{(mn)^2}\lesssim \frac1{m^2}.
	\end{align*}
	By \eqref{Eq.J_series} and \eqref{Eq.P_nG^1}, for $0<\th_2<\th_1<\th_2+1$, we have
	\begin{align*}
		&\left|\cJ_m[0,0](\th_1)-\cJ_m[0,0](\th_2)\right|\lesssim \sum_{n=1}^\infty \left|\cP_{mn}[0,0](\th_1)-\cP_{mn}[0,0](\th_2)\right|\\
		\lesssim&\ \sum_{n=1}^\infty\frac1{(mn)^2}\frac1{\langle\th_1\rangle^{1+\al}}|\th_1-\th_2|^\al+\sum_{n=1}^\infty\frac{\th_1^\al}{(mn)^{1+\al}\langle\th_1\rangle^{2\al}}\min\left\{\frac1{mn}, \frac{\langle\th_1\rangle}{\th_1}|\th_1-\th_2|\right\}\\
		\lesssim&\ \frac1{m^2}\frac{|\th_1-\th_2|^\al}{\langle\th_1\rangle^{1+\al}}+\frac{\th_1^\al}{m^{1+\al}\langle\th_1\rangle^{2\al}}\min\left\{\frac1m\sum_{n=1}^\infty\frac1{n^{2+\al}}, \sum_{n=1}^\infty \frac1{n^{1+\al}} \frac{\langle\th_1\rangle}{\th_1}|\th_1-\th_2|\right\}\\
		\lesssim&\ \frac1{m^2}\frac{|\th_1-\th_2|^\al}{\langle\th_1\rangle^{1+\al}}+\frac{\th_1^\al}{m^{1+\al}\langle\th_1\rangle^{2\al}}\min\left\{\frac1m, \frac{\langle\th_1\rangle}{\th_1}|\th_1-\th_2|\right\}\\
		\lesssim&\ \frac1{m^2}\frac{|\th_1-\th_2|^\al}{\langle\th_1\rangle^{1+\al}}+ \frac{\th_1^\al}{m^{1+\al}\langle\th_1\rangle^{2\al}}\frac1{m^{1-\al}}\left(\frac{\langle\th_1\rangle}{\th_1}|\th_1-\th_2|\right)^\al\lesssim \frac1{m^2}\frac{|\th_1-\th_2|^\al}{\langle\th_1\rangle^{\al}}.
	\end{align*}
	Hence, \eqref{Eq.J00} holds. Let $(r_1,\g_1), (r_2,\g_2)\in B_X(\varepsilon_1)$. By \eqref{Eq.J_series} and \eqref{Eq.DeltaP_nC^0}, we have
	\begin{align*}
		&\norm{\langle\th\rangle^\al\left(\cJ_m[r_1,\g_1]-\cJ_m[r_2,\g_2]\right)}_{L^\infty}\leq \sum_{n=1}^\infty\norm{\langle\th\rangle^\al\left(\cP_{mn}[r_1,\g_1]-\cP_{mn}[r_2,\g_2]\right)}_{L^\infty}\\
		\lesssim&\ \sum_{n=1}^\infty\frac1{(mn)^{1+\al}}\norm{(r_1,\g_1)-(r_2,\g_2)}_X\lesssim \frac{1}{m^{1+\al}} \norm{(r_1,\g_1)-(r_2,\g_2)}_X.
	\end{align*}
	Let $\Delta\cJ_m:=\cJ_m[r_1,\g_1]-\cJ_m[r_2,\g_2]$. By \eqref{Eq.J_series} and \eqref{Eq.DeltaP_nG^1}, for $0<\th_2<\th_1<\th_2+1$, we have (here we assume $(r_1,\g_1)\neq (r_2,\g_2)$ without loss of generality)
	\begin{align*}
		&\frac{\left|\Delta\cJ_m(\th_1)-\Delta\cJ_m(\th_2)\right|}{\norm{(r_1,\g_1)-(r_2,\g_2)}_X}\leq \sum_{n=1}^\infty\frac{\left|\Delta\cP_{mn}(\th_1)-\Delta\cP_{mn}(\th_2)\right|}{\norm{(r_1,\g_1)-(r_2,\g_2)}_X}\\
		\lesssim &\ \sum_{n=1}^\infty \frac{|\th_1-\th_2|^\al}{(mn)^2\langle\th_1\rangle^{1+\al}}+\sum_{n=1}^\infty\frac{\th_1^\al}{(mn)^\al\langle\th_1\rangle^{2\al}}\min\left\{\frac1{mn}, \frac{\langle\th_1\rangle}{\th_1}|\th_1-\th_2|\right\}\\
		\lesssim&\ \frac{|\th_1-\th_2|^\al}{m^2\langle\th_1\rangle^{1+\al}}+\sum_{n\leq \frac{\th_1}{m\langle\th_1\rangle}|\th_1-\th_2|^{-1}}\frac{\th_1^{\al-1}\langle\th_1\rangle^{1-2\al}}{m^\al n^\al}|\th_1-\th_2|+ \sum_{n\geq \frac{\th_1}{m\langle\th_1\rangle}|\th_1-\th_2|^{-1}}\frac{\th_1^\al\langle\th_1\rangle^{-2\al}}{m^{1+\al}n^{1+\al}}\\
		\lesssim&\ \frac{|\th_1-\th_2|^\al}{m^2\langle\th_1\rangle^{1+\al}}+\frac{\th_1^{\al-1}\langle\th_1\rangle^{1-2\al}}{m^\al}|\th_1-\th_2|\left(\frac{\th_1}{m\langle\th_1\rangle}\right)^{1-\al}|\th_1-\th_2|^{\al-1}\\
		&\qquad\qquad\qquad+\frac{\th_1^\al\langle\th_1\rangle^{-2\al}}{m^{1+\al}} \left(\frac{\th_1}{m\langle\th_1\rangle}\right)^{-\al}|\th_1-\th_2|^\al\\
		\lesssim&\ \frac{|\th_1-\th_2|^\al}{m\langle\th_1\rangle^\al}.
	\end{align*}
	Hence, $[\Delta\cJ_m]_0\lesssim\frac1m\norm{	(r_1,\g_1)-(r_2,\g_2)}_X$, and then \eqref{Eq.J_Lip} follows.
\end{proof}

\section{Proof of Theorem \ref{mainthm} and Theorem \ref{mainthm2}}\label{Sec.Proof-weak-sol}
In this section, we prove Theorem \ref{mainthm} and Theorem \ref{mainthm2}. Let $\mu>1/2$ and $\al\in(0,1)$. By Theorem \ref{Thm.main}, there exists $m_*=m_*(\mu,\al)>0$ such that for all $m>m_*$, the equation \eqref{Eq.perturbation} has a solution $(r_m,\g_m)\in \overline{B_X(1/m)}$. Fix $m\in \Z\cap(m_*,+\infty)$. Recalling \eqref{special_sol}, we let $\wt r_m=r_0+r_m$, $\wt \g_m=\g_0+\g_m$. Then $(\wt r_m,\wt \g_m)=(\wt r_m,\wt \g_m)(\th)$ is a solution to \eqref{maineq}. 

\subsection{Quantitative behavior of vortex sheet}\label{Subsec.qualitative}

In this subsection, we prove the quantitative properties claimed in Theorem \ref{mainthm}. Specifically, we demonstrate the consistency between our derived solutions and Kaden's algebraic spiral vortex sheets. Since
\[\wt \g_m'(\th)=\g_0'(\th)+\g_m'(\th)=\th^{-2\mu}\left[2\pi(2\mu-1)+\th^{2\mu}\g_m'(\th)\right],\quad\forall\ \th>0\]
and $\|\th^{2\mu}\g_m'\|_{L^\infty}\leq \|\th^{2\mu}\g_m'\|_1\leq \|\g_m\|_{X_\g}\leq 1/m$, by adjusting $m_*>0$ to be larger if necessary, we have $\wt \g_m'(\th)>0$ for all $\th>0$. Hence, $\wt \g_m:\R_+\to\R$ is strictly increasing. Similarly, since
\[\wt \g_m(\th)=-2\pi\th^{1-2\mu}+\g_m(\th)=-2\pi\th^{1-2\mu}\left(1-\frac1{2\pi}\th^{2\mu-1}\g_m(\th)\right),\quad\forall\ \th>0\]
and $\|\th^{2\mu-1}\g_m\|_{L^\infty}\leq \|\g_m\|_{X_r}\leq 1/m$, we have $\wt \g_m(\th)<0$ for all $\th>0$. Consequently, $\wt \g_m:\R_+\to(-\infty, 0)$ is strictly increasing. Similarly, one checks that $\wt r_m:\R_+\to\R_+$ is strictly decreasing.

Recall the fact that if $f\in\mathscr{G}^k$ for some $k\geq0$, then the limit $f(0+):=\lim_{\th\to0+}f(\th)$ exists. Consequently, for each $m>m_*$, there are real numbers $a_m$ and $b_m$ such that
\begin{align}
	\lim_{\th\to0+}\th^\mu\wt r_m(\th)&=1+\lim_{\th\to0+}\th^\mu r_m(\th)=a_m,\label{Eq.r_m-far}\\ \lim_{\th\to0+}\th^{\mu+1}\wt r_m'(\th)&=-\mu+\lim_{\th\to0+}\th^{\mu+1} r_m'(\th)=-\mu a_m,\nonumber\\ -\lim_{\th\to0+}\th^{2\mu-1}\wt\g_m(\th)&=2\pi-\lim_{\th\to0+}\th^{2\mu-1}\g_m(\th)=b_m,\label{Eq.gamma_m-far}\\ \lim_{\th\to0+}\th^{2\mu}\wt\g_m'(\th)&=2\pi(2\mu-1)+\lim_{\th\to0+}\th^{2\mu}\g_m'(\th)=(2\mu-1)b_m,\nonumber
\end{align}
where we also used Corollary \ref{Cor.cG^0_0}. Moreover, it follows from \eqref{Eq.solution_bound} that
\begin{equation}\label{Eq.a_mb_m-limit}
	|a_m-1|+|b_m-2\pi|\leq \|(r_m,\g_m)\|_{X}\leq \frac C{m^2}<\frac12,\quad \forall\ m>m_*.
\end{equation}
On the other hand, by the definition of $\mathscr G^k$, we have $\lim_{\th\to+\infty}f(\th)=0$ for all $f\in\mathscr G^k$. Hence,
\begin{align}
	\lim_{\th\to+\infty}\th^\mu\wt r_m(\th)&=1,\quad \lim_{\th\to+\infty}\th^{\mu+1}\wt r_m'(\th)=-\mu,\label{Eq.r_m-center}\\ \lim_{\th\to+\infty}\th^{2\mu-1}\wt\g_m(\th)&=-2\pi,\quad \lim_{\th\to+\infty}\th^{2\mu}\wt\g_m'(\th)=2\pi(2\mu-1).\label{Eq.gamma_m-center}
\end{align}
Combining the strictly increasing property of $\wt\g_m$, the strictly decreasing property of $\wt r_m$, \eqref{Eq.r_m-far}, \eqref{Eq.gamma_m-far}, \eqref{Eq.a_mb_m-limit}, \eqref{Eq.r_m-center} and \eqref{Eq.gamma_m-center}, we know that 
\begin{equation}\label{Eq.gamma-r-bijective}
	\begin{aligned}
		\text{$\wt\g_m:\R_+\to(-\infty,0)$ is strictly increasing and bijective,}\\ \text{$\wt r_m: \R_+\to\R_+$ is strictly decreasing and bijective}.
	\end{aligned}
\end{equation}

We denote the inversion of $\wt \g_m:\R_+\to(-\infty, 0)$ by $\vartheta_m:(-\infty, 0)\to\R_+$. Then
\begin{equation}\label{Eq.theta-bijective}
	\text{$\vartheta_m:(-\infty, 0)\to\R_+$ is strictly increasing and bijective.}
\end{equation}
Let
\begin{equation}\label{Eq.z_m-def}
	z_m(\g):=\wt r_m(\vartheta_m(\g))\e^{\ii\vartheta_m(\g)},\quad \forall\ \g<0.
\end{equation}
Then $z_m:(-\infty,0)\to\C$ is a classical solution to \eqref{ss_BR2}. Thus, 
\begin{equation}\label{Eq.Z_m-def}
	Z_m(t,\Gamma):=t^\mu z_m\left(t^{1-2\mu}\Gamma\right),\quad t>0, \ \Gamma<0
\end{equation}
solves the Birkhoff-Rott equation \eqref{BR}.

By \eqref{Eq.r_m-far} and \eqref{Eq.gamma_m-far}, we have $\lim_{\g\to-\infty}\vartheta_m(\g)=0$ and
\begin{align*}
	\lim_{\g\to-\infty}\vartheta_m(\g)^\mu \wt r_m(\vartheta_m(\g))=a_m,\quad -\lim_{\g\to-\infty}\g\vartheta_m(\g)^{2\mu-1}=b_m,
\end{align*}
hence,
\[\lim_{\g\to-\infty}(-\g)^{-\frac{\mu}{2\mu-1}}\wt r_m(\vartheta_m(\g))=a_mb_m^{-\frac{\mu}{2\mu-1}}=:d_m.\]
As a consequence, the initial data of $Z_m$ is given by
\begin{equation}\label{Eq.Z_m(0)}
	Z_m(0,\Gamma)=a_mb_m^{-\frac{\mu}{2\mu-1}}(-\Gamma)^{\frac{\mu}{2\mu-1}},\quad\forall\ \Gamma<0.
\end{equation}

Now we explore the behavior of $Z_m$ near the spiral center, hence proving \eqref{Eq.center-behavior}. By \eqref{Eq.Z_m-def}, we have $Z_m(t,\Gamma)=R_m(t,\Gamma)\e^{\ii\Theta_m(t,\Gamma)}$, where $R_m(t,\Gamma):=t^\mu \wt r_m\left(\vartheta_m\left(t^{1-2\mu}\Gamma\right)\right)$ and $\Theta_m(t,\Gamma):=\vartheta_m\left(t^{1-2\mu}\Gamma\right)$. It follows from \eqref{Eq.gamma-r-bijective} and \eqref{Eq.theta-bijective} that $\Gamma\in(-\infty,0)\mapsto R_m(t,\Gamma)\in\R_+$ is $C^{1,\al}$, strictly decreasing and bijective; moreover, $\Gamma\in(-\infty,0)\mapsto \Theta_m(t,\Gamma)\in\R_+$ is $C^{1,\al}$, strictly increasing and bijective.
Using $\Theta$ as the parameter, by \eqref{Eq.z_m-def} and \eqref{Eq.Z_m-def}, we have 
\[Z_m(t,\Theta)=\wt R_m(t,\Theta)\e^{\ii\Theta},\quad\forall\ \Theta>0,\]
where $\wt R_m(t,\Theta):=t^\mu\wt r_m(\Theta)$. It follows from $\wt r_m=r_0+r_m$ and \eqref{Eq.r_m-center} that $\Theta^\mu\wt r_m\in L^\infty$, $\Theta^{\mu+1}\wt r_m'\in L^\infty$, and 
\begin{equation*}
	\lim_{\Theta\to+\infty}\Theta^\mu \wt r_m(\Theta)=1,\quad \lim_{\Theta\to+\infty}\Theta^{\mu+1}\wt r_m'(\Theta)=-\mu.
\end{equation*}
Moreover, $\wt r_m:\R_+\to\R_+$ is strictly decreasing and bijective, recalling \eqref{Eq.gamma-r-bijective}.

A standard scaling argument implies that if the initial data is given by 
\begin{equation}\label{Eq.initial-A}
	Z_m(0,\Gamma)=A(-\Gamma)^{\frac{\mu}{2\mu-1}},\quad\forall\ \Gamma<0,
\end{equation}
for some constant $A>0$, then (using $\Theta$ as the parameter) $Z_m(t,\Theta)=t^\mu\wt r_m(\Theta)\e^{\ii\Theta}$ with
\begin{equation*}
	\lim_{\Theta\to+\infty}\Theta^\mu \wt r_m(\Theta)=\left(\frac{d_m}{A}\right)^{2\mu-1},\quad \lim_{\Theta\to+\infty}\Theta^{\mu+1}\wt r_m'(\Theta)=-\mu\left(\frac{d_m}{A}\right)^{2\mu-1}.
\end{equation*}
Taking $A=\left(\frac{2\mu-1}{\mu}\right)^{\frac{\mu}{2\mu-1}}$ in \eqref{Eq.initial-A} implies that
\begin{equation}\label{Eq.tilde-beta_m}
	\wt \beta_m=d_m^{2\mu-1}\left(\frac{\mu}{2\mu-1}\right)^\mu
\end{equation}
in \eqref{Eq.center-behavior}. Taking $A=1$ in \eqref{Eq.initial-A} implies that $\beta_m=d_m^{2\mu-1}$ in \eqref{Eq.beta_m}. Moreover, by \eqref{Eq.a_mb_m-limit}, we have $\lim_{m\to\infty}\beta_m=(2\pi)^{-\mu}$.

\subsection{Behavior of the velocity field}
The remaining part of this section is devoted to proving that our constructed vortex sheets solve the primitive 2-D Euler equations \eqref{Eq.Euler-velocity} in the weak sense, i.e., Theorem \ref{mainthm2}. First of all, for each $m>m_*$, the velocity field associated with our vortex sheet solution is given by
\begin{align}
	u_m(t,x):&=\frac1m\sum_{k=0}^{m-1}\int_{-\infty}^{0}K_2\left(x-\xi_m^kZ_m(t,\Gamma)\right)\,\mathrm d\Gamma\nonumber\\
	&=\frac{\ii}{2\pi m}\sum_{k=0}^{m-1}\int_{-\infty}^{0}\frac{\mathrm d\Gamma}{\left(x-\xi_m^kZ_m(t,\Gamma)\right)^*},\quad\forall\ t\geq0,\ x\notin S_m(t),\label{Eq.u_m-def}
\end{align}
where $K_2(x)=\frac1{2\pi}\frac{x^\perp}{|x|^2}$ for $x\in\R^2\setminus\{0\}$, $Z_m$ is given by \eqref{Eq.Z_m-def} and \eqref{Eq.Z_m(0)}, $\xi_m=\exp\left(\frac{2\pi\ii}{m}\right)$ is the $m$-th unit root, and $S_m(t):=\left\{\xi_m^kZ_m(t,\Gamma): k\in\Z\cap[0,m-1], \Gamma<0\right\}\cup\{0\}$ for all $t\geq0$. Here we also identify $\R^2\simeq\C$. Recall the discussion in the end of Subsection \ref{Subsec.BR} and the properties of $Z_m$ presented in Subsection \ref{Subsec.qualitative}. It is classical to show that 
\begin{equation}\label{Eq.u_m-weaksol-far}
	\text{$u_m$, defined by \eqref{Eq.u_m-def}, is a weak solution to \eqref{Eq.Euler-velocity} in $(t,x)\in [0,+\infty)\times(\R^2\setminus\{0\})$.}
\end{equation}
Moreover, we have
\begin{align}\label{Eq.u_m-conver-far}
	u_m\in C([0,+\infty); L^2_{\text{loc}}(\R^2\setminus\{0\})),\quad \omega_m\in C([0,+\infty); H^{-1}_{\text{loc}}(\R^2\setminus\{0\})).
\end{align}

Consequently, we only need to verify that $u_m$ is a weak solution to \eqref{Eq.Euler-velocity} near the origin. Thanks to \eqref{Eq.Z_m-def}, we know that $u_m$ is self-similar, i.e.,
\begin{equation}\label{Eq.u_m-self-similar}
	u_m(t,x)=t^{\mu-1}v_m\left(t^{-\mu}x\right),\quad\forall\ t>0,\ x\notin S_m(t),
\end{equation}
where $v_m$ is the self-similar profile given by
\begin{equation}\label{Eq.v_m1}
	v_m(z)^*=\frac1{2m\pi\ii}\sum_{k=0}^{m-1}\int_{-\infty}^0\frac{\mathrm d\g}{z-\xi_m^kz_m(\g)}=\frac1{2\pi\ii}\int_{-\infty}^0\frac{z^{m-1}}{z^m-(z_m(\g))^m}\,\mathrm d\g,\quad\forall\ z\notin S_m,
\end{equation}
with $S_m:=\left\{\xi_m^kz_m(\g): k\in\Z\cap[0,m-1], \gamma<0\right\}\cup\{0\}$. We remark that if $z\in S_m\setminus\{0\}$, then
\[v_m(z)^*=\frac1{2\pi\ii}\pv\int_{-\infty}^0\frac{z^{m-1}}{z^m-z_m(\g)^m}\,\mathrm d\g.\]
Using $\theta$ as the parameter, we rewrite \eqref{Eq.v_m1} in the following form:
\begin{equation}\label{Eq.v_m-def}
	v_m(z)^*=\frac1{2\pi\ii}\int_0^\infty \frac{z^{m-1}}{z^m-\left(\wt r_m(\th)\e^{\ii\th}\right)^m}\wt\g_m'(\th)\,\mathrm d\th,\quad\forall\ z\notin S_m.
\end{equation}
To verify the weak solution property near the origin, it is crucial to understand the behavior of $v_m$, which is given by the following proposition.

\begin{proposition}\label{Prop.v_m-est}
	Assume that $\mu>1/2$ and $\al\in(0,1)$. Let $m_*=m_*(\mu,\al)>0$ be given by Theorem \ref{Thm.main} and fix $m>m_*$. For each $z\in\C\setminus S_m$ such that $0<|z|<1$, let $\th_0>0$ be the unique solution to $\wt r_m(\th_0)=|z|$. Then
	\begin{equation}\label{Eq.v_m-est}
		\left|v_m(z)^*-\frac{\wt\g_m(\th_0)}{2\pi\mathrm{i} z}\right|\leq C|z|\big(1+|\ln|z||\big),
	\end{equation}
	where $C>0$ is a constant independent of $z$ (but may depend on $m$).
\end{proposition}
\begin{proof}
	We emphasize that, throughout this and the next subsection, $m>m_*$ is fixed. The existence and uniqueness of $\th_0$ follows directly from \eqref{Eq.gamma-r-bijective}. For simplicity, we denote $z(\th):=\wt r_m(\th)\e^{\ii\th}$ and $\wt z(\th):=\left(\wt r_m(\th)\e^{\ii\th}\right)^m$ for $\th>0$. Then it follows from $\wt r_m=r_0+r_m$, $r_0(\th)=\th^{-\mu}$ and $\|(r_m, \g_m)\|_X\leq 1/2$ that
    \begin{equation}\label{Eq.r_m-approx}
        |z(\th)|=\wt r_m(\th)\sim \th^{-\mu},\quad |z|=\wt r_m(\th_0)\sim \th_0^{-\mu}.
    \end{equation}
    By \eqref{Eq.v_m-def}, we have
	\begin{align}
		2\pi\ii z\left(v_m(z)^*-\frac{\wt\g_m(\th_0)}{2\pi\ii z}\right)&=z\int_{\th_0}^\infty \left(\frac{z^{m-1}}{z^m-z(\th)^m}-\frac1z\right)\wt\g_m'(\th)\,\mathrm d\th+\int_0^{\th_0}\frac{z^m\wt\g_m'(\th)}{z^m-z(\th)^m}\,\mathrm d\th\nonumber\\
		&=\int_{\th_0}^\infty \frac{z(\th)^m}{z^m-z(\th)^m}\wt\g_m'(\th)\,\mathrm d\th+\int_0^{\th_0}\frac{z^m\wt\g_m'(\th)}{z^m-z(\th)^m}\,\mathrm d\th,\label{Eq.v_m-leading}
	\end{align}
	where we have used the fact that $\lim_{\th\to+\infty}\wt\g_m(\th)=0$, recalling \eqref{Eq.gamma-r-bijective}. Since $|z|<1$, \eqref{Eq.gamma-r-bijective} and $\wt r_m(\th_0)=|z|$, there is a constant $\th_*>0$ such that $(\wt r_m)^{-1}(|z|)>\th_*$ for all $|z|<1$. We assume without loss of generality that $\th_0>1$ for all $|z|<1$. We choose a smooth bump function $\rho\in C_c^\infty(\R_+;[0,1])$ which supports near $\th_0$ and is defined by
	\begin{equation}\label{Eq.bump}
		\rho_0\in C_c^\infty(\R;[0,1]),\quad \rho_0\big|_{[-1/2,1/2]}\equiv1,\quad \operatorname{supp}\rho_0\subset[-1,1],\quad \rho(\th)=\rho_0\big(m(\th-\th_0)\big).
	\end{equation}
	Let (below ``N'' stands for ``Near'' and ``F'' stands for ``Far'')
	\begin{align}
		I^{\text{N}}:&=\int_{\th_0}^\infty \frac{z(\th)^m}{z^m-z(\th)^m}\wt\g_m'(\th)\rho(\th)\,\mathrm d\th+\int_0^{\th_0}\frac{z^m\wt\g_m'(\th)}{z^m-z(\th)^m}\rho(\th)\,\mathrm d\th,\label{Eq.I^N-def}\\
		I^{\text{F}}:&=\int_{\th_0}^\infty \frac{z(\th)^m}{z^m-z(\th)^m}\wt\g_m'(\th)(1-\rho(\th))\,\mathrm d\th+\int_0^{\th_0}\frac{z^m\wt\g_m'(\th)}{z^m-z(\th)^m}(1-\rho(\th))\,\mathrm d\th.\label{Eq.I^F-def}
	\end{align}
    Then by \eqref{Eq.v_m-leading}, we have $2\pi\ii z\left(v_m(z)^*-\frac{\wt\g_m(\th_0)}{2\pi\ii z}\right)=I^{\text{N}}+I^{\text{F}}$. Hence, \eqref{Eq.v_m-est} is a direct consequence of the following estimates
    \begin{align}
        \left|I^{\text{N}}\right|\leq C|z|^2\big(1+|\ln|z||\big),\label{Eq.I^N-est}\\
        \left|I^{\text{F}}\right|\leq C|z|^2\big(1+|\ln|z||\big),\label{Eq.I^F-est}
    \end{align}
    where $C>0$ is a constant independent of $\th_0>1$. 
    
    We first deal with $I^{\text{F}}$. By \eqref{Eq.gamma-r-bijective}, we know that $|z(\th)|=\wt r_m(\th)<\wt r_m(\th_0)=|z|$ for $\th>\th_0$ and $|z(\th)|=\wt r_m(\th)>\wt r_m(\th_0)=|z|$ for $\th\in(0,\th_0)$. Thus,
    \begin{align*}
        \sum_{n=1}^\infty\left(\frac{z(\th)}{z}\right)^{mn}(1-\rho(\th))=\frac{z(\th)^m}{z^m-z(\th)^m}(1-\rho(\th)),\quad\forall\ \th>\th_0,\\
        -\sum_{n=1}^\infty\left(\frac z{z(\th)}\right)^{mn}(1-\rho(\th))=\frac{z^m}{z^m-z(\th)^m}(1-\rho(\th)),\quad\forall\ \th\in(0,\th_0).
    \end{align*}
   Moreover, both series converge absolutely and uniformly with respect to $\th$, because $1-\rho$ is supported outside a neighborhood of $\th_0$. Consequently, we have
    \begin{equation}\label{Eq.I^F-sum}
        I^{\text{F}}=\sum_{n=1}^\infty I^{\text{F}}_n,\quad\text{where}\quad  I^{\text{F}}_n:=I^{\text{F}}_{n,1}-I^{\text{F}}_{n,2},
    \end{equation}
    with
    \begin{align}
        I^{\text{F}}_{n,1}:&=\int_{\th_0}^\infty \left(\frac{z(\th)}{z}\right)^{mn}\wt\g_m'(\th)(1-\rho(\th))\,\mathrm d\th\nonumber\\
        &=\frac1{z^{mn}}\int_{\th_0}^\infty z(\th)^{mn-1}z'(\th)\frac{z(\th)}{z'(\th)}\wt\g_m'(\th)(1-\rho(\th))\,\mathrm d\th,\label{Eq.I_n,1^F}\\
        I^{\text{F}}_{n,2}:&=\int_0^{\th_0}\left(\frac z{z(\th)}\right)^{mn}\wt\g_m'(\th)(1-\rho(\th))\,\mathrm d\th\nonumber\\
        &=z^{mn}\int_0^{\th_0}z(\th)^{-mn-1}z'(\th)\frac{z(\th)}{z'(\th)}\wt\g_m'(\th)(1-\rho(\th))\,\mathrm d\th.\label{Eq.I_n,2^F}
    \end{align}
    We compute that $\frac{z(\th)}{z'(\th)}\wt\g_m'(\th)=\frac{\th^{1-2\mu}}{-\mu+\ii\th}F(\th)$, where
    \begin{align}\label{Eq.F-def}
        F(\th):=\frac{\left(1+\th^\mu r_m(\th)\right)\left(2\pi(2\mu-1)+\th^{2\mu-1} \g_m'(\th)\right)}{1+\frac{\th^{\mu+1}}{-\mu+\ii\th}\left(r_m'(\th)+\ii r_m(\th)\right)}\in \mathscr G^1_*.
    \end{align}
    Here we have used $\|(r_m,\g_m)\|_X\ll1$, \eqref{4.16} and the Banach algebra property of $\mathscr G^1_*$. By \eqref{Eq.I_n,1^F}, \eqref{Eq.I_n,2^F} and Lemma \ref{Lem.far-est}, we have
    \begin{align}
        \left|I^{\text{F}}_{n,1}\right|+\left|I^{\text{F}}_{n,2}\right|\lesssim n^{-1-\al}\big(1+|\ln|z||\big),\quad\forall\ n\in\Z_+.\label{Eq.I_n,1,2^F-est}
    \end{align}
    Therefore, \eqref{Eq.I^F-est} follows from \eqref{Eq.I^F-sum} and \eqref{Eq.I_n,1,2^F-est}. 
    
    It remains to prove \eqref{Eq.I^N-est}. Recall that $\wt z(\th)=z(\th)^m$, hence $I^{\text N}$ can be rewritten as
    \begin{equation}\label{Eq.I^N-sum}
        I^{\text N}=I^{\text N}_1+I^{\text N}_2,
    \end{equation}
    where
    \begin{align*}
        I^{\text N}_1:&=\int_{\th_0}^\infty \frac{\wt z'(\th)}{z^m-\wt z(\th)}\frac{\wt z(\th)}{\wt z'(\th)}\wt\g_m'(\th)\rho(\th)\,\mathrm d\th,\\
        I^{\text N}_2:&=\int_0^{\th_0}\frac{z^m\wt z'(\th)}{\wt z(\th)\left(z^m-\wt z(\th)\right)}\frac{\wt z(\th)}{\wt z'(\th)}\wt\g_m'(\th)\rho(\th)\,\mathrm d\th.
    \end{align*}
    We compute that $\frac{\wt z(\th)}{\wt z'(\th)}\wt\g_m'(\th)=\frac1m\frac{\th^{1-2\mu}}{-\mu+\ii\th}F(\th)$, where $F$ is given by \eqref{Eq.F-def}. Now, the desired estimate \eqref{Eq.I^N-est} follows from \eqref{Eq.I^N-sum} and Lemma \ref{Lem.near-est}.
\end{proof}

\begin{lemma}\label{Lem.far-est}
    Under the assumptions of Proposition \ref{Prop.v_m-est}, let $F\in\mathscr G^1_*$, we define
    \begin{align*}
        J_{n,1}:&=\frac1{z^{mn}}\int_{\th_0}^\infty z(\th)^{mn-1}z'(\th)\frac{\th^{1-2\mu}}{-\mu+\mathrm{i}\th}F(\th)(1-\rho(\th))\,\mathrm d\th,\\
        J_{n,2}:&=z^{mn}\int_0^{\th_0}z(\th)^{-mn-1}z'(\th)\frac{\th^{1-2\mu}}{-\mu+\mathrm{i}\th}F(\th)(1-\rho(\th))\,\mathrm d\th,
    \end{align*}
    for all $n\in\Z_+$, where $\rho=\rho(\th)$ is given by \eqref{Eq.bump}. Then 
    \begin{align}
    \left|J_{n,1}\right|&\lesssim\left(\frac1{n^2}|z|^2+\frac1{n^{1+\al}}|z|^{2+\frac\al\mu}\right)\|F\|_{*,1},\label{Eq.J_n,1est}\\
    \left|J_{n,2}\right|&\lesssim\left(\frac1{n^2}|z|^2(1+|\ln|z||)+\frac1{n^{1+\al}}|z|^{2}\right)\|F\|_{*,1}.\label{Eq.J_n,2est}
    \end{align}
    Here the implicit constants are independent of $\th_0>1$, $n\in\Z_+$ and $F\in\mathscr G^1_*$.
\end{lemma}
\begin{proof}
    Recall that $\mathscr G^1_*=\mathscr G^1\oplus\C$. Our desired \eqref{Eq.J_n,1est} and \eqref{Eq.J_n,2est} are direct consequences of the following estimates:
    \begin{align}
        F\equiv 1&\Longrightarrow \left|J_{n,1}\right|\lesssim n^{-2}{|z|^2},\label{Eq.J_n,1-est1}\\
        F\equiv 1&\Longrightarrow \left|J_{n,2}\right|\lesssim n^{-2}{|z|^2}(1+|\ln|z||),\label{Eq.J_n,2-est1}\\
        F\in\mathscr G^1&\Longrightarrow \left|J_{n,1}\right|\lesssim n^{-1-\al}{|z|^{2+\al/\mu}}\|F\|_1,\label{Eq.J_n,1-est2}\\
        F\in\mathscr G^1&\Longrightarrow \left|J_{n,2}\right|\lesssim n^{-1-\al}{|z|^{2}}\|F\|_1\label{Eq.J_n,2-est2}
    \end{align}

    \underline{\bf Proof of \eqref{Eq.J_n,1-est1}.}
    Let $\varphi(\th):=\frac{\th^{1-2\mu}}{-\mu+\mathrm{i}\th}(1-\rho(\th))$ for all $\th>0$. Then
    \begin{equation}\label{Eq.varphi-property}
        |\varphi(\th)|\lesssim\th^{1-2\mu}\langle\th\rangle^{-1},\quad \left|\varphi'(\th)\right|\lesssim\th^{-2\mu}\langle\th\rangle^{-1},\quad\forall\ \th\in\R_+.
    \end{equation}
    Assume that $F\equiv1$. Using integration by parts, $\varphi(\th_0)=0$, \eqref{Eq.varphi-property} and \eqref{Eq.r_m-approx}, we have 
    \begin{align*}
        |J_{n,1}|&=\left|\frac1{z^{mn}}\int_{\th_0}^\infty z(\th)^{mn-1}z'(\th)\varphi(\th)\,\mathrm d\th\right|=\left|\frac1{mnz^{mn}}\int_{\th_0}^\infty z(\th)^{mn}\varphi'(\th)\,\mathrm d\th\right|\\
        &\lesssim \frac1{mn|z|^{mn}}\int_{\th_0}^\infty|z(\th)|^{mn}\th^{-2\mu-1}\,\mathrm d\th\lesssim \frac1{mn\th_0^{-\mu mn}}\int_{\th_0}^\infty \th^{-\mu mn-2\mu-1}\,\mathrm d\th\\
        &\lesssim n^{-2}\th_0^{-2\mu}\lesssim n^{-2}|z|^2.
    \end{align*}
    This proves \eqref{Eq.J_n,1-est1}.

    \underline{\bf Proof of \eqref{Eq.J_n,2-est1}.} Assume that $F\equiv1$. Using integration by parts, $\varphi(\th_0)=0$, \eqref{Eq.varphi-property} and \eqref{Eq.r_m-approx}, we have 
    \begin{align*}
        |J_{n,2}|&=\left|z^{mn}\int_0^{\th_0}z(\th)^{-mn-1}z'(\th)\varphi(\th)\,\mathrm d\th\right|=\left|\frac{z^{mn}}{mn}\int_0^{\th_0} z(\th)^{-mn}\varphi'(\th)\,\mathrm d\th\right|\\
        &\lesssim \frac{|z|^{mn}}{mn}\int_0^{\th_0}|z(\th)|^{-mn}\th^{-2\mu}\langle\th\rangle^{-1}\,\mathrm d\th\lesssim\frac{\th_0^{-\mu mn}}{mn}\int_0^{\th_0}\th^{\mu(mn-2)}\langle\th\rangle^{-1}\,\mathrm d\th\\
        &\lesssim\begin{cases}
            \th_0^{-2\mu}\ln\th_0\lesssim |z|^2(1+|\ln|z||) & \text{if}\ mn=2\\
            n^{-2}\th_0^{-2\mu}\lesssim n^{-2}|z|^2 & \text{if}\ mn>2
        \end{cases}\lesssim n^{-2}|z|^2(1+|\ln|z||).
    \end{align*}
    This proves \eqref{Eq.J_n,2-est1}. 

    \underline{\bf Proof of \eqref{Eq.J_n,1-est2}.} Assume that $F\in\mathscr G^1$. By Lemma \ref{Lem.decompose_G^1}, there exist $F_0\in C(\R_+;\C)$ and $F_1\in C^1(\R_+;\C)$ such that $F=F_0+F_1$ and
    \begin{equation}\label{Eq.F-decompose}
        \left\|\langle\th\rangle^{1+\al}F_0\right\|_{L^\infty}\lesssim n^{-\al}\|F\|_1,\quad \left\|\langle\th\rangle^{\al}F_1\right\|_{L^\infty}\lesssim \|F\|_1,\quad \left\|\langle\th\rangle^{1+\al}F_1'\right\|_{L^\infty}\lesssim n^{1-\al}\|F\|_1.
    \end{equation}
    Hence $J_{n,1}=J_{n,1,0}+J_{n,1,1}$, where
    \begin{align*}
        J_{n,1,0}:=\frac1{z^{mn}}\int_{\th_0}^\infty z(\th)^{mn-1}z'(\th)F_0(\th)\varphi(\th)\,\mathrm d\th,\ J_{n,1,1}:=\frac1{z^{mn}}\int_{\th_0}^\infty z(\th)^{mn-1}z'(\th)F_1(\th)\varphi(\th)\,\mathrm d\th.
    \end{align*}
    For $J_{n,1,0}$, by \eqref{Eq.r_m-approx}, \eqref{Eq.varphi-property} and \eqref{Eq.F-decompose}, we have
    \begin{align}
        \left|J_{n,1,0}\right|&\lesssim \frac{\left\|\langle\th\rangle^{1+\al}F_0\right\|_{L^\infty}}{|z|^{mn}}\int_{\th_0}^\infty \th^{-\mu(mn-1)}\th^{-\mu-1}\langle\th\rangle\langle\th\rangle^{-1-\al}\th^{1-2\mu}\langle\th\rangle^{-1}\,\mathrm d\th\nonumber\\
        &\lesssim\frac{\left\|\langle\th\rangle^{1+\al}F_0\right\|_{L^\infty}}{n\th_0^{-\mu mn}}\th_0^{-\mu mn-2\mu-\al}\lesssim \frac1{n^{1+\al}}\th_0^{-2\mu-\al}\|F\|_1\lesssim\frac1{n^{1+\al}}|z|^{2+\frac\al\mu}\|F\|_1,\label{Eq.J_n,1,0}
    \end{align}
    where we have used the fact
    \begin{equation}
        |z'(\th)|\sim \th^{-\mu-1}|-\mu+\th^{\mu+1}r_m'(\th)|+\th^{-\mu}|1+\th^\mu r_m(\th)|\sim \th^{-\mu-1}\langle\th\rangle,\quad\forall\ \th>0.
    \end{equation}
    As for $J_{n,1,1}$, we employ integration by parts, which motivates us to estimate $(F_1\varphi)'$. Indeed, by \eqref{Eq.varphi-property} and \eqref{Eq.F-decompose}, we have
    \begin{align}
        \left|(F_1\varphi)'(\th)\right|&\lesssim \langle\th\rangle^{-1-\al}\left\|\langle\th\rangle^{1+\al}F_1'\right\|_{L^\infty}\th^{1-2\mu}\langle\th\rangle^{-1}+\langle\th\rangle^{-\al}\left\|\langle\th\rangle^{\al}F_1\right\|_{L^\infty}\th^{-2\mu}\langle\th\rangle^{-1}\nonumber\\
        &\lesssim n^{1-\al}\th^{-2\mu}\langle\th\rangle^{-1-\al}\|F\|_1,\quad\forall\ \th>0.\label{Eq.F-phi-derivative}
    \end{align}
    Now it follows from integration by parts, \eqref{Eq.r_m-approx} and \eqref{Eq.F-phi-derivative} that
    \begin{align}
        \left|J_{n,1,1}\right|&=\left|\frac1{mn z^{mn}}\int_{\th_0}^\infty z(\th)^{mn}(F_1\varphi)'(\th)\,\mathrm d\th\right|\lesssim \frac{n^{-\al}\|F\|_1}{\th_0^{-\mu mn}}\int_{\th_0}^\infty \th^{-\mu mn}\th^{-2\mu}\langle\th\rangle^{-1-\al}\,\mathrm d\th\nonumber\\
        &\lesssim n^{-1-\al}\th_0^{-2\mu-\al}\|F\|_1\lesssim n^{-1-\al}|z|^{2+\frac\al\mu}\|F\|_1.\label{Eq.J_n,1,1}
    \end{align}
    Combining \eqref{Eq.J_n,1,0} and \eqref{Eq.J_n,1,1} implies \eqref{Eq.J_n,1-est2}.

    \underline{\bf Proof of \eqref{Eq.J_n,2-est2}.} Assume that $F\in\mathscr G^1$. Similar to the proof of \eqref{Eq.J_n,1-est2}, we can obtain
    \begin{equation*}
        |J_{n,2}|\lesssim\begin{cases}
            n^{-1-\al}|z|^2\|F\|_1 &  \text{if}\ mn=2\\
            n^{-1-\al}|z|^{2+\al/\mu}\|F\|_1 &  \text{if}\ mn>2
        \end{cases},
    \end{equation*}
    where the dichotomy is a consequence of $\int_0^{\th_0}\th^{\mu(mn-2)}\langle\th\rangle^{-1-\al}\,\mathrm d\th\lesssim 1$ (without ``$\al$'' in the exponent) for $mn=2$. This directly implies \eqref{Eq.J_n,2-est2}.
\end{proof}

\begin{lemma}\label{Lem.near-est}
    Under the assumptions of Proposition \ref{Prop.v_m-est}, let $F\in\mathscr G^1_*$, we define
    \begin{align*}
        K_1:&=\int_{\th_0}^\infty \frac{\wt z'(\th)}{z^m-\wt z(\th)}\frac{\th^{1-2\mu}}{-\mu+\mathrm i\th}F(\th)\rho(\th)\,\mathrm d\th,\\
        K_2:&=\int_0^{\th_0}\frac{z^m\wt z'(\th)}{\wt z(\th)\left(z^m-\wt z(\th)\right)}\frac{\th^{1-2\mu}}{-\mu+\mathrm i\th}F(\th)\rho(\th)\,\mathrm d\th,
    \end{align*}
    where $\rho=\rho(\th)$ is given by \eqref{Eq.bump}. Then 
    \begin{align}\label{Eq.K_1+K_2-est}
    |K_1+K_2|\lesssim |z|^2\big(1+|\ln |z||\big)\|F\|_{*,1}.
    \end{align}
    Here the implicit constants are independent of $\th_0>1$, $n\in\Z_+$ and $F\in\mathscr G^1_*$.
\end{lemma}
\begin{proof}
    Let $\wt\varphi(\th):=\frac{\th^{1-2\mu}}{-\mu+\mathrm i\th}F(\th)$ for all $\th>0$. Then we have
    \begin{align}\label{Eq.tilde-varphi-est}
    	\left|\wt\varphi(\th)\right|\lesssim\th^{-2\mu}\|F\|_{*,1},\quad \left|\wt\varphi(\th)-\wt\varphi(\th_0)\right|\lesssim \th_0^{-2\mu-1}\left(|\th-\th_0|+\th_0^{-\al}|\th-\th_0|^\al\right)\|F\|_{*,1}
    \end{align}
    for all $\th\in[\th_0-1/m,\th_0+1/m]$, in view of $\operatorname{supp}\rho\subset [\th_0-1/m,\th_0+1/m]$. Notice that the function $\th\in [\th_0-1/m,\th_0+1/m]\mapsto \wt z(\th)-z^m$ maps into a sector in $\C$ with an angle smaller than $2\pi$, thanks to the assumption $z\notin \{\xi_m^kz(\th): k\in\Z\cap[0, m-1], \th>0\}$. Hence, we can choose\footnote{\label{Footnote.branch}We remark that the choice of the logarithmic branch $\operatorname{Ln}$ depends on $z$ (and $\th_0$). Nonetheless, the estimate \eqref{Eq.Ln-est} is uniform with respect to $\th_0>1$.} a unique branch $\operatorname{Ln}$ of the logarithmic function such that  $\operatorname{Ln}\left(\wt z(\th_0)-z^m\right)\in\R$ and $\th\mapsto \operatorname{Ln}\left(\wt z(\th)-z^m\right)$ is a $C^1$ function on $\th\in [\th_0-1/m,\th_0+1/m]$. Moreover, there holds
    \begin{equation}\label{Eq.Ln}
    	\frac{\mathrm d}{\mathrm d\th}\operatorname{Ln}\left(\wt z(\th)-z^m\right)=\frac{\wt z'(\th)}{\wt z(\th)-z^m},\quad \forall\ \th\in\left[\th_0-\frac1m,\th_0+\frac1m\right].
    \end{equation} 
   	For $\th\in [\th_0-1/m,\th_0+1/m]$, by \eqref{Eq.r_m-approx} and $\wt r_m'(\th)\sim \th^{-\mu-1}$, we have
   	\begin{align}
   		\left|\wt z(\th)-z^m\right|&\geq \left|\left|\wt z(\th)\right|-|z|^m\right|=\left|\wt r_m(\th)^m-\wt r_m(\th_0)^m\right|=m\left|\int_{\th_0}^\th \left(\wt r_m(\wt\th)\right)^{m-1}\wt r_m'(\wt\th)\,\mathrm d\wt\th\right|\label{Eq.z-z-lower}\\
   		&\gtrsim m\left|\int_{\th_0}^\th \wt\th^{-\mu m-1}\,\mathrm d\wt\th\right|\gtrsim \th_0^{-\mu m-1}|\th-\th_0|,\nonumber\\
   		\left|\wt z(\th)-z^m\right|&\leq \wt r_m(\th)^m+\wt r_m(\th_0)^m\lesssim \th_0^{-\mu m}.\label{Eq.z-z-upper}
   	\end{align}
   	Specifically, there holds
   	\begin{equation}\label{Eq.Ln-est}
   		\left|\operatorname{Ln}\left(\wt z(\th)-z^m\right)\right|\lesssim 1+|\ln\th_0|\lesssim 1+|\ln|z||,\quad \forall\ \frac1{2m}\leq|\th-\th_0|\leq\frac{1}{m}.
   	\end{equation}
    
    Now we are ready to estimate $K_1$ and $K_2$. We start with $K_1$. Note that $K_1=K_{1,1}+K_{1,2}$, 
    \begin{align*}
    	K_{1,1}:&=\wt\varphi(\th_0)\int_{\th_0}^\infty\frac{\wt z'(\th)}{z^m-\wt z(\th)}\rho(\th)\,\mathrm d\th,\\
    	K_{1,2}:&=\int_{\th_0}^\infty\frac{\wt z'(\th)}{z^m-\wt z(\th)}\left(\wt\varphi(\th)-\wt\varphi(\th_0)\right)\rho(\th)\,\mathrm d\th.
    \end{align*}
    Using \eqref{Eq.Ln}, integration by parts and $\rho(\th_0)=1$, we obtain
    \begin{align*}
    	K_{1,1}=\wt\varphi(\th_0)\operatorname{Ln}\left(\wt z(\th_0)-z^m\right)+\wt\varphi(\th_0)\int_{\th_0}^\infty \operatorname{Ln}\left(\wt z(\th)-z^m\right)\rho'(\th)\,\mathrm d\th.
    \end{align*}
    Hence, \eqref{Eq.bump}, \eqref{Eq.tilde-varphi-est} and \eqref{Eq.Ln-est} imply that
    \begin{align}
    	&\left|K_{1,1}-\wt\varphi(\th_0)\operatorname{Ln}\left(\wt z(\th_0)-z^m\right)\right|\lesssim \left|\wt\varphi(\th_0)\right|\int_{\th_0+\frac1{2m}}^{\th_0+\frac1m}\left|\operatorname{Ln}\left(\wt z(\th)-z^m\right)\right|\,\mathrm d\th\nonumber\\
    	&\qquad\qquad\lesssim \th_0^{-2\mu}\|F\|_{*,1}\left(1+|\ln|z||\right)\lesssim |z|^2\left(1+|\ln|z||\right)\|F\|_{*,1}.\label{Eq.K_1,1-est}
    \end{align}
    On the other hand, by \eqref{Eq.z-z-lower}, \eqref{Eq.tilde-varphi-est} and the fact (thanks to \eqref{Eq.r_m-approx})
   	\begin{equation}\label{Eq.z'(th)-est}
   		\left|\wt z'(\th)\right|=m|z(\th)|^{m-1}\left|\wt r_m'(\th)+\ii\wt r_m(\th)\right|\sim \th^{-\mu m-1}\langle\th\rangle,\quad\forall\ \th>0,
   	\end{equation}
   	we have
   	\begin{align}
   		\left|K_{1,2}\right|&\lesssim \int_{\th_0}^{\th_0+\frac1m}\frac{\th_0^{-\mu m}\th_0^{-2\mu-1}}{\th_0^{-\mu m-1}|\th-\th_0|}\left(|\th-\th_0|+\th_0^{-\al}|\th-\th_0|^\al\right)\|F\|_{*,1}\,\mathrm d\th\nonumber\\
   		&\lesssim \left(\th_0^{-2\mu}+\th_0^{-2\mu-\al}\right)\|F\|_{*,1}\lesssim \left(|z|^2+|z|^{2+\frac\al\mu}\right)\|F\|_{*,1}\lesssim |z|^2\|F\|_{*,1}.\label{Eq.K_1,2-est}
   	\end{align}
   	Combining \eqref{Eq.K_1,1-est} and \eqref{Eq.K_1,2-est} gives that
   	\begin{equation}\label{Eq.K_1-est}
   		\left|K_1-\wt\varphi(\th_0)\operatorname{Ln}\left(\wt z(\th_0)-z^m\right)\right|\lesssim |z|^2\left(1+|\ln|z||\right)\|F\|_{*,1}.
   	\end{equation}
   	
   	Similar to $K_1$, we decompose $K_2=K_{2,1}+K_{2,2}$, where
   	\begin{align*}
   		K_{2,1}:&=\wt\varphi(\th_0)\int_0^{\th_0}\frac{z^m\wt z'(\th)}{\wt z(\th)\left(z^m-\wt z(\th)\right)}\rho(\th)\,\mathrm d\th,\\
   		K_{2,2}:&=\int_0^{\th_0}\frac{z^m\wt z'(\th)}{\wt z(\th)\left(z^m-\wt z(\th)\right)}\left(\wt\varphi(\th)-\wt\varphi(\th_0)\right)\rho(\th)\,\mathrm d\th.
   	\end{align*}
   	Since $\wt z(\th)=\wt r_m(\th)^m\e^{\ii m\th}$, we can choose\footnote{Recall footnote \ref{Footnote.branch}. While the choice of $\wt{\operatorname{Ln}}$ depends on $\th_0$, the estimate \eqref{Eq.tilde-Ln-est} is uniform with respect to $\th_0>1$.} a unique branch $\wt{\operatorname{Ln}}$ of logarithmic function such that $\wt{\operatorname{Ln}}\ \wt z(\th_0)\in\R$ and $\theta\mapsto \wt{\operatorname{Ln}}\ \wt z(\th)$ is $C^1$ on $\th\in[\th_0-1/m,\th_0+1/m]$. Moreover, there holds
   	\begin{equation}\label{Eq.tilde-Ln}
   		\frac{\mathrm d}{\mathrm d\th}\wt{\operatorname{Ln}}\ \wt z(\th)=\frac{\wt z'(\th)}{\wt z(\th)},\quad \forall\ \th\in\left[\th_0-\frac1m,\th_0+\frac1m\right].
   	\end{equation}
   	By \eqref{Eq.r_m-approx}, we have $|\wt z(\th)|\sim \th^{-\mu m}$ for all $\th>0$, hence, 
   	\begin{equation}\label{Eq.tilde-Ln-est}
   		\left|\wt{\operatorname{Ln}}\ \wt z(\th)\right|\lesssim 1+|\ln \th_0|\lesssim 1+|\ln |z||, \quad \forall\ \th\in\left[\th_0-1/m,\th_0+1/m\right].
   	\end{equation}
   	Using \eqref{Eq.Ln}, \eqref{Eq.tilde-Ln}, integration by parts and $\rho(\th_0)=1$, we obtain
   	\begin{align*}
   		K_{2,1}&=\wt\varphi(\th_0)\int_0^{\th_0}\frac{\mathrm d}{\mathrm d\th}\left(\wt{\operatorname{Ln}}\ \wt z(\th)-\operatorname{Ln}\left(\wt z(\th)-z^m\right)\right)\rho(\th)\,\mathrm d\th\\
   		&=\wt\varphi(\th_0)\left(\wt{\operatorname{Ln}}\ \wt z(\th_0)-\operatorname{Ln}\left(\wt z(\th_0)-z^m\right)\right)-\wt\varphi(\th_0)\int_0^{\th_0}\left(\wt{\operatorname{Ln}}\ \wt z(\th)-\operatorname{Ln}\left(\wt z(\th)-z^m\right)\right)\rho'(\th)\,\mathrm d\th.
   	\end{align*}
   	Hence, \eqref{Eq.bump}, \eqref{Eq.tilde-varphi-est}, \eqref{Eq.Ln-est} and \eqref{Eq.tilde-Ln-est} imply that
   	\begin{align}
   		&\left|K_{2,1}-\wt\varphi(\th_0)\left(\wt{\operatorname{Ln}}\ \wt z(\th_0)-\operatorname{Ln}\left(\wt z(\th_0)-z^m\right)\right)\right|\nonumber\\
   		\lesssim&\  \th_0^{-2\mu}\|F\|_{*,1}\left(1+|\ln|z||\right)\lesssim |z|^2\left(1+|\ln|z||\right)\|F\|_{*,1}.\label{Eq.K_2,1-est}
   	\end{align}
   	On the other hand, by \eqref{Eq.r_m-approx}, \eqref{Eq.z-z-lower}, \eqref{Eq.tilde-varphi-est} and \eqref{Eq.z'(th)-est}, we have
   	\begin{align}
   		\left|K_{2,2}\right|&\lesssim |z|^m\int_{\th_0-\frac1m}^{\th_0}\frac{\th_0^{-\mu m}\th_0^{-2\mu-1}}{\th_0^{-\mu m}\th_0^{-\mu m-1}|\th-\th_0|}\left(|\th-\th_0|+\th_0^{-\al}|\th-\th_0|^\al\right)\|F\|_{*,1}\,\mathrm d\th\nonumber\\
   		&\lesssim \left(\th_0^{-2\mu}+\th_0^{-2\mu-\al}\right)\|F\|_{*,1}\lesssim \left(|z|^2+|z|^{2+\frac\al\mu}\right)\|F\|_{*,1}\lesssim |z|^2\|F\|_{*,1}.\label{Eq.K_2,2-est}
   	\end{align}
   	Combining \eqref{Eq.K_2,1-est} and \eqref{Eq.K_2,2-est} gives that
   	\begin{equation}\label{Eq.K_2-est}
   		\left|K_2-\wt\varphi(\th_0)\left(\wt{\operatorname{Ln}}\ \wt z(\th_0)-\operatorname{Ln}\left(\wt z(\th_0)-z^m\right)\right)\right|\lesssim |z|^2\left(1+|\ln|z||\right)\|F\|_{*,1}.
   	\end{equation}
   	
   	It follows from \eqref{Eq.K_1-est} and \eqref{Eq.K_2-est} that
   	\begin{align*}
   		\left|K_1+K_2-\wt\varphi(\th_0)\wt{\operatorname{Ln}}\ \wt z(\th_0)\right|\lesssim |z|^2\left(1+|\ln|z||\right)\|F\|_{*,1}.
   	\end{align*}
   	Therefore, by \eqref{Eq.tilde-varphi-est},  \eqref{Eq.tilde-Ln-est} and \eqref{Eq.r_m-approx}, we have
   	\begin{align*}
   		\left|K_1+K_2\right|&\lesssim \left|\wt\varphi(\th_0)\right|\left|\wt{\operatorname{Ln}}\ \wt z(\th_0)\right|+|z|^2\left(1+|\ln|z||\right)\|F\|_{*,1}\\
   		&\lesssim \th_0^{-2\mu}\|F\|_{*,1}\left(1+|\ln|z||\right)+|z|^2\left(1+|\ln|z||\right)\|F\|_{*,1}\\
   		&\lesssim |z|^2\left(1+|\ln|z||\right)\|F\|_{*,1}.
   	\end{align*}
   	This proves \eqref{Eq.K_1+K_2-est}. 
\end{proof}

As a direct consequence of Proposition \ref{Prop.v_m-est}, we obtain the following estimate for  $v_m$.

\begin{corollary}\label{Cor.v_m-est}
	Assume that $\mu>1/2$ and $\al\in(0,1)$. Let $m_*=m_*(\mu,\al)>0$ be given by Theorem \ref{Thm.main} and fix $m>m_*$. Then there exists a constant $C>0$ such that
	\begin{equation}\label{Eq.v_m-est-rough}
		\left|v_m(z)\right|\leq C|z|^{1-\frac1\mu},\quad \text{for}\ \operatorname{a.e.}\ 0<|z|<1.
	\end{equation}
\end{corollary}
\begin{proof}
	The inequality \eqref{Eq.v_m-est-rough} follows immediately from \eqref{Eq.v_m-est}, \eqref{Eq.r_m-approx} and the fact $\wt\g_m(\th)\sim \th^{1-2\mu}$ for all $\th>0$.
\end{proof}

\begin{remark}
	Similarly, one can show that the same estimate holds for $|z|>1$, hence,
	\begin{equation}\label{Eq.v_m-est-whole}
		|v_m(z)|\lesssim |z|^{1-\frac1\mu},\quad \text{for}\ \operatorname{a.e.}\ |z|>0,
	\end{equation}
	which implies that $v_m\in L^2_{\text{loc}}(\R^2)$. Alternatively, the $L^2$ local-integrability also follows from $u_m(t,\cdot)\in L^2_{\text{loc}}(\R^2\setminus\{0\})$ for all $t>0$ (see \eqref{Eq.u_m-conver-far}) and $v_m\in L^2(|z|\leq 1)$ (by \eqref{Eq.v_m-est-rough}). Moreover, one has exactly the same estimate for $u_m(0,\cdot)$, the initial velocity defined in \eqref{Eq.u_m-def}, i.e.,
	\begin{equation}\label{Eq.u_m(0)-est}
		|u_m(0,z)|\lesssim |z|^{1-\frac1\mu},\quad \text{for}\ \operatorname{a.e.}\ |z|>0.
	\end{equation}
\end{remark}

\subsection{Weak solutions of 2-D Euler equations}
With the preparations provided in the previous subsection, we now turn to verifying that our constructed vortex sheets are weak solutions to the primitive 2-D Euler equations  \eqref{Eq.Euler-velocity} on the whole space $\R^2$. The content in this subsection closely parallels \cite[Section 7.3]{SWZ}.

\begin{proposition}[Weak solutions of the self-similar Euler equations]\label{Prop.weak-sol-ss}
	Assume that $\mu>1/2$ and $\alpha\in(0,1)$.  Let $m_*=m_*(\mu,\al)>0$ be given by Theorem \ref{Thm.main} and fix $m>m_*$. Let the self-similar profile $v_m$ (corresponds to $u_m$ defined in \eqref{Eq.u_m-def}) be defined by \eqref{Eq.v_m-def}. Then for any vector field $w\in C_c^\infty(\R^2; \R^2)$ with $\mathrm{div }\ w=0$, there holds
	\begin{equation}\label{Eq.v_m-weaksol}
		\int_{\R^2}(3\mu-1)v_m\cdot w-\left(v_m\otimes v_m\right): \nabla w+\mu v_m\cdot (x\cdot \nabla w)\,\mathrm dx=0.
	\end{equation}
\end{proposition}
\begin{proof}
	The proof is similar to that of \cite[Proposition 7.1]{SWZ}. For clarification, we list the key points here. Since $\text{div }w=0$, there exists a scalar function $\eta\in C_c^\infty(\R^2)$ such that $w=\nabla^\perp \eta$. Let $\rho\in C_c^\infty(\R)$ be a smooth bump function satisfying $\rho|_{[-1,1]}\equiv 1$ and $\text{supp }\rho{\ \subset}(-2,2)$. For any $\delta>0$, we define $\rho_\delta\in C_c^\infty(\R^2)$ by $\rho_\delta(x)=\rho\left(|x|/\delta\right)$ for $x\in\R^2$. It follows from \eqref{Eq.u_m-weaksol-far} that $v_m$ is a weak solution to the self-similar Euler equations outside the origin, hence,
	\begin{align*}
		\int_{\R^2}(3\mu-1)v_m\cdot\nabla^\perp\left(\eta(1-\rho_\delta)\right)&-\left(v_m\otimes v_m\right): \nabla \nabla^\perp\left(\eta(1-\rho_\delta)\right)\\
		&+\mu v_m\cdot (x\cdot \nabla \nabla^\perp\left(\eta(1-\rho_\delta)\right))\,\mathrm dx=0,\quad \forall\ \delta>0.
	\end{align*}
	As a consequence, to prove \eqref{Eq.v_m-weaksol}, it suffices to show that
	\begin{equation}
		\begin{aligned}
			\int_{\R^2}(3\mu-1)v_m\cdot\nabla^\perp\left(\eta\rho_\delta\right)-\left(v_m\otimes v_m\right): \nabla \nabla^\perp\left(\eta\rho_\delta\right)+\mu v_m\cdot (x\cdot \nabla \nabla^\perp\left(\eta\rho_\delta\right))\,\mathrm dx\to0
		\end{aligned}
	\end{equation}
	as $\delta\to0+$. By \eqref{Eq.v_m-est-rough} and the same argument as in the proof of \cite[Proposition 7.1]{SWZ}, we only need to prove that
	\begin{align}
		\lim_{\delta\to0+}\int_{\R^2}\left(v_m\otimes v_m\right): \nabla \nabla^\perp\rho_\delta\,\mathrm dx&=0, \label{Eq.convergence1}\\
		\lim_{\delta\to0+}\int_{\R^2}\left(v_m\otimes v_m\right): \nabla \nabla^\perp\left(x_j \rho_\delta\right)\,\mathrm dx&=0 \ \ (j\in\{1,2\}). \label{Eq.convergence2}
	\end{align}

    We decompose the velocity field $v_m$ into a summation of its radial part and rotational part: $v_m=v_m^re_r+v_m^\th e_\th$, where $e_r:=z/|z|$ and $e_\th:=\ii z/|z|$, then $v_m^*=v_m^re_r^*+v_m^\th e_\th^*$. Using $\frac1{\ii z}=-\frac{\ii z^*}{|z|^2}=\frac1{|z|}e_\th^*$, it follows from Proposition \ref{Prop.v_m-est} that
    \begin{equation*}
    	\left|v_m^re_r^*+v_m^\th e_\th^*-\frac{\wt\g_m\left(\wt r_m^{-1}(|z|)\right)}{2\pi|z|}e_\th^*\right|\lesssim |z|\big(1+|\ln|z||\big),\quad\text{for a.e.}\ 0<|z|<1.
    \end{equation*}
    Therefore,
    \begin{align}\label{Eq.v_m^r-th-est}
    	\left|v_m^r(r,\th)\right|+\left|v_m^\th(r,\th)-f_0(r)\right|\lesssim r\big(1+|\ln r|\big)
    \end{align}
    for a.e. $0<r<1, \th\in\mathbb T$, where $f_0(r):=\frac{\wt\g_m\left(\wt r_m^{-1}(r)\right)}{2\pi r}$.
    
    \underline{\bf Proof of \eqref{Eq.convergence1}.} Direct computation gives 
    $$\left(v_m\otimes v_m\right): \nabla \nabla^\perp\rho_\delta=\frac1{\delta^2}v_m^rv_m^\theta\left(\rho''\left(\frac r\delta\right)-\frac\delta r\rho'\left(\frac r\delta\right)\right).$$
    Since $v_m$ is divergence-free, we know that $\int_{\mathbb T}v_m^r(r,\th)\,\mathrm d\th$ is a constant independent of $r$, thus $\int_{\mathbb T}v_m^r(r,\th)\,\mathrm d\th=0$, thanks to \eqref{Eq.v_m-est-rough}. Using \eqref{Eq.v_m^r-th-est}, we obtain
    \begin{align*}
    	&\left|\int_{\R^2}\left(v_m\otimes v_m\right): \nabla \nabla^\perp\rho_\delta\,\mathrm dx\right|\\
    	=&\ \frac1{\delta^2}\left|\int_\delta^{2\delta}r\left(\int_{\mathbb T} v_m^r(r,\theta)\left(v_m^\theta(r,\theta)-f_0(r)+f_0(r)\right)\,\mathrm d\theta\right)\left(\rho''\left(\frac r\delta\right)-\frac\delta r\rho'\left(\frac r\delta\right)\right)\,\mathrm dr\right|\\
    	\lesssim &\ \frac1{\delta^2}\int_\delta^{2\delta}r\int_{\mathbb T}\left|v_m^r(r,\theta)\left(v_m^\theta(r,\theta)-f_0(r)\right)\right|\,\mathrm d\theta\,\mathrm dr\\
    	\lesssim&\ \frac1{\delta^2}\int_\delta^{2\delta}r\cdot r\big(1+|\ln r|\big)\cdot r\big(1+|\ln r|\big)\,\mathrm dr\lesssim \delta\to0,\quad \text{as}\ \delta\to0+.
    \end{align*}
    This proves \eqref{Eq.convergence1}.
    
    \underline{\bf Proof of \eqref{Eq.convergence2}.} We only prove the limit for $j=1$, since the proof of $j=2$ is similar. Direct computation gives
    \begin{align*}
    	\left(v_m\otimes v_m\right): \nabla \nabla^\perp\left(x_1 \rho_\delta\right)&=\frac1\delta\left(|v_m^r|^2-|v_m^\theta|^2\right)\rho'\left(\frac r\delta\right)\sin\theta\\
    	&\qquad\qquad\qquad\qquad+\frac r{\delta^2}v_m^rv_m^\theta\left(\rho''\left(\frac r\delta\right)+\frac\delta r\rho'\left(\frac r\delta\right)\right)\cos\theta.
    \end{align*}
    For simplicity, we denote $\wt{v_m^\th}(r,\th):=v_m^\th(r,\th)-f_0(r)$. Then $$\left|v_m^\th(r,\th)\right|^2=\wt{v_m^\th}(r,\th)\left(\wt{v_m^\th}(r,\th)+2f_0(r)\right)+f_0(r)^2.$$
    Since $\int_{\mathbb T}\sin \theta\,d\theta=0$, the term $f_0(r)^2$ contributes nothing into the integral in \eqref{Eq.convergence2}. Thus,
    \begin{align*}
    	&\left|\int_{\R^2}\left(v_m\otimes v_m\right): \nabla \nabla^\perp\left(x_1 \rho_\delta\right)\,\mathrm dx\right|\\
    	\lesssim &\ \frac1\delta\int_{\delta}^{2\delta}r\cdot r\big(1+|\ln r|\big)\cdot r^{1-\frac1\mu}\,\mathrm dr+\frac1{\delta^2}\int_{\delta}^{2\delta}r^2\cdot r\big(1+|\ln r|\big)\cdot r^{1-\frac1\mu}\,\mathrm dr\lesssim\delta^{2-\frac1\mu}\to0
    \end{align*}
    as $\delta\to0+$. This completes the proof.
\end{proof}

\begin{proposition}
	Assume that $\mu>1/2$ and $\alpha\in(0,1)$.  Let $m_*=m_*(\mu,\al)>0$ be given by Theorem \ref{Thm.main} and fix $m>m_*$. Let $u_m$ be given by \eqref{Eq.u_m-def} and $\omega_m=\operatorname{curl}u_m$ (see also \eqref{Eq.omega-def}). Then
	\begin{equation}\label{Eq.u-omega-converg}
		u_m\in C\left([0,+\infty); L^2_{\mathrm{loc}}(\R^2)\right),\quad \omega_m\in C\left([0,+\infty); H^{-1}_{\mathrm{loc}}(\R^2)\right)\cap C_{\text w}\left([0,+\infty); \cM(\R^2)\right).
	\end{equation}
\end{proposition}
\begin{proof}
	Since $u_m$ is self-similar \eqref{Eq.u_m-self-similar}, by \eqref{Eq.v_m-est-whole} and \eqref{Eq.u_m(0)-est},  we have
	\begin{equation}\label{Eq.u_m-est}
		|u_m(t,x)|\lesssim |x|^{1-\frac1\mu}\in L^2_{\text{loc}}(\R^2),\quad\forall\ t\geq0,
	\end{equation}
	where the implicit constant is independent of $t$. Using \eqref{Eq.u_m-conver-far} and \eqref{Eq.u_m-est}, a standard $\varepsilon$-$\delta$ argument implies \eqref{Eq.u-omega-converg}. 
\end{proof}

\begin{proposition}[Weak solutions of the primitive Euler equations]
	Assume that $\mu>1/2$ and $\alpha\in(0,1)$.  Let $m_*=m_*(\mu,\al)>0$ be given by Theorem \ref{Thm.main} and fix $m>m_*$. Let $u_m$ be given by \eqref{Eq.u_m-def}. Then $u_m=u_m(t,x)$ is a weak solution to the 2-D Euler equations \eqref{Eq.Euler-velocity} on $[0,\infty)\times \R^2$. 
\end{proposition}

\begin{proof}
	This result follows directly from Proposition  \ref{Prop.weak-sol-ss} and \eqref{Eq.u-omega-converg}. We omit the detailed proof here, as it follows the same line of reasoning as the proof of \cite[Proposition 7.4]{SWZ} exactly.
\end{proof}

\section*{Acknowledgments}
The authors express their gratitude to Dale I. Pullin for kindly supplying an electronic copy of \cite{Pullin1989}.  Z. Zhang is partially supported by  NSF of China  under Grant 12288101.

\end{document}